\documentclass[10pt,reqno]{amsart}

\usepackage[sort]{cite}
\usepackage{amsfonts} 
\usepackage{amsmath} 
\usepackage{amssymb} 
\usepackage{amsthm} 
\usepackage{latexsym} 
\usepackage{mathrsfs} 
\usepackage{mathtools} 
\usepackage{slashed}  

\usepackage{verbatim}
\usepackage{cases}
\usepackage[dvips]{epsfig}
\usepackage{epsf} 
\usepackage[bookmarksnumbered,pdfpagelabels=true,plainpages=false,colorlinks=true,
            linkcolor=black,citecolor=black,urlcolor=black]{hyperref}
\usepackage{url}
\usepackage{color}
\usepackage{xcolor}
\usepackage{graphicx}
\usepackage{enumitem} 

\usepackage{pifont} 
\usepackage{upgreek} 
\usepackage{fancyhdr} 
\usepackage{calligra} 
\usepackage{marvosym} 
\usepackage[percent]{overpic} 
\usepackage{pict2e} 
\usepackage[hypcap=false]{caption} 
\usepackage{wasysym} 
\usepackage{accents}
\usepackage[margin=3cm]{geometry}
\usepackage{lipsum}
\theoremstyle{plain}
\newtheorem{theorem}{Theorem}[section]

\newtheorem{lemma}[theorem]{Lemma}
\newtheorem{proposition}[theorem]{Proposition}
\newtheorem{corollary}[theorem]{Corollary}

\theoremstyle{definition}
\newtheorem{remark}[theorem]{Remark}

\numberwithin{equation}{section}
\allowdisplaybreaks

\newcommand{\tr}{\operatorname{tr}}

\newcommand{\cA}{\mathcal{A}}

\newcommand{\cD}{\mathcal{D}}
\newcommand{\cE}{\mathcal{E}}

\newcommand{\cG}{\mathcal{G}}

\newcommand{\cR}{\mathcal{R}}

\newcommand{\bC}{\mathbb{C}}
\newcommand{\bE}{\mathbb{E}}

\newcommand{\bH}{\mathbb{H}}
\newcommand{\bN}{\mathbb{N}}
\newcommand{\bM}{\mathbb{M}}

\newcommand{\bR}{\mathbb{R}}

\newcommand{\fF}{\mathfrak{F}}

\newcommand{\scA}{\mathscr{A}}

\newcommand{\sd}{\mathsf{d}}

\newcommand{\sg}{\mathsf{g}}

\newcommand{\sM}{\mathsf{M}}

\newcommand{\sT}{\mathsf{T}}
\newcommand{\sz}{\mathsf{z}}

\newcommand{\PH}{\mathbb{P}_H}

 \newcommand{\curl}{{\rm curl}}

\begin{document}
\title[Non-isothermal Magnetoviscoelastic fluids without Deformation Diffusion]{Global solutions for non-isothermal magnetoviscoelastic fluids in the absence of deformation diffusion}

\author{Yuanzhen Shao}
\address{The University of Alabama\\ 
	Tuscaloosa, Alabama \\
	USA}
\email{yshao8@ua.edu}

\author{Gieri Simonett}
\address{Department of Mathematics\\
        Vanderbilt University\\
        Nashville, Tennessee\\
        USA}
\email{gieri.simonett@vanderbilt.edu}

\thanks{}





\begin{abstract}
We study a non-isothermal model for incompressible magnetoviscoelastic fluids in $\mathbb{R}^N$, $N=2,3$, coupling the incompressible Navier-Stokes equations with the evolution equations for the deformation tensor, temperature, and magnetization. The material coefficients are allowed to depend on the temperature, and no artificial diffusion is imposed on the deformation tensor, leading to a coupled hyperbolic-parabolic system.
We establish local well-posedness of strong solutions for sufficiently regular initial data and prove global existence for small perturbations of a constant equilibrium. The global analysis relies on a suitable auxiliary variable combining the velocity and the inverse deformation tensor, which reveals a hidden dissipative structure and allows us to derive uniform higher-order energy estimates for the coupled system.
\end{abstract}

\maketitle


\section{Introduction}\label{S:Intro}

Magnetoviscoelastic fluids form a class of complex media exhibiting the simultaneous effects of viscosity, elasticity, and magnetization. They arise naturally in the modeling of magneto-sensitive polymers, ferromagnetic gels, and related materials whose mechanical response is influenced by an underlying magnetic microstructure. From the viewpoint of partial differential equations, the resulting models couple incompressible fluid dynamics, transport of elastic deformation, and nonlinear evolution of a constrained magnetization field. In a non-isothermal setting, one must in addition account for the evolution of the temperature and for the dependence of the constitutive coefficients on the thermal state. This leads to a strongly coupled hyperbolic-parabolic system with geometric constraints and nontrivial dissipative structure.

In this article we study the following Cauchy problem on $\bR^N$, $N=2,3$:
\begin{equation}
\label{magneto sys}
\left\{\begin{aligned}
\partial_t u + u \cdot \nabla u - \nabla \cdot (\nu(\theta)\nabla u) +\nabla \pi
&=-\nabla \cdot(\nabla m \odot \nabla m) + \nabla \cdot (F F^{\sT}),\\
\nabla \cdot u &=0 ,\\
\partial_t F + u \cdot\nabla F  &=(\nabla u)^{\sT} F ,\\
\partial_t\theta+u\cdot\nabla\theta -  \nabla \cdot (\kappa(\theta)\nabla \theta)
&= \nu(\theta)|\nabla u|^2 + \alpha(\theta)\big|\Delta m+|\nabla m|^2m\big|^2,\\
\partial_t m + u \cdot\nabla m
&= - \alpha(\theta)\, m\times (m \times \Delta m) - \beta(\theta)\, m \times \Delta m,\\
|m|  &=1   , \\[0.5mm]
 (u(0), F(0), \theta (0),  m(0))& =(u_0, F_0, \theta_0, m_0).
\end{aligned}\right.
\end{equation}
Here $u:(0,T)\times \bR^N\to \bR^N$ is the velocity field, $\pi:(0,T)\times \bR^N\to \bR$ is the pressure, $F:(0,T)\times \bR^N\to \bM^N:=\bR^{N\times N}$ is the deformation tensor, $\theta:(0,T)\times \bR^N\to \bR$ is the absolute temperature, and $m:(0,T)\times \bR^N\to \mathbb S^2$ is the magnetization field. The viscosity coefficient $\nu(\theta)$, the heat conductivity tensor $\kappa(\theta)$, and the magnetic parameters $\alpha(\theta)$ and $\beta(\theta)$ are assumed to depend on the temperature. The term $-\nabla\cdot(\nabla m\odot \nabla m)$ represents the magnetic stress induced by exchange energy, while $\nabla\cdot(FF^\top)$ is the elastic stress generated by the deformation. The heat equation contains the viscous dissipation $\nu(\theta)|\nabla u|^2$ and the magnetic dissipation $\alpha(\theta)|\Delta m+|\nabla m|^2m|^2$. The equation for $m$ is a convected Landau-Lifshitz-Gilbert system subject to the pointwise constraint $|m|=1$.

The heat flux is governed by the generalized Fourier law
\begin{equation}
\label{eqn:heatflux-intro}
q=-\kappa(\theta)\nabla \theta,
\end{equation}
where $\kappa(\theta)$ is positive definite and may be anisotropic. Thus \eqref{magneto sys} combines an incompressible Navier-Stokes system with variable viscosity, a transport-stretch equation for the deformation tensor, a nonlinear convective heat equation, and a constrained geometric evolution for the magnetization. Following \cite{DSS2302}, one can show that \eqref{magneto sys} satisfies the first and second laws of thermodynamics.

\medskip

The mathematical analysis of magnetoviscoelastic flows has its roots in several related theories. On the one hand, incompressible viscoelasticity has been intensively studied; see, for instance, Liu-Walkington \cite{LiuWal01}, Chen-Zhang \cite{ChenZhang06}, Lei-Liu-Zhou \cite{LeiLiuZhou}, and Lin-Zhang \cite{LinZhang08}. These works revealed the central role of the transport-stretch structure of the deformation tensor and the subtle cancellation mechanisms underlying global small-data theories. 
On the other hand, closely related hydrodynamic systems coupled with director variables have been extensively studied in the theory of nematic liquid crystals. A seminal contribution is due to Lin-Liu \cite{LinLiu95}, who introduced and analyzed a simplified Ericksen-Leslie system coupling the incompressible Navier-Stokes equations to a transported director-field equation. This model exhibits several structures that are also present in magnetoviscoelastic flows, notably the constraint $|d|=1$, the Ericksen stress $-\nabla\cdot(\nabla d\odot\nabla d)$, and the transported harmonic-map heat flow. Further developments concerning global weak solutions and partial regularity for the sphere-constrained system can be found, for example, in Lin-Lin-Wang \cite{LinLinWang10}, Hong-Xin \cite{HongXin12}, and Lin-Wang \cite{LinWang16}. Quasilinear approaches to strong well-posedness and long-time dynamics, including thermodynamically consistent non-isothermal Ericksen-Leslie models, were developed in \cite{HNPS14,HieberPruss17}.
In the context of hyperbolic--parabolic coupled systems, we  also mention the general framework developed by Kawashima \cite{KawashimaDissertation}.

The first mathematical model specifically designed for magnetoviscoelastic fluids in an isothermal environment was introduced in \cite{BFLS18,For16l}. In that setting the system reads
\begin{equation}
\label{magneto sys-isothermal-intro}
\left\{\begin{aligned}
\partial_t u + u \cdot \nabla u - \nu \Delta u  +\nabla \pi
&=-\nabla \cdot(\nabla m \odot \nabla m) + \nabla \cdot (F F^{\sT}),\\
\nabla \cdot u &=0 ,\\
\partial_t F - \mu \Delta F  + u \cdot\nabla F
&=(\nabla u)^{\sT} F ,\\
\partial_t m + u \cdot\nabla m
&= - \alpha\, m\times (m \times \Delta m) - \beta\, m \times \Delta m  ,\\
|m|  &=1   ,
\end{aligned}\right.
\end{equation}
where $\nu,\mu,\alpha>0$ and $\beta\ge 0$ are constants. From the modeling point of view, the deformation tensor should satisfy the transport-stretch equation \eqref{magneto sys}$_3$. However, because of the hyperbolic character of that equation and the lack of dissipation mechanism, the works \cite{BFLS18,For16l} introduced an artificial diffusion term $\mu\Delta F$.

For the regularized isothermal system \eqref{magneto sys-isothermal-intro}, Benesov\'a, Forster, Liu, and Schl\"omerkemper \cite{BFLS18} proved the existence of weak solutions in two space dimensions by means of a Galerkin approximation under a suitable smallness assumption on the initial data. Later, Kalousek, Kortum, and Schl\"omerkemper \cite{KKS21} extended this analysis to more general elastic energy densities and, moreover, established the local-in-time existence of strong solutions without the smallness condition together with a weak-strong uniqueness principle. In the two-dimensional periodic setting, De~Anna, Kortum, and Schl\"omerkemper \cite{KorSch22} constructed global Struwe-like solutions with  partial regularity by carrying out a delicate blow-up analysis near singular times.

In three space dimensions, the pointwise constraint $|m|=1$ creates an additional major analytical difficulty. To overcome this obstacle, several authors considered a Ginzburg-Landau approximation in which the sphere constraint is relaxed through a penalization term; see \cite{For16l,GKMS21,SchZab18}. More precisely, the magnetization equation is replaced by
\begin{equation}
\label{eqn:GLappro-intro}
\partial_t m + u \cdot \nabla m
=\Delta m-\frac{1}{\varepsilon^2}(|m|^2-1)m.
\end{equation}
Adapting the approach of Lin-Liu \cite{LinLiu95}, Forster \cite{For16l} combined a Galerkin scheme with a fixed point argument to prove the existence of weak solutions. Schl\"omerkemper and \v{Z}abensk\'y \cite{SchZab18} established weak-strong uniqueness under a Prodi-Serrin type condition. 
Zhao \cite{Zhao18} proved local well-posedness for magneto-viscoelastic flows in a periodic domain and derived several blow-up criteria. Subsequently, the same author \cite{Zhao20} established a weak-strong uniqueness principle for the system by means of a relative-energy method.
For initial data of higher regularity, Garcke, Knopf, Mitra, and Schl\"omerkemper \cite{GKMS21} proved strong well-posedness and stability by deriving a priori estimates that are uniform at the level of the approximating system.

For the constrained system \eqref{magneto sys-isothermal-intro}, the authors of the present paper established in \cite{DSS23} the local well-posedness, stability, and convergence of strong solutions in bounded three-dimensional domains. 
More recently, Liu \cite{Liu26} obtained local classical solutions and global well-posedness for small $L_2(\bR^N)$-initial data; Li and Liu \cite{LiLiu26} derived temporal decay rates for higher-order spatial derivatives of small solutions on $\bR^3$.

Another line of development concerns the physically more realistic situation in which the deformation tensor evolves according to the transport-stretch equation \eqref{magneto sys}$_3$ without artificial diffusion. The first result in this direction is due to Kalousek and Schl\"omerkemper \cite{KS21}, who introduced a notion of dissipative solutions for magnetoviscoelastic flows in bounded domains in dimensions $N=2,3$. Subsequently, Jiang, Liu, and Luo \cite{JiangLiuLuo23} established local existence and uniqueness of strong solutions, as well as global existence for small initial data on $\bR^N$, $N=2,3$. Their analysis is closely related to the classical small-data theory for incompressible viscoelastic fluids developed in \cite{LinLiuZhang05, LinZhang08}. More recently, Wang, Yang, and Yuan \cite{WangYangYuan25} proved blow-up criteria, global existence results, and weak-strong uniqueness for incompressible magneto-viscoelastic flows. We also mention the recent preprint \cite{HaoHuangJiangZhao}, which investigates an inviscid $\nu=0$ and non-dissipative $\alpha=0$ regime.

In the direction of modeling non-isothermal magnetoviscoelastic fluids, we introduced and analyzed in \cite{DSS2302} a thermodynamically consistent three-dimensional non-isothermal model. For the corresponding system with a regularized deformation equation, we established local existence and uniqueness of strong solutions and showed that solutions starting sufficiently close to a constant equilibrium exist globally and converge to a (possibly different) constant equilibrium. Moreover, we proved that every solution which is eventually bounded in the topology of the natural state space exists globally and converges to the set of equilibria. Our analysis is based on the $L_p$-maximal regularity theory for quasilinear parabolic systems; see \cite{PruSim16}.

The present paper continues this line of research and investigates the Cauchy problem \eqref{magneto sys} on $\bR^N$, $N=2,3$, in the physically natural setting where the deformation tensor satisfies the transport-stretch equation without artificial diffusion. Our main objective is to establish local existence and uniqueness of strong solutions, as well as global existence for small initial data. In contrast to the isothermal case, the temperature equation is fully coupled with both the momentum and magnetization equations through the temperature-dependent coefficients $\nu(\theta)$, $\kappa(\theta)$, $\alpha(\theta)$, and $\beta(\theta)$, as well as through the dissipative source terms on the right-hand side of \eqref{magneto sys}$_4$.
 
To prove local well-posedness, we employ a linear iteration scheme. Particular care is required in the treatment of the temperature and magnetization equations due to the presence of initial data that are not $L_2$-integrable. Our approach can be viewed as a generalization of the framework introduced in \cite{KawashimaDissertation} for hyperbolic-parabolic composite systems. Combined with the techniques developed in \cite{DisconziShao23}, the method in this manuscript can also be extended to treat non-symmetric hyperbolic-parabolic composite systems.

The principal analytical difficulty in establishing global existence arises from the temperature-dependent viscosity. In the isothermal analysis of \cite{JiangLiuLuo23}, following the strategy of Lin-Zhang \cite{LinLiuZhang05, LinZhang08}, one exploits a curl-free structure to introduce the auxiliary quantity $G=F^{-1}-I_N$, which satisfies $G^{\sT}=\nabla\psi$ for some potential $\psi$. This transforms the first three equations in \eqref{magneto sys} into
\begin{equation}
\label{psi-eq}
\left\{
\begin{aligned}
\partial_t u + u \cdot \nabla u - \nu \Delta u + \Delta \psi + \nabla \pi + \nabla \cdot G
&= \nabla \cdot g(G) - \nabla \cdot (\nabla m \odot \nabla m),\\
\nabla \cdot u &= 0,\\
\partial_t \psi + u + u \cdot \nabla \psi &= 0,
\end{aligned}
\right.
\end{equation}
where $g(G)$ collects higher-order terms in the expansion of $F F^{\sT}$; see \eqref{g-G}. 

Introducing the new variable $w=\nu u - \psi$, one obtains a generalized Stokes system of the form
\begin{equation*}
\left\{
\begin{aligned}
-\Delta w + \nabla \pi + \nabla \cdot G
&= -\partial_t u - u \cdot \nabla u + \nabla \cdot g(G) - \nabla \cdot (\nabla m \odot \nabla m),\\
\nabla \cdot w &= -\nabla \cdot \psi.
\end{aligned}
\right.
\end{equation*}
Applying $\nabla \partial^\sigma$ to \eqref{psi-eq}$_3$ and testing by $\nabla \partial^\sigma \psi$ for $|\sigma|\le s$ yields
\begin{equation}
\label{est psi}
\frac{1}{2} \frac{d}{dt} \| \nabla \partial^\sigma \psi \|_{L_2}^2 
+ \frac{1}{\nu} \|\nabla \partial^\sigma \psi\|_{L_2}^2 
= -\frac{1}{\nu} \left( \nabla \partial^\sigma w \mid \nabla \partial^\sigma \psi \right) 
- \left( \nabla \partial^\sigma (u \cdot \nabla \psi) \mid \nabla \partial^\sigma \psi \right).
\end{equation}
Combined with elliptic estimates for the generalized Stokes system, this identity allows one to derive global bounds at the level of $\nabla \psi$, rather than $\psi$ itself. This feature is crucial in the viscoelastic argument, since it produces a favorable term involving $\|\nabla \partial_t u\|_{H^{s-2}}$, which can be absorbed by the dissipation in the $H^{s-2}$-estimate of $\partial_t u$.

In the non-isothermal case, however, the dependence of the viscosity on the temperature introduces additional commutator terms involving $\nabla \theta$ in \eqref{est psi}, which destroy the cancellation structure available in the constant-viscosity setting. As a consequence, one can no longer avoid estimates of $\|\psi\|_{H^{s+1}}$, and the approach of \cite{LinLiuZhang05, LinZhang08}, as implemented in \cite{JiangLiuLuo23}, ceases to be applicable. For the same reason, the method of Lei-Liu-Zhou \cite{LeiLiuZhou} cannot be directly extended to the present setting.

To overcome this difficulty, we adapt a strategy developed for incompressible viscoelastic fluids in \cite{ChenZhang06}, where the lack of dissipation in the deformation equation is compensated by introducing the auxiliary variable $M=\nu(\theta)\nabla u - G^{\sT}$. We show that a similar approach can be successfully applied to \eqref{magneto sys}, despite the additional complications caused by the temperature-dependent coefficients.

The remainder of this paper is organized as follows. 
In Section~\ref{Section:main thm}, we state the main results of the manuscript. 
Section~\ref{Section:equi form} is devoted to deriving an equivalent formulation of \eqref{magneto sys}, which is more suitable for the local wellposedness analysis, and to introducing the functional framework and several auxiliary tools used throughout the paper. 
In Section~\ref{Section:Local wellposed}, we establish the local existence and uniqueness of strong solutions to \eqref{magneto sys}. 
Finally, in Section~\ref{Section:global existence}, we prove the global existence of solutions under a smallness assumption on the initial data.

\medskip
\noindent
{\bf Notation:}
For the readers' convenience, we list here some notation and conventions used throughout the manuscript.

In the following, all vectors $a=(a_j)_{j=1}^N\in \bR^N$ are viewed as column vectors. 
For two vectors $a,b\in\bR^N$, the Euclidean inner product is denoted by $a\cdot b$.
Given two matrices $A,B\in \bM^N$, the Frobenius matrix inner product $A:B$ is given by
$$
A:B={\rm Tr} (AB^{\sT}),
$$
where ${}^{\sT}$ is the transpose.
If $u\in C^1( \bR^N;\bR^N)$, we set $\nabla u(x)= e_j \otimes \partial_j u(x)$ for $x\in  \bR^N$.
Here and in the sequel, we use the summation convention, indicating that terms with repeated indices are added.
Hence, for $u=(u_j)_{j=1}^N\in C^1( \bR^N; \bR^N)$, we have 
$$[\nabla u(x)]_{ij}= \partial_i u_j(x), \;\; 1\le i, j\le N, \;\; x\in  \bR^N.$$
We note that $[\nabla u(x)]^{\sT}$ corresponds to the Fr\'echet derivative of $u$ at $x\in \bR^N$.

If $A\in C^1(\bR^N;\bM^N)$, its divergence $\nabla \cdot A$ is the vector function defined by
\begin{equation}
\label{divergence-matrix}
(\nabla \cdot A)(x)=(\partial_j A(x))^{\sT}e_j, \;\; x\in  \bR^N
\end{equation}
Hence, if $A=[a_{ij}]\in C^1( \bR^N;\bM^N)$, its divergence is given by
 $$[(\nabla\cdot A)(x)]_i=\partial_j a_{ji}(x), \;\; 1\le i,j\le N\;\; x\in  \bR^N. $$
Here and in the sequel, we use the summation convention, indicating that terms with repeated indices are added.
We note   that 
 \eqref{divergence-matrix} implies
\begin{equation}
\label{divergence-property}
(\nabla\cdot A)\cdot u = \nabla\cdot (Au)- A:\nabla u, \quad A\in C^1( \bR^N;\bM^N), \ u\in C^1( \bR^N;\bR^N).
\end{equation}
For a matrix $A\in C^1( \bR^N;\bM^N)$, we set 
$|\nabla A|^2=\partial_j A: \partial_jA.$

\smallskip
\noindent
For $N=3$ and $A=[A_1,A_2,A_3]\in C^1(\bR^3;\bM^3)$,
where $A_j$ denotes the $j$-th column of $A$, we define the curl
of $A$ columnwise by
\[
\curl A
=
[\curl A_1,\curl A_2,\curl A_3],\qquad
\text{or equivalently,}
\quad
[\curl A]_{ij}
=
\varepsilon_{ik\ell}\partial_k A_{\ell j},
\qquad 1\leq i,j\leq 3,
\]
where $\varepsilon_{ik\ell}$ is the Levi--Civita symbol.
Accordingly, $\curl\curl A$ is also understood columnwise.
For $N=2$, we use the corresponding two-dimensional convention.
With these definitions,
\begin{equation}\label{curl-curl-matrix}
\nabla(\nabla\cdot A)
=
\Delta A+\curl\curl A.
\end{equation}
\smallskip
For any $T\in (0,\infty)$, let 
\[
J_T=[0,T].
\]
$K :\bR_+\to \bR_+$  ($C  $, resp.) denotes  a continuous increasing function (a  positive constant, resp.), which has at most polynomial growth.

\smallskip
Given a domain $E\subset \bR^N$ or $E\subset \bR^{N+1}$, $BC^k (E)$ denotes the space of all functions on $E$ whose derivatives up to $k$-th order are continuous and bounded. Let $BC^\infty(E)=\bigcap\limits_{k= 0}^\infty BC^k(E)$. 

\smallskip
We use $H^s$  to denote the $L_2(\bR^N)$-based Sobolev spaces with $H^0=L_2 (\bR^N )$.

Let $\PH: L_2 ( \bR^N; \bR^N) \to L_{2,\sigma}  ( \bR^N; \bR^N)$   be the Helmholtz projection, where 
$$
L_{2,\sigma}  ( \bR^N; \bR^N) :=\PH (L_2 ( \bR^N; \bR^N)).
$$
For any function space $\fF(\sM)$ with $\sM\in \{  \bR^N , J_T\times  \bR^N , (0,T)\times \bR^N\}$, $\fF_\sigma(\sM; \bR^N)= \fF(\sM; \bR^N)\cap L_{2,\sigma}( \bR^N;\bR^N)$.
Note that $H^s_\sigma $ is a closed subspace of $H^s$.

\smallskip
Given $\sigma=(\sigma_i)_{i=1}^N,\tau=(\tau_i)_{i=1}^N\in \bR^N$,
$\tau \leq \sigma $ means that $\tau_i \leq \sigma_i$. The relation $\tau<\sigma$ is defined similarly with the additional condition $|\tau| < |\sigma|$.

\smallskip
Given a matrix $M \in \bM^N$, $M^*= \overline{M}^{\sT}$ denotes the adjoint of $M$.  
$(\cdot | \cdot)$ denotes the inner product in   $ L_2 (\bR^N )$.    

\smallskip
The admissible class $\scA (\bR^n; \bR^m)$ for the parameters $\alpha,\beta,\nu,\kappa$ is defined as the collection of all functions   $ f\in C^\infty (\bR^n; \bR^m)$ such that for each multi-index $\gamma \in \bN^n$, there exists a continuous increasing function $K_\gamma: \bR_+ \to \bR$ such that
$$
| \partial^\gamma_\xi  f (\xi) | \leq K_\gamma (  |\xi |  ) , \quad \xi \in \bR^n.
$$
Here $\bN$ denotes the set of all natural numbers including $0$.


\medskip
\noindent
{\bf Assumptions:}
The following will be assumed  throughout this article.
\begin{itemize}
\item  $s\in \bN$ and $s> \frac{N}{2}  $; 
\item the parameters $\alpha,\beta,\nu  \in \scA(\bR; \bR)$, $\kappa \in \scA(\bR ;{\rm sym}\,(\bM^N))$ and 
there exist  positive constants  $\underline{\alpha}$,  $\underline{\nu}$, $\underline{\kappa}$ and $\overline{\alpha}$, $\overline{\beta}$,  $\overline{\nu}$, $\overline{\kappa}$ such that
$$
\underline{\alpha} \leq \alpha \leq \overline{\alpha}, \quad 0\leq \beta   \leq \overline{\beta}, \quad \underline{\nu} \leq \nu \leq \overline{\nu} , \quad   \underline{\kappa} I_N \leq \kappa \leq \overline{\kappa} I_N .
$$  
\end{itemize}
Note that the first assumption implies $H^s ( \bR^N) \hookrightarrow C( \bR^N)$ and thus $H^s( \bR^N)$ is a Banach algebra.
 
\section{Main Theorem}\label{Section:main thm}

\begin{theorem}
\label{Thm: Local wellposedness}
Let $s\geq 2$ be an integer.
Assume $(u_0,F_0-I_N, \nabla \theta_0, \nabla m_0)\in H^s_\sigma \times H^s \times H^{s-1} \times H^s$ satisfies 
$|m_0|\equiv 1$,  $\theta_0 \in L_\infty(  \bR^N)$ and  $\theta_0\geq c $ for some   $c\in \bR_+$. Then there exists a unique solution 
\begin{align*}
u & \in L_2((0,T); H^{s+1}_\sigma) \cap H^1((0,T); H^{s-1}_\sigma)  \\
F -I_N & \in   C(J_T,H^s) \cap C^1(J_T; H^{s-1}) \\
 \theta  & \in L_\infty((0,T)\times \bR^N)    \\
\nabla \theta & \in L_2((0,T); H^{s}) \cap H^1((0,T); H^{s-2}) \quad \text{and} \quad \partial_t \theta \in L_2((0,T); H^{s-1})\\
\nabla m & \in L_2((0,T); H^{s+1}) \cap H^1((0,T); H^{s-1})  \quad \text{and} \quad \partial_t m \in L_2((0,T); H^{s})
\end{align*}
to \eqref{magneto sys} for some $T>0$ that only depends on  
$$
\|(u_0, F_0 - I_N,  \nabla m_0)\|_{H^s} + \| \nabla \theta_0\|_{H^{s-1}} + \| \theta_0\|_\infty.
$$ 
Moreover, it holds that
$$
\theta(t,x) \geq   \inf_{y\in \bR^N} \theta_0(y), \quad (t,x)\in [0,T]\times   \bR^N.
$$
Each solution can be extended to a maximal interval of existence $[0,T_*)$.
\end{theorem}
\medskip

\begin{theorem}
\label{Thm: global wellposedness}
Let $s\geq 2$ be an integer.
Assume $(u_0,F_0 - I_N, \theta_0 - \theta_*, \nabla m_0)\in H^s_\sigma \times H^s \times H^s \times H^s$ for some $\theta_*\in \bR_+$ satisfies $|m_0|\equiv 1$,   $\theta_0\geq c$ for some  $c\in \bR_+$, $\nabla \cdot F_0=0$,  and $\det F_0=1$. Then  there exists $\varepsilon>0$ such that whenever,
$$
\|u_0\|_{H^s} + \|F_0 -I_N\|_{H^s} + \|\theta_0 - \theta_*\|_{H^s} + \| \nabla m_0  \|_{H^s}\leq\varepsilon,
$$
the solution to \eqref{magneto sys} exists globally, i.e. $T_*=+\infty$.
\end{theorem}

\begin{remark}\label{Remark:main_thm}
\begin{itemize}
\item[]
\item[(i)]  In Theorem~\ref{Thm: global wellposedness}, the initial data satisfy the conditions in Theorem~\ref{Thm: Local wellposedness}. Consequently, the temperature satisfies $\theta (t) -\theta_* \in H^s$ for all $t\in [0,T_*)$.
\item[(ii)] The above theorems still hold true if the coefficients $\alpha,\beta, \kappa,\nu$ depend on all arguments $z=(u,F,\theta,m)$. 
We only consider the case of temperature-dependent coefficients in this article for analytic simplicity.
\item[(iii)]  In virtue of the identity
\[
\partial_t (\det F) + u \cdot \nabla (\det F) = (\nabla \cdot u) \det F,
\]
when $\det F_0=1$, the condition $\det F =1$   is equivalent to the incompressibility condition $\nabla\cdot u=0$.
See also \cite{For13}.
\item[(iv)] 
Consider a flow map from a reference configuration $\Omega_0^X$ to a deformed configuration $\Omega_t^x$, where $X$ is the Lagrangian coordinates and $x$ is the Eulerian coordinates. The deformation gradient is defined by
\begin{equation}
\label{deformation tensor}
\tilde{F}(X,t)= (\nabla_X x(X,t))^{\sT}.
\end{equation}
See \cite[Section~2.1]{For16l}.
Note that our convention of taking gradients differs from that in \cite{For16l}.
Then, under the fluid incompressibility condition $\nabla\cdot u=0$, $F$ satisfies
$$
\partial_t ( \nabla \cdot F) + ( u \cdot \nabla ) (\nabla \cdot F ) =0 .
$$
See \cite[Equation~(1.7)]{LinLiuZhang05}.
Under the condition $\nabla\cdot F_0=0$, we find that $\nabla\cdot F=0$ whenever the solution exists.
\end{itemize}
\end{remark}

\section{Equivalent formulation and Preliminary inequalities}\label{Section:equi form}


We will first express
$$ \alpha m\times (m \times \Delta m) + \beta m \times \Delta m$$
in a form that is more convenient for our analysis. 
By the well-known identity $a\times (b\times c)= (a\cdot c)b -(a\cdot b)c$, we have
\begin{equation*}
m\times (m\times \Delta m)= (m \cdot \Delta m) m - |m|^2 \Delta m .
\end{equation*}
By using the facts that $|m|=1$ and
\begin{equation*}
0= \Delta |m|^2 =2|\nabla m|^2 + 2m \cdot \Delta m,
\end{equation*}
we obtain
\begin{equation*}
m\times (m \times \Delta m)= - (\Delta m + |\nabla m|^2 m ),
\end{equation*}
 provided $m$ is sufficiently smooth.
 Setting
\begin{align*}
\sM(m)=
\begin{bmatrix}
0 &  -m_3 &m_2 \\
m_3 & 0 &  -m_1 \\
-m_2 & m_1 & 0
\end{bmatrix}, \quad m=(m_1,m_2, m_3),
\end{align*}
we can write $m\times \Delta m = \sM(m) \Delta m.$
Set $H=F-I_N$ and $H_0=F_0-I_N$, where $I_N$ is the $N\times N$ identity matrix.
Under the constraint $|m|\equiv 1$, \eqref {magneto sys} is equivalent to  the following system:
\begin{equation}
\label{magneto sys 3}
\left\{\begin{aligned}
\partial_t u + u \cdot \nabla u -    \nabla \cdot (\nu (\theta ) \nabla u )  +\nabla \pi
&=   \nabla \cdot (H H^{\sT}  )  + \nabla\cdot H + \nabla\cdot H^{\sT} -\nabla \cdot (\nabla m \odot \nabla m)    ,\\
\nabla \cdot u &=0 ,\\
\partial_t H + u \cdot\nabla H &=   (\nabla u)^{\sT} H   +   (\nabla u)^{\sT} ,\\
\partial_t\theta -  \nabla \cdot (\kappa(\theta)  \nabla \theta) &= \nu  (\theta)   |\nabla u|^2 + \alpha  (\theta) |\Delta m+|\nabla m|^2m|^2 -  u\cdot\nabla\theta\\
\partial_t m -( \alpha (\theta) I_3- \beta (\theta) \sM(m))\Delta m  &= \alpha (\theta) |\nabla m|^2 m - u \cdot\nabla m  ,\\
|m|  &=1   , \\
(u(0), H(0), \theta (0),  m(0))& =(u_0, H_0, \theta_0,  m_0).
\end{aligned}\right.
\end{equation}

\medskip
We  put $\cA (z): =\alpha (\theta) I_3 - \beta (\theta) \sM(m)  $, where $z=(u,H,\theta, m)$. 
An easy computation shows  
\begin{equation*}
\sigma (\cA (\theta (x), m(x) )= \{ \alpha(\theta (x)),  \alpha (\theta(x)) \pm  i\beta (\theta(x)) |m(x)|\},\quad x\in \bR^N,
\end{equation*}
where $\sigma (\cA) $ denotes the spectrum of $\cA$.

In addition,  we define  
\begin{equation}
\label{def-G}
\begin{split}
G_u(z): &=     \nabla \cdot (H H^{\sT}  )  + \nabla\cdot H + \nabla\cdot H^{\sT}  -\nabla\cdot (\nabla m \odot \nabla m)  -u\cdot \nabla u       ,\\
G_H(z) : &=    (\nabla u)^{\sT} H + (\nabla u)^{\sT}    ,\\
G_{\theta}(z): &=  \nu (\theta) | \nabla u|^2 + \alpha (\theta) | \Delta m + |\nabla m|^2 m |^2   -  u\cdot\nabla\theta, \\
G_m(z) :&= \alpha (\theta) |\nabla m|^2 m -u \cdot\nabla m  ,
\end{split}
\end{equation}
where $z= (u,H, \theta, m)$. 
Given any $r \in \bN \setminus \{0\}$ and   $T>0$, the following function spaces will be used throughout this article:
\begin{align*}
X^r: &= \left\{(u,H,\theta, m): (u,H,\nabla \theta,\nabla m)\in H^r_\sigma \times H^r \times H^{r-1}  \times H^r  \quad\text{and}\quad (\theta, m)\in L^\infty \right\}\\
\widetilde{X}^r :&= H^r_\sigma \times H^r \times H^r  \times H^{r+1} \\
\bE^r_\sigma (J_T) :&=L_2((0,T); H^{r+1}_\sigma) \cap H^1((0,T); H^{r-1}_\sigma)  \\
\bE^r (J_T) :&=L_2((0,T); H^{r+1}) \cap H^1((0,T); H^{r-1})  \\
\bH^r(J_T):&=\left\{f\in L_\infty((0,T)\times\bR^N): \, \nabla f \in \bE^{r-1} (J_T) \text{ and } \partial_t f \in L_2((0,T);H^{r-1})\right\}\\
\bC^r(J_T):&= C(J_T;H^r) \cap C^1(J_T; H^{r-1}) ,
\end{align*}
where $J_T=[0,T]$.
Note that it follows from the Lions-Magenes Theorem, cf. \cite[Lemma 1.2]{MarionTemamBook}, that
\begin{equation}
\label{L-M embedding}
\bE^r (J_T) \hookrightarrow C(J_T; H^r).
\end{equation}
The  following Sobolev embeddings
\begin{equation}
\label{Sobolev embedding}
H^1(\bR^N) \hookrightarrow L_p(\bR^N) \quad\text{for } 2\le p\le 6,    \qquad  H^2(\bR^N) \hookrightarrow BC(\bR^N),  
\end{equation}
as well as  the estimate from the Gagliardo-Nirenberg inequality
\begin{equation}
\label{G-N ineq}
\| f\|_{L_4} \leq C \| f\|_{L_2}^{1-\frac{N}{4}} \| \nabla  f\|_{L_2}^{\frac{N}{4}} 
\end{equation}
will be used throughout.

In the rest of this section, we will state and prove several useful lemmas for our future analysis.
The following Moser type  inequality, cf. \cite[Lemma~3.4]{MajdaAndreabook},  will be frequently used throughout this article.
\begin{lemma}\label{Lem: Moser}
Let $r\in \bN$. Then
\begin{itemize}
\item[{\em (i)}] $\displaystyle  \| fg\|_{H^r} \leq C \left( \|f\|_\infty \| \nabla^r g \|_{L_2} + \| \nabla^r f \|_{L_2} \|g\|_\infty  \right) $.
\item[{\em (ii)}] $\displaystyle \sum_{|\alpha |\leq r} \| \partial^\alpha (fg) - f \partial^\alpha g \|_{L_2} \leq C \left(  \|\nabla f\|_\infty \| \nabla^{r-1} g\|_{L_2} + \|g\|_\infty \|\nabla^r f\|_{L_2} \right) $.
\end{itemize}
\end{lemma}

\noindent
The next result concerns   Nemyskii-type operators.
\begin{lemma}\label{Lem: Nemyskii}
Let $  s, r\in \bN$ with $s\geq 2$ and $\mu\in \scA(  \bR; \bR)$.
\begin{itemize}
\item[{\em (i)}]  If $f\in L_\infty$, $\nabla f \in H^{r-1}$, $\displaystyle  \| \nabla^r \mu (f) \|_{L_2} \leq K(\|f\|_\infty) \| \nabla^r f \|_{L_2} .$
\item[{\em (ii)}] If $f\in \bH^s(J_T)$, then 	$\mu (f)\in  \bH^s(J_T)$.
\end{itemize}
\end{lemma}
\begin{proof}
(i) The asserted inequality clearly holds for $r=1$. When $r>1$, it follows from H\"older's inequality, and the Gagliardo-Nirenberg inequality that for $\partial^\sigma$ with $|\sigma|=r-1$
\begin{align*}
& \| \partial^\sigma ( \mu'(f) \nabla f) \|_{L_2} \\
& \leq \|   \mu'(f) \partial^\sigma  ( \nabla f) \|_{L_2}  + C \hspace{-4mm}\sum_{\substack{\tau_i> 0 \\ \tau
   + \sum_{i=1}^h \tau_i=\sigma }} \| \mu^{(h+1)}(f)    \partial^{\tau_1} f  \partial^{\tau_2} f  \cdots \partial^{\tau_h} f  \partial^\tau (\nabla f) \|_{L_2}  \\
& \leq  \|   \mu'(f)\|_\infty  \| \partial^\sigma  ( \nabla f) \|_{L_2}    + C  \| \mu^{(h+1)}(f) \|_\infty \hspace{-4mm} \sum_{\substack{\tau_i> 0 \\ \tau+ \sum_{i=1}^h \tau_i=\sigma }} \left(
\prod_{i=1}^h \|\partial^{\tau_i}f\|_{L_{\frac{2r}{|\tau_i|}}} \right) \|\partial^\tau\nabla f\|_{L_{\frac{2r}{|\tau|+1}}}.  \\
& \leq K(\|f\|_{\infty}) \| \nabla^r f\|_{L_2} + K(\|f\|_{\infty})  \hspace{-4mm} \sum_{\substack{\tau_i> 0 \\ \tau+ \sum_{i=1}^h \tau_i=\sigma }} \left(
\prod_{i=1}^h \|\nabla^{|\tau_i|}f\|_{L_{\frac{2r}{|\tau_i|}}} \right) \| \nabla^{(|\tau|+1)} f\|_{L_{\frac{2r}{|\tau|+1}}} \\
& \leq K(\|f\|_{\infty}) \| \nabla^r f\|_{L_2} + K(\|f\|_{\infty}) \hspace{-4mm} \sum_{\substack{\tau_i> 0 \\ \tau+ \sum_{i=1}^h \tau_i=\sigma }}   \| \nabla^r f\|_{L_2}^{\sum_{i=1}^h\frac{|\tau_i|}{r} + \frac{|\tau|+1}{r} }\\
& = K(\|f\|_{\infty}) \| \nabla^r f\|_{L_2} .
\end{align*}

(ii) The fact that $\mu(f)\in L_\infty$ is obvious. It follows from  Part (i) that
\begin{align*}
\|  \nabla \mu  (f)  \|_{H^s}      =\sum_{0<|\sigma| \leq s+1} \| \partial^\sigma \mu(f)\|_{L_2}  \leq K(\|f\|_\infty) \|\nabla f\|_{H^s}.
\end{align*}
This implies  $\nabla \mu(f) \in L_2((0,T);H^s)$. When $s=2$, \eqref{Sobolev embedding} implies 
\begin{align*}
\|   \mu '(f)  \partial_t f\|_{H^1} &  \le \| \mu'(f)\|_\infty \left( \| \partial_t f\|_{L_2} + \| \nabla \partial_t f\|_{L_2} \right) + \| \mu''(f)\|_\infty \|\nabla f\|_{L_4} \| \partial_t f\|_{L_4} \\
&\leq K(\|f\|_\infty) (1+ \|\nabla f\|_{H^1} ) \| \partial_t f\|_{H^1}. 
\end{align*}
This implies $\partial_t \mu(f) \in L_2((0,T);H^1)$.
When $s\geq 3$, it suffices to estimate
\begin{align*}
\| \partial^\sigma (  \mu '(f)  \partial_t f ) \|_{L_2} &  \leq  C \sum_{\tau \leq \sigma} \| \partial^{\sigma -\tau}  \mu '(f)  \partial^\tau \partial_t f   \|_{L_2} 
\end{align*}
for $2\leq |\sigma| \leq s-1$. 
When $|\tau|=0$, we have
\begin{align*}
 \| \partial^\sigma  \mu '(f)    \partial_t f   \|_{L_2} & \leq  \| \partial^\sigma  \mu '(f)\|_{L_2} \| \partial_t f   \|_\infty \leq K(\|f\|_\infty ) \| \nabla f \|_{H^{s-2}}  \| \partial_t f   \|_{H^{s-1}} \qquad &&\text{when } |\tau|=0 ;\\
  \|  \mu '(f)  \partial^\sigma   \partial_t f   \|_{L_2} & \leq  \|  \mu '(f) \|_\infty  \|\partial^\sigma   \partial_t f   \|_{L_2} \leq K(\|f\|_\infty )  \| \partial_t f   \|_{H^{s-1}} \qquad &&\text{when } \tau=\sigma.
\end{align*}
For all other $\tau\leq \sigma$, 
\begin{align*}
\| \partial^{\sigma -\tau}  \mu '(f)  \partial^\tau \partial_t f   \|_{L_2}  & \leq \| \partial^{\sigma -\tau}  \mu '(f) \|_{L_4} \| \partial^\tau \partial_t f   \|_{L_4} \\
&\leq C  \| \partial^{\sigma -\tau}  \mu '(f) \|_{H^1} \| \partial^\tau \partial_t f   \|_{H^1} \\
&\leq K(\|f\|_\infty ) \| \nabla f \|_{H^{s-2}}  \| \partial_t f   \|_{H^{s-1}}.
\end{align*}
To sum up, we have proved $\partial_t \mu(f) \in L_2((0,T);H^{s-1})$.
\end{proof}

\noindent
Based on \eqref{Sobolev embedding} and Lemma  \ref{Lem: Nemyskii}, we will derive an estimate for a commutator operator.
\begin{lemma}\label{Lem: Commutator}
Given any multi-index $\sigma\in \bN^N$   such that $|\sigma|=r \in \bN$ and $\mu\in {\scA} (\bR; \bR)$, we define
$$
\cR^{\sigma}_{\mu(f)} g : = \partial^\sigma (\mu(f) g)- \mu(f) \partial^\sigma g,
$$
for $g\in   H^k$, $ f \in L_\infty$ and $  \nabla f \in  H^{k-1}$ with $k=\max\{2,r\}$. Then
\begin{align*}
\| \cR^{\sigma}_{\mu(f)} g \|_{L_2}   \leq K (\| f  \|_\infty )   \| \nabla f\|_{H^{k-1}} \|g\|_{H^r} .
\end{align*}
In particular,
$
\| \cR^{\sigma}_{\mu(f)} g \|_{L_2}   \leq C  \| \nabla f\|_{H^{k-1}} \|g\|_{H^r} .
$
\end{lemma}
\begin{proof}
When $r=1$  and  $\sigma=e_j$, Lemma  \ref{Lem: Nemyskii} implies 
 \begin{align*}
\| \cR^{e_j	}_{\mu(f)} g \|_{L_2} & = \| \partial_j (\mu(f) ) g\|_{L_2} \leq \| \partial_j \mu(f)\|_{L_4} \|g\|_{L_4} \\
&\leq  C \| \partial_j \mu(f)\|_{H^1} \|g\|_{H^1} \leq  K (\| f  \|_\infty )   \| \nabla f\|_{H^1} \|g\|_{H^1} .
\end{align*}
When $r\geq 2$, similar computations show
\begin{align*}
\| \cR^{\sigma}_{\mu(f)} g \|_{L_2} &   \le   \| \partial^\sigma ( \mu(f)) g\|_{L_2} + C \sum_{0<\tau<\sigma} \| \partial^\tau ( \mu(f)) \partial^{\sigma -\tau}g\|_{L_2} \\
&\leq   K (\| f  \|_\infty )   \| \nabla^r f\|_{L_2} \|g\|_{H^r}+ C  \sum_{0<\tau<\sigma}\| \partial^\tau ( \mu(f))\|_{H^1} \|\partial^{\sigma -\tau} g\|_{H^1} \\
& \leq  K (\| f  \|_\infty )   \| \nabla f\|_{H^{r-1}} \|g\|_{H^r} .
\end{align*}
\end{proof}

\noindent
Finally, we will need the following estimate for transported equations.
\begin{lemma}
\label{lem: transport-est}
For   $v,u\in H^s$, it holds
\begin{equation}
\label{transport-est}
\| v \nabla u\|_{H^{s-1}} \leq C \|v\|_{H^s} \| \nabla u\|_{H^{s-1}} .
\end{equation}
\end{lemma}
\begin{proof}
When $s \geq 3$, the fact that $H^{s-1}$ is a Banach algebra implies
$$
\| v \nabla u\|_{H^{s-1}}  \leq C \|v\|_{H^{s-1}} \|\nabla u\|_{H^{s-1}}  .
$$
When $s=2$, we obtain with \eqref{Sobolev embedding} 
\begin{align*}
\| v \nabla u\|_{H^{s-1}}  & \leq    \| v   \nabla  u  \|_{L_2} + \|\nabla (  v   \nabla  u) \|_{L_2} \\
&\leq C \|v\|_{\infty}\|\nabla u\|_{L_2} + C \|\nabla v\|_{L_4} \|\nabla u\|_{L_4} + C \|v\|_{\infty}  \|\nabla^2 u\|_{L_2}    \\
&\leq  C \|v\|_{H^s} \| \nabla u\|_{H^{s-1}} .
\end{align*}
\end{proof}

\section{Local well-posedness}\label{Section:Local wellposed}
We will establish the local existence and uniqueness of a solution to  \eqref{magneto sys 3} by means of a   linear iteration argument.
For properly chosen positive constants $\sd_0, \sd_1, M_1, M_2,M_3 $, the invariant set during the iteration is defined as $R_T( \sd_0, \sd_1, M_1, M_2,M_3   )$, which is the collection of all $z=(u,H,\theta,m)$ such that $|m(0)|\equiv 1$ and  
\begin{itemize}
\item $0<\sd_0 \leq \theta(0)  \leq \sd_1 $ a.e., and
\item $(u,H,  \theta,   m) \in \bE^s_\sigma(J_T)\times \bC^s(J_T)\times \bH^s (J_T) \times \bH^{s+1} (J_T)$, and
\item the following bounds are satisfied
\begin{equation}
\label{invariant estimmate}
\begin{split}
\sup\limits_{t\in [0,T]} \left( \|u(t)\|_{H^s}^2 + \|H (t)\|_{H^s}^2 +\|\nabla \theta (t)\|_{H^{s-1}}^2 +  \| \nabla m(t)\|_{H^s}^2  \right) &\leq M_1^2 \\
\sup\limits_{t\in  [0,T]}   \|  \theta(t) - \theta (0)\|_{L_2} + \sup\limits_{t\in  [0,T]}   \|  m(t) - m (0)\|_{L_2} & \leq M_2 \\
\int_0^T \|u(t)\|_{H^{s+1}}^2 \, dt + \int_0^T \|\nabla \theta (t)\|_{H^s}^2 \, dt +  \int_0^T \| \Delta m(t)\|_{H^{s}}^2 \, dt & \leq M_3^2 . 
\end{split}
\end{equation}
\end{itemize}
We will first prove some a priori estimates for the solutions to the following   linearized system
\begin{equation}
\label{magneto sys abstract linear}
\left\{\begin{aligned}
\partial_t u  -  \nabla \cdot (\nu (\vartheta)   \nabla u  )   +\nabla \pi 
&= G_u(\sz) ,\\
\nabla \cdot u &=0 ,\\
 \partial_t H + v \cdot\nabla H   &=  G_H (\sz)      ,\\
\partial_t \theta -  \nabla \cdot (\kappa (\vartheta) \nabla \theta)    &=  G_\theta  (\sz)   ,\\
\partial_t m -   \cA(\sz) \Delta m    &= G_m  (\sz),\\
(u(0), H(0), \theta (0),m_0 )& = (u_0, H_0, \theta_0 , m_0 ) , 
\end{aligned}\right.
\end{equation}
where $\sz=(v, \cG, \vartheta, h)$ and  $z_0=( u_0, H_0, \theta_0, m_0)$ are given functions.

\begin{proposition}
\label{prop: existence linear}
Let $T>0$    be given.  Assume that 
$$
z_0 =(u_0,H_0, \theta_0 , m_0 )\in X^s  \qquad \text{such that}\quad  |m_0|\equiv 1 , \quad 0<\sd_0 \leq \theta_0 \leq \sd_1
$$ 
for some constants $\sd_0$ and $\sd_1$, and $\sz= (v, \cG,  \vartheta,h )\in R_T(\sd_0, \sd_1,M_1,M_2,M_3)$.
Then  \eqref{magneto sys abstract linear} has a solution $(u , \pi,H,\theta , m )$  such that $(u,H,  \theta,   m) \in \bE^s_\sigma(J_T)\times \bC^s(J_T)\times \bH^s (J_T) \times \bH^{s+1} (J_T) $ and $\nabla \pi\in L_2((0,T); H^{s-1})$.
\end{proposition}
\begin{proof}
See the Appendix.
\end{proof}

\subsection{A priori estimates}\label{sec.local.est}

In this subsection, we will first obtain some fundamental a priori estimates for the linearized system.
For $i=1,2$, assume that
$\sz_i= (v_i, \cG_i,  \vartheta_i,h_i ) $ and  $z_0\in (u_0,H_0, \theta_0 , m_0 )$ satisfy the conditions in Proposition~\ref{prop: existence linear}.
For notational brevity, we put 
$$\sg_i=( g_{u,i},g_{H,i},g_{\theta,i},  g_{m,i})= (G_u(\sz_i), G_H (\sz_i), G_\theta (\sz_i), G_m(\sz_i)).$$
Let $z_i=(u_i,F_i,\theta_i, m_i)$ be a solution  to \eqref{magneto sys abstract linear} asserted by Proposition~\ref{prop: existence linear} corresponding to $(\sz_i,z_0)$. 
In the rest of this subsection, it is understood that an object without the index $i$ corresponds to the case $i=1$, e.g. $z=z_1$, $\sz=\sz_1$, and $\sg=\sg_1$. 

In the reminder of this subsection, we will first derive the fundamental a priori estimates for the solution $z$ to the linearized system~\eqref{magneto sys abstract linear} with $\sz \in R_T(\sd_0, \sd_1,M_1,M_2,M_3)$. 
It is understood that an object without the index $i$ corresponds to the case $i=1$, e.g. $z=z_1$, $\sz=\sz_1$, and $\sg=\sg_1$. 

The second step is to obtain the fundamental a priori estimates for the difference   $\tilde{z}$, which satisfies
\begin{equation}
\label{magneto sys abstract linear difference}
\left\{\begin{aligned}
\partial_t \tilde{u}  -  \nabla \cdot ( (\nu (\vartheta_1)   \nabla  \tilde{u}  )   +\nabla  \tilde{\pi }
&=  \tilde{g}_u  +  \nabla \cdot (\nu (\vartheta_1) - \nu (\vartheta_2) )  \nabla  u_2 ),\\
\nabla \cdot \tilde{u} &=0 ,\\
 \partial_t \tilde{H} + v_1 \cdot\nabla \tilde{H}   &=  \tilde{g}_H -\tilde{v} \cdot\nabla H_2     ,\\
\partial_t \tilde{\theta} -  \nabla \cdot (\kappa (\vartheta_1) \nabla \tilde{\theta})    &=  \tilde{g}_\theta +\nabla \cdot ( (\kappa (\vartheta_1) -\kappa (\vartheta_2))\nabla \theta_2)     ,\\
\partial_t \tilde{m} -   \cA(\sz_1) \Delta \tilde{m}    &= \tilde{g}_m +  (\cA(\sz_1)-\cA(\sz_2))   \Delta m_2 ,\\
(\tilde{u}(0), \tilde{H}(0), \tilde{\theta} (0),\tilde{m}_0 )& = 0 , 
\end{aligned}\right.
\end{equation}
where $\tilde{f}=f_1-f_2$ with $f\in \{z,u,\pi, H,\theta, m, \sz, v,\cG,\vartheta, h, g_u,g_H,g_{\theta},g_m \}$,
under the assumption 
\[
z_1,z_2, \sz_1,\sz_2 \in R_T(\sd_0, \sd_1,M_1,M_2,M_3) .
\] 
Let $r\in \{s-1,s\}$. 
Throughout the following analysis, $C_0 \geq 1$ denotes a generic constant.

\subsubsection{Estimates of $H$-equation}\label{sec.local.H.est}
We will begin with the estimates for $H$. 
Take $\partial^\sigma$ on both sides of \eqref{magneto sys abstract linear}$_3$ for $|\sigma|\leq r$, multiply it by $\partial^\sigma H$ and then integrate over $ \bR^N$. This yields
\begin{align*}
\frac{1}{2} \frac{d}{dt} \| H \|_{H^r}^2  + \sum\limits_{|\sigma|\leq r} ( \partial^\sigma (v^k \partial_k H ) | \partial^\sigma H ) = \sum\limits_{|\sigma|\leq r}(\partial^\sigma g_H | \partial^\sigma H).
\end{align*}
The RHS is majorized by $C \|g_H\|_{H^r} \| H\|_{H^r}$. Using the incompressibility condition $\nabla \cdot v =0$, we immediately find that the second term on the LHS is given by
\begin{align*}
( \partial^\sigma (v^k \partial_k H ) | \partial^\sigma H )   =  ( \partial^\sigma (v^k \partial_k H ) - v^k \partial_k (\partial^\sigma H) | \partial^\sigma H ) = \sum_{\tau<\sigma}  
\begin{pmatrix}
\sigma \\
\tau
\end{pmatrix}
 ( \partial^{\sigma  -\tau} v^k \partial_k (\partial^\tau H) | \partial^\sigma H ) .
\end{align*}
When $|\tau|=|\sigma|-1$,  \eqref{Sobolev embedding} implies 
\begin{align*}
 ( \partial^{\sigma  -\tau} v^k \partial_k (\partial^\tau H) | \partial^\sigma H ) \leq C \|\partial^{\sigma  -\tau} v \|_{L_\infty} \| H\|_{H^r}^2 \leq C \| v \|_{H^{s+1}} \| H\|_{H^r}^2 .
\end{align*}
When $|\tau|<|\sigma|-1$, a similar computation leads to
\begin{align*}
 ( \partial^{\sigma  -\tau} v^k \partial_k (\partial^\tau H) | \partial^\sigma H ) \leq C \|\partial^{\sigma  -\tau}  v \|_{L_4} \|\partial_k (\partial^\tau H) \|_{L_4} \| H\|_{H^r}  \leq C \| v \|_{H^{s+1}} \| H\|_{H^r}^2 .
\end{align*}
These discussions lead to 
\begin{equation}
\label{EE linear H derivative}
\begin{split}
\frac{d}{dt} \| H \|_{H^r}   \leq C \|v\|_{H^{s+1}} \|H\|_{H^r} + C \|g_H\|_{H^r}.
\end{split}
\end{equation}
{Applying Gronwall and H\"older's inequality  for $r=s$ shows that}  
\begin{equation}
\label{EE linear-H}
\begin{split}
\sup_{t\in [0,T]} \|H(t) \|_{H^s}^2   & \leq   e^{ \sqrt{T}  C_0 M_3 } \left( \|H_0\|_{H^s}^2 + C_0 T \int_0^T \|g_H (t)\|_{H^s}^2 \, d t  \right)    .
\end{split}
\end{equation}
For $r=s-1$, applying \eqref{EE linear H derivative} to \eqref{magneto sys abstract linear difference}$_3$ yields 
\begin{equation}
\label{EE linear-H-difference}
\begin{split}
&\sup_{t\in [0,T]}\|\tilde{H}(t) \|_{H^{s-1}}^2     \leq   e^{ \sqrt{T}  C_0 M_3 } \left(   C_0  T  \int_0^T \|\tilde{g}_H (t)\|_{H^{s-1}}^2 \, d t +    T K(M_1) \int_0^T \| \tilde{v}(t)\|_{H^s}^2 \, dt  \right)     ,
\end{split}
\end{equation}
where we have used \eqref{transport-est} to gain the estimate
\begin{align*}
\| \tilde{v} \nabla H_2 \|_{H^{s-1}} \leq K(M_1) \| \tilde{v}(t)\|_{H^s}.
\end{align*}

\subsubsection{Estimates of $\theta$-equation}\label{sec.local.theta.est}

Setting $h_0=h(0)$ and $\vartheta_0=\vartheta (0)$,
we will first derive an $L_\infty$-estimate for the functions $(\vartheta, h)$
appearing in the definition of  $\sg$.
\begin{equation}
\label{est h L infty}
\begin{split}
\|h\|_{L_\infty} & \leq \| h - h_0\|_{L_\infty} + \| h_0\|_{L_\infty}  \leq C \| h - h_0\|_{H^2} + 1   \\
&\leq C \| h   - h_0\|_{L_2} + \| \nabla h\|_{H^s} + \| \nabla h_0\|_{H^s} +1  = C \| h   -h_0\|_{L_2}  +K(M_1).
\end{split}
\end{equation}
Similarly,
\begin{equation}
\label{est vartheta L infty}
\begin{split}
\|\vartheta\|_{L_\infty} & \leq  C \|\vartheta   - \vartheta_0\|_{L_2} +  K(M_1 +\sd_1).
\end{split}
\end{equation}
To estimate \eqref{magneto sys abstract linear}$_4$,  we take $\partial^\sigma$ with $0<|\sigma| \leq r$ on both sides of \eqref{magneto sys abstract linear}$_4$,  multiply it  by $\partial^\sigma \theta$, and  then integrate the resulting equality over $ \bR^N$.
\begin{align*}
\frac{1}{2} \frac{d}{dt} \| \nabla \theta \|_{H^{r-1}}^2 - \sum_{0<|\sigma|\leq r}( \partial^\sigma\ \nabla \cdot  (\kappa(\vartheta) \nabla \theta) | \partial^\sigma \theta ) = \sum_{0<|\sigma|\leq r}( \partial^\sigma g_\theta | \partial^\sigma \theta) .
\end{align*}
Direct computations show
\begin{align*}
 \sum_{0<|\sigma|\leq r}( \partial^\sigma\ \nabla \cdot  (\kappa(\vartheta) \nabla \theta) | \partial^\sigma \theta )   
&= - \sum_{0<|\sigma|\leq r}( \kappa (\vartheta) \partial^\sigma \nabla \theta |  \partial^\sigma \nabla \theta )  -\sum_{0<|\sigma|\leq r} (\cR^\sigma_\kappa \nabla \theta|  \partial^\sigma \nabla \theta ) \\
& \leq - \underline{\kappa} \| \Delta \theta\|_{H^{r-1}}^2  - \sum_{0<|\sigma|\leq r} \sum_{\tau<\sigma} 
\begin{pmatrix}
\sigma\\
\tau
\end{pmatrix}
 (\partial^{\sigma-\tau}\kappa (\vartheta) \nabla \partial^\tau \theta|  \partial^\sigma \nabla \theta ).
\end{align*}
When $\tau=0$ and $r\geq 2$, we can apply Lemma~\ref{Lem: Nemyskii}, \eqref{est vartheta L infty}, and Young's inequality to obtain
\begin{align*}
  (\partial^{\sigma }\kappa (\vartheta) \nabla   \theta|  \partial^\sigma \nabla \theta )   &\leq C   \|\partial^{\sigma } \kappa(\vartheta) \|_{L_2} \| \nabla \theta \|_{L_\infty} \|\nabla \theta\|_{H^{r}} \\
  & \leq K(\|\vartheta\|_{L_\infty} ) \| \nabla \vartheta\|_{H^{s-1}} \| \nabla \theta \|_{H^{r-1}}^{1-\gamma} \| \nabla \theta \|_{H^{r}}^{1+\gamma} \\
 &\leq  \varepsilon \| \nabla \theta \|_{H^{r}}^2 +   K(M_1 + M_2+ \sd_1) \| \nabla \theta \|_{H^{r-1}}^2   ,
\end{align*}
where  we have used the Sobolev  embedding estimate  $\|f\|_{L_\infty} \leq C \|f \|_{H^{1+\gamma}}$ and   interpolation theory to obtain
\begin{equation}
\label{W1 est}
\|f\|_{L_\infty} \leq C \|f \|_{H^{1+\gamma}} \leq C \| f\|_{H^1}^{1-\gamma} \| f\|_{H^{2}}^\gamma  ,\quad \gamma\in \left( \frac{1}{2},1 \right).
\end{equation}
If $r=1$, then
\begin{align*}
(\partial^{\sigma }\kappa (\vartheta) \nabla   \theta|  \partial^\sigma \nabla \theta )  & \leq C \| \partial^\sigma \kappa(\vartheta) \|_{L_4} \| \nabla \theta \|_{L_4} \| \nabla \theta\|_{H^{r}}  \\
 &\leq K(\|\vartheta\|_{L_\infty} ) \|  \nabla \vartheta\|_{H^{s-1}} \| \nabla \theta \|_{H^{r-1}}^{1-\frac{N}{4}}   \| \nabla \theta \|_{H^{r}}^{1+\frac{N}{4}}   \\
 &\leq  \varepsilon \| \nabla  \theta \|_{H^{r}}^2 +   K(M_1 + M_2+\sd_1 ) \| \nabla  \theta \|_{H^{r-1}}^2   .
\end{align*}
When $0<\tau < \sigma$, \eqref{Sobolev embedding} and \eqref{G-N ineq} show
\begin{align*}
(\partial^{\sigma-\tau}\kappa (\vartheta) \nabla \partial^\tau \theta|  \partial^\sigma \nabla \theta )  &  \leq \|\partial^{\sigma-\tau}\kappa (\vartheta) \|_{L_4} \| \nabla \partial^\tau \theta \|_{L_4} \|\nabla \theta \|_{H^{r}} \\ 
&\leq K(\|\vartheta\|_{L_\infty} ) \| \nabla \vartheta\|_{H^{s-1}} \|\nabla \theta \|_{H^{r-1}}^{1- \frac{N}{4}}	\|\nabla \theta \|_{H^{r}}^{1+ \frac{N}{4}}\\
& \leq  \varepsilon \| \nabla  \theta \|_{H^{r}}^2 +   K(M_1 + M_2 +\sd_1 ) \| \nabla  \theta \|_{H^{r-1}}^2   .
\end{align*}
Note that in this case $r\geq 2$.
Young's inequality implies 
\begin{align*}
\sum_{0<|\sigma|\leq r}( \partial^\sigma g_\theta | \partial^\sigma \theta) \leq \varepsilon \| \nabla  \theta\|_{H^{r}}^2 + C(\varepsilon)  \| g_\theta \|_{H^{r-1}}^2.
\end{align*}
Hence choosing $\varepsilon>0$ sufficiently small, we derive  
\begin{equation}
\label{EE linear theta-derivative}
\begin{split}
\frac{d}{dt} \| \nabla \theta \|_{H^{r-1}}^2 + C \|\nabla \theta \|_{H^{r}}^2     \leq  K( M_1 + M_2+\sd_1) \| \nabla \theta \|_{H^{r-1}}^2 + C \| g_{\theta} \|_{H^{r-1}}^2  .
\end{split}
\end{equation}
Multiplying both sides of \eqref{magneto sys abstract linear}$_4$ by $\theta-\theta_0$ and integrating over $\bR^N$ yields
\begin{align*}
\frac{1}{2} \frac{d}{dt}  \| \theta-\theta_0 \|_{L_2}^2  &\leq  ( \nabla \cdot (\kappa(\vartheta) \nabla \theta )| \theta-\theta_0 ) + (g_\theta | \theta-\theta_0 ) \\
&\leq \left( \|\kappa(\vartheta)\|_\infty \|\Delta \theta\|_{L_2} + \| \nabla  \kappa(\vartheta)\|_{L_4} \|\nabla \theta\|_{L_4}   + \| g_{\theta} \|_{L_2}\right)\| \theta-\theta_0 \|_{L_2} \\
&\leq   \| \theta-\theta_0 \|_{L_2}^2 +   \| g_{\theta} \|_{L_2}^2 + K(M_1+M_2+ \sd_1 ) \| \nabla \theta \|_{H^{s-1}}^2 .
\end{align*}
Combined with \eqref{EE linear theta-derivative}, the above estimate implies that when $r=s$,
\begin{equation*}
\begin{split}
& \frac{d}{dt} \left(\| \nabla \theta \|_{H^{s-1}}^2 +\| \theta-\theta_0 \|_{L_2}^2  \right) + C \|\nabla \theta \|_{H^{s}}^2   \\
&  \leq  K( M_1 + M_2+\sd_1) \left(\| \nabla \theta \|_{H^{s-1}}^2 +\| \theta-\theta_0 \|_{L_2}^2   \right) + C \| g_{\theta} \|_{  H^{s-1}}^2  .
\end{split}
\end{equation*}
Applying Gronwall's inequality yields    
\begin{equation}
\label{EE linear-theta}
\begin{split}
& \sup_{t\in [0,T]}\left(\| \nabla \theta (t) \|_{H^{s-1}}^2 +\| \theta (t) -\theta_0 \|_{L_2}^2  \right) + \int_0^T \|\nabla \theta (t)\|_{H^{s}}^2     \, dt \\
& \leq C_0 e^{T  K( M_1 + M_2+\sd_1)}  \left( \| \nabla \theta_0\|_{H^{s-1}}^2 +    \int_0^T \|g_{\theta} (t)\|_{H^{s-1}}^2 \, d t    \right)   .
\end{split}
\end{equation}
For $r=s-1$, we apply \eqref{EE linear theta-derivative} to  \eqref{magneto sys abstract linear difference}$_4$ and obtain
\begin{equation}
\label{EE linear-theta-difference-0}
\begin{split}
& \frac{d}{dt} \| \nabla \tilde{\theta} \|_{H^{s-2}}^2 + C \|\nabla \tilde{\theta} \|_{H^{s-1}}^2   \\
&  \leq  K( M_1 + M_2+\sd_1) \| \nabla \tilde{\theta} \|_{H^{s-2}}^2 + C \| \tilde{g}_{\theta} \|_{H^{s-2}}^2  + K(M_1+M_2+\sd_1)  \left( 1+    \| \nabla \theta_2\|_{H^s}^{2\gamma} \right) \|  \tilde{\sz}  \|_{H^{s-1}}^2  
\end{split}
\end{equation}
for some $\gamma \in \left( \frac{N}{4},1 \right)$, 
where we have employed \eqref{G-N ineq} and \eqref{W1 est} to attain the estimate
\begin{align*}
 \| \nabla \cdot ((\kappa (\vartheta_1) - \kappa (\vartheta_2))\nabla \theta_2 ) \|_{H^{s-2}}  
\leq K(M_1+M_2+\sd_1 ) \left( 1 +  \| \nabla \theta_2 \|_{H^{s}}^{\frac{N}{4}} + \| \nabla \theta_2 \|_{H^{s}}^{\gamma} \right) \|   \tilde{\sz}   \|_{H^{s-1}}.
\end{align*}

Indeed, when $s=3$, by Lemma~\ref{Lem: Nemyskii} and the mean value theorem, we have
\begin{align*}
\| \nabla \cdot ((\kappa (\vartheta_1) - \kappa (\vartheta_2))\nabla \theta_2 ) \|_{H^{s-2}} & \leq   \| (\kappa (\vartheta_1) - \kappa (\vartheta_2))\nabla \theta_2 ) \|_{H^{s-1}} \\
& \le C \| \kappa (\vartheta_1) - \kappa (\vartheta_2)\|_{H^{s-1}} \|\nabla \theta_2  \|_{H^{s-1}}\\
&\le C \left( \int_0^1 \| \kappa ( t\vartheta_1 + (1-t) \vartheta_2 ) \|_{H^{s-1}}   \, dt \right) \| \tilde{\vartheta}\|_{H^{s-1}} \|\nabla \theta_2  \|_{H^s}  \\
& \leq K(M_1+M_2+\sd_1 )\|\nabla \theta_2  \|_{H^{s-1}} \| \tilde{\sz}\|_{H^{s-1}}.
\end{align*}
When $s=2$, it follows from \eqref{G-N ineq} and \eqref{W1 est} that
\begin{align*}
&\| \nabla \cdot ((\kappa (\vartheta_1) - \kappa (\vartheta_2))\nabla \theta_2 ) \|_{H^{s-2}} \\
& \le  C \| \nabla  ((\kappa (\vartheta_1) - \kappa (\vartheta_2)) \|_{L_2} \|\nabla \theta_2 \|_{\infty} + \| \kappa (\vartheta_1) - \kappa (\vartheta_2)\|_{L_4} \| \nabla^2  \theta_2  \|_{L_4} \\
&\le  K(M_1+M_2+\sd_1 )  \| \tilde{\sz}\|_{H^{s-1}} \left( \| \nabla \theta_2 \|_{H^{s-1}}^{1-\frac{N}{4}}\| \nabla \theta_2 \|_{H^{s}}^{\frac{N}{4}} + \| \nabla \theta_2 \|_{H^{s-1}}^{1-\gamma}\| \nabla \theta_2 \|_{H^{s}}^{\gamma} \right)
\end{align*}
Multiplying both sides of  \eqref{magneto sys abstract linear difference}$_4$ by $\tilde{\theta}$ and integrating over $\bR^N$ gives
\begin{align*}
\frac{1}{2} \frac{d}{dt}  \| \tilde{\theta} \|_{L_2}^2 & \leq \|\tilde{g}_\theta \|_{L_2} \| \tilde{\theta}\|_{L_2} + \left| ( ((\kappa (\vartheta_1) - \kappa (\vartheta_2))\nabla \theta_2 | \nabla \tilde{\theta} ) \right| \\
&\leq \|\tilde{g}_\theta \|_{L_2} \| \tilde{\theta}\|_{L_2} + K(M_1+M_2+\sd_1) \|  \tilde{\theta} \|_{H^{s-1}} \| \tilde{\sz} \|_{H^{s-1}} \\
&\leq  \|  \tilde{\theta} \|_{H^{s-1}}^2 + C \|\tilde{g}_\theta \|_{L_2}^2 +   K(M_1+M_2+\sd_1 )   \| \tilde{\sz} \|_{H^{s-1}}^2.
\end{align*}
Together with \eqref{EE linear-theta-difference-0}, this implies  
\begin{equation*}
\begin{split}
& \frac{d}{dt} \|   \tilde{\theta} \|_{H^{s-1}}^2 + C \|\nabla \tilde{\theta} \|_{H^{s-1}}^2   \\
&  \leq  K( M_1 + M_2+\sd_1) \|   \tilde{\theta} \|_{H^{s-1}}^2 + C \| \tilde{g}_{\theta} \|_{H^{s-2}}^2  + K(M_1+M_2+\sd_1 )  \left( 1+    \| \nabla \theta_2\|_{H^s}^{2\gamma} \right) \|  \tilde{\sz} \|_{H^{s-1}}^2  .
\end{split}
\end{equation*}
Gronwall's inequality then gives
\begin{equation}
\label{EE linear-theta-difference}
\begin{split}
& \sup_{t\in [0,T]} \|   \tilde{\theta}(t) \|_{H^{s-1}}^2 + \int_0^T \|\nabla \tilde{\theta} (t) \|_{H^{s-1}}^2  \, dt \leq C_0 e^{T K( M_1 + M_2+\sd_1) } \\
&    \times \left(   \int_0^T \| \tilde{g}_{\theta} (t)\|_{H^{s-2}}^2\, dt + K( M_1 + M_2+\sd_1 ) (T+ T^{1-\gamma} M_3^{2\gamma})\sup_{t\in [0,T]}  
      \|  \tilde{\sz} (t)\|_{H^{s-1}}^2  \right)  .
\end{split}
\end{equation}

\subsubsection{Estimates of $u$-equation}\label{sec.local.u.est}

The same computations leading to \eqref{EE linear theta-derivative} yield
\begin{equation*}
\label{EE linear u-derivative}
\begin{split}
\frac{d}{dt} \| u \|_{H^r}^2 + C \|  u \|_{H^{r+1}}^2     \leq  K( M_1 + M_2+\sd_1) \| u\|_{H^r}^2 + C \| g_u \|_{H^{r-1}}^2  .
\end{split}
\end{equation*}
Then Gronwall's inequality implies
\begin{equation}
\label{EE linear-u}
\begin{split}
 \sup_{t\in [0,T]} \| u (t) \|_{H^s}^2 + \int_0^T \|  u (t)\|_{H^{s+1}}^2     \, dt  
  \leq C_0 e^{T  K( M_1 + M_2+\sd_1)}  \left( \| u_0\|_{H^s}^2 +    \int_0^T \|g_{u} (t)\|_{H^{s-1}}^2 \, d t    \right)   .
\end{split}
\end{equation}
Similarly, the estimate for $\tilde{u}$ can be derived as in \eqref{EE linear-theta-difference}:
\begin{equation}
\label{EE linear-u-difference}
\begin{split}
& \sup_{t\in [0,T]} \|   \tilde{u}(t) \|_{H^{s-1}}^2 + \int_0^T \|  \tilde{u} (t) \|_{H^{s}}^2  \, dt \leq C_0 e^{T K( M_1 + M_2+\sd_1) } \\
&    \times \left(   \int_0^T \| \tilde{g}_u (t)\|_{H^{s-2}}^2\, dt + K( M_1 + M_2+\sd_1 ) (T+ T^{1-\gamma} M_3^{2\gamma})\sup_{t\in [0,T]} 
 \|   \tilde{\sz}(t)\|_{H^{s-1}}^2  \right)  .
\end{split}
\end{equation}

\subsubsection{Estimates of $m$-equation}\label{sec.local.m.est}

Taking $\partial^\sigma  $ with $0<|\sigma | \leq r+1$ on both sides of \eqref{magneto sys abstract linear}$_5$, multiplying by $\partial^\sigma  m$, and then  integrating over $ \bR^N$ yields  
\begin{align*}
   \frac{d}{dt} \| \nabla  m \|_{H^r}^2
  = \sum_{0<|\sigma|\leq r+1}  \left[ (  \partial^\sigma   \partial_t m | \partial^\sigma   m  ) + (   \partial^\sigma m  | \partial^\sigma    \partial_t m)   \right]  =:   I+ II  .
\end{align*}
These two terms can be decomposed in a similar way:
\begin{align*}
I  
& = \sum_{0<|\sigma|\leq r+1 } (   \partial^\sigma    (\cA(\vartheta, h) \partial_{jj} m )| \partial^\sigma      m ) 
+    \sum_{0<|\sigma|\leq r+1 }   (   \partial^\sigma      g_m| \partial^\sigma      m )  \\
&= - \!\!   \sum_{0<|\sigma|\leq r+1 } \!\!  (         \cA(\vartheta, h) \partial^\sigma\partial_{j} m    | \partial^\sigma \partial_{j} m )  + \!\!  \sum_{0<|\sigma|\leq r+1}   \!\!  (  \cR^{\sigma}_{\cA} \partial_{jj} m | \partial^\sigma   m )  -  \!\!  \sum_{0<|\sigma|\leq r+1 } \!\! ( \   (\partial_j  \cA(\vartheta, h)) \partial^\sigma\partial_j    m  | \partial^\sigma   m )  \\
&\quad  +  \!\!  \sum_{0<|\sigma|\leq r+1 }  \!\!  (  \partial^\sigma    g_m| \partial^\sigma    m )  \\ 
& = : I.1 + I.2 + I.3 + I.4  .
\end{align*}
In virtue of Young's inequality  and the relation $\cA(\vartheta, h)+ \cA(\vartheta, h)^{\sT} =2\alpha   I_3 $,    we have
\begin{align*}
  I.1 + II.1 
 & = -\sum_{0<|\sigma|\leq r+1} \!\!  \left[(  \cA(\vartheta, h)  \partial^\sigma \partial_{j} m | \partial^\sigma \partial_{j}  m ) + (   \partial^\sigma \partial_{j} m | \cA(\vartheta, h)  \partial^\sigma \partial_{j} m ) \right]  \\
& =  -\sum_{0<|\sigma|\leq r+1}  \!\!  \left(  (\cA(\vartheta, h)+ \cA(\vartheta, h)^{\sT})  \partial^\sigma \partial_{j}  m \big| \partial^\sigma \partial_{j}  m \right) \\
& \leq  - 2 \underline{\alpha}   \| \Delta m \|_{H^{r}}^2 .
\end{align*}
Note that
$$
I.2 + II.2   = \sum_{\tau< \sigma } 
\begin{pmatrix}
\sigma \\
\tau 
\end{pmatrix}
\left(   ( \partial^{\sigma-\tau} \cA  )\partial^\tau \Delta m | \partial^{\sigma} m  \right).
$$
When $|\tau| = |\sigma|-1$,  \eqref{G-N ineq}, \eqref{est h L infty}, \eqref{est vartheta L infty}, Lemma~\ref{Lem: Nemyskii}, and Young's inequality imply
\begin{equation}
\begin{split}
\label{est-I}
\left(   (\partial^{\sigma-\tau} \cA )\partial^\tau \Delta m \big| \partial^{\sigma} m \right) & \leq  \|  \nabla \cA  \|_{L_4} \| \partial^\tau \Delta m  \|_{L_2} \left\| \partial^{\sigma }   m \right\|_{L_4}  \\
&\leq K ( M_1+M_2 + \sd_1)  \| \nabla m  \|_{H^r}^{1-\frac{N}{4}}  \| \Delta m\|_{H^{r}}^{1+\frac{N}{4}}   \\
&\leq \varepsilon \| \Delta m \|_{H^{r}}^2 + K ( M_1+M_2 + \sd_1) \| \nabla m \|_{H^{r }}^2 .
\end{split}
\end{equation}
When $\tau=0$, if $r=s$ 
\begin{align*}
   \left(  (\partial^{\sigma } \cA )  \Delta m \big| \partial^{\sigma} m \right)    
& \leq C \sum_{\nu \le \sigma}\left\| \partial^{\sigma }\alpha(\vartheta) -  \partial^{\sigma-\nu } \beta (\vartheta) \partial^\nu \sM(h)  \right\|_{L_2}   \| \Delta m \|_{L_4} \| \partial^\sigma   m \|_{L_4}  \\
&\leq  K ( M_1+M_2+\sd_1)  \left(  1+  \|\nabla \vartheta \|_{H^{s}} \right)     \| \nabla m\|_{H^{r}}^{2-\frac{N}{4}}  \| \Delta m\|_{H^{r}}^{   \frac{N}{4}}  \\
&\leq \varepsilon \| \Delta m \|_{H^{r}}^2 + K ( M_1+M_2+\sd_1) \left(  1+  \| \nabla \vartheta \|_{H^{s}}^{\frac{8}{8-N}}   \right) \|\nabla m\|_{H^{r}}^2 ;
\end{align*} 
if $r=s-1$, we have
\begin{align*}
   \left(  (\partial^{\sigma } \cA )  \Delta m \big| \partial^{\sigma} m \right)   
& \leq   C\sum_{\nu \le \sigma}\left\| \partial^{\sigma }\alpha(\vartheta) -  \partial^{\sigma-\nu } \beta (\vartheta) \partial^\nu \sM(h)  \right\|_{L_4} \| \Delta m \|_{L_2} \| \partial^\sigma  	 m \|_{L_4}  \\
&\leq  K ( M_1+M_2+\sd_1)  \left(  1+  \| \nabla \vartheta \|_{H^{s}}  \right)     \| \nabla m\|_{H^{r}}^{2-\frac{N}{4}}  \| \Delta m\|_{H^{r}}^{   \frac{N}{4}}  \\
&\leq \varepsilon \| \Delta m \|_{H^{r}}^2 + K ( M_1+M_2+\sd_1) \left(  1+  \| \nabla \vartheta \|_{H^{s}}^{\frac{8}{8-N}}    \right) \|\nabla m\|_{H^{r}}^2  .
\end{align*} 
For all other $\tau<\sigma$, a similar computation shows 
\begin{align*}
\left( (\partial^{\sigma-\tau} \cA )\partial^\tau \Delta m | \partial^{\sigma} m \right) & \leq \varepsilon \| \Delta m \|_{H^{r}}^2 + K ( M_1+M_2+\sd_1) \|\nabla m\|_{H^{r}}^2 .
\end{align*} 
Next, we apply \eqref{Sobolev embedding}, \eqref{G-N ineq}, and Lemma~\ref{Lem: Nemyskii} and Young's inequality to obtain
\begin{align*}
 I.3 + II.3  
&\leq C \sum_{0<|\sigma|\leq r+1}  \|\nabla \cA(\vartheta, h ) \|_{L_4}  \| \partial^\sigma \nabla m\|_{L_2}  \| \partial^\sigma   m\|_{L_4} \\
& \leq  K ( M_1+M_2+\sd_1)  \| \nabla m\|_{H^{r}}^{1- \frac{N}{4}}  \| \Delta m\|_{H^{r}}^{1+ \frac{N}{4}}  \\
& \leq  \varepsilon \| \Delta m \|_{H^{r}}^2 + K ( M_1+M_2+\sd_1) \|\nabla m\|_{H^{r}}^2 .
\end{align*}
In addition,     Young's inequality yields
\begin{align*}
I.4+II.4  \leq  \varepsilon \|\Delta m \|_{H^{r}}^2 + C(\varepsilon) \|g_m\|_{H^r}^2.
\end{align*}
To sum up, by choosing $\varepsilon>0$ sufficiently small, we end up with
\begin{equation}
\label{EE linear m-derivative}
\begin{split}
    \frac{d}{dt} \| \nabla m\|_{H^{r}}^2   + C \| \Delta m \|_{H^{r}}^2   
 \leq K ( M_1+M_2+\sd_1) \left(  1+ \| \nabla \vartheta \|_{H^{s}}^{\frac{8}{8-N}} \right) \| \nabla m \|_{H^{r}}^2      +   C \|g_m\|_{H^r}^2  .
\end{split}
\end{equation}
Note that $\frac{8}{8-N}<2$ for $N=2,3$. 
Multiplying both sides of \eqref{magneto sys abstract linear}$_5$ by $m-m_0$ and integrating over $\bR^N$ yields
\begin{align*}
\frac{1}{2}  \frac{d}{dt} \| m-m_0\|_{L_2}^2 
& \leq  \left( \| \alpha(\vartheta) \|_{\infty}  +  \| \beta(\vartheta) \|_{\infty}\| h\|_{\infty} \right)\| m-m_0\|_{L_2}  \| \Delta m \|_{L_2}  +\| m-m_0\|_{L_2} \|g_m\|_{L_2}  \\
& \leq   \| m-m_0\|_{L_2}^2 + \|g_m\|_{L_2}^2 + K(M_1 +M_2+\sd_1 ) \| \nabla m\|_{H^r}^2 .
\end{align*}
Combined with \eqref{EE linear m-derivative}, this implies  for $r=s$
\begin{equation}
\label{EE linear m}
\begin{split}
 &\quad   \frac{d}{dt} \left( \| \nabla m\|_{H^{s}}^2 +\| m-m_0\|_{L_2}^2  \right)  + C \| \Delta m \|_{H^{s}}^2     \\
& \leq K ( M_1 +M_2+\sd_1 ) \left(1 +   \| \nabla \vartheta \|_{H^{s}}^{\frac{8}{8-N}}  \right)\left( \| \nabla m \|_{H^s}^2  + \| m-m_0\|_{L_2}^2 \right) + C(\varepsilon) \|g_m\|_{H^s}^2.
\end{split}
\end{equation}
Applying Gronwall's inequality yields
\begin{equation}
\label{EE linear m combine}
\begin{split}
&\quad \sup_{t\in [0,T]}\left(\| \nabla m (t) \|_{H^{s}}^2 +\| m (t) -m_0\|_{L_2}^2 \right) + \int_0^T \| \Delta m (t) \|_{H^{s}}^2  \, dt \\
&\leq C_0 e^{   K(M_1+M_2+\sd_1) \left(T+   T^{\frac{4-N}{8-N}}  M_3^{\frac{8}{8-N}}  \right)}     \left( \|\nabla m_0\|_{H^{s}}^2    + \int_0^T \|g_m (t)\|_{H^s}^2\, dt \right).
\end{split}
\end{equation}
Next, taking $r=s-1$ in \eqref{EE linear m-derivative} and applying it to \eqref{magneto sys abstract linear difference}$_5$ yields
\begin{equation}
\label{EE m-difference-prelim}
\begin{split}
   \frac{d}{dt} \| \nabla \tilde{m}\|_{H^{s-1}}^2   + C \| \Delta \tilde{m} \|_{H^{s-1}}^2     
& \leq K (M_1 +M_2 +\sd_1) \left(1 +   \| \nabla \vartheta \|_{H^{s}}^{\frac{8}{8-N}}   \right)\| \nabla \tilde{m} \|_{H^{s-1}}^2      +   C \|\tilde{g}_m\|_{H^{s-1}}^2  \\
& \quad + K(M_1+M_2+\sd_1) (1 + \|\Delta m_2 \|_{H^s}^{2\gamma} ) \| \tilde{\sz}\|_{H^{s-1}}^2,
\end{split}
\end{equation}
where we have employed \eqref{G-N ineq} and \eqref{W1 est} to obtain the estimate
\begin{align*}
\| (\cA(\sz_1) - \cA(\sz_2)) \Delta m_2 \|_{H^{s-1}} \leq K(M_1+M_2+\sd_1) (1 + \|\Delta m_2 \|_{H^s}^{\gamma} ) \| \tilde{\sz}\|_{H^{s-1}} 
\end{align*}
for some $\gamma \in \left( \frac{N}{4} , 1\right)$.
Multiplying both sides of \eqref{magneto sys abstract linear difference}$_5$  by $\tilde{m}$ and integrating over $\bR^N$ yields
\begin{align*}
 \frac{1}{2} \partial_t \|\tilde{m}\|_{L_2}^2   
& \leq \|\tilde{g}_m\|_{L_2} \| \tilde{m}\|_{L_2}+ \left| (\partial_j \cA(\sz_1)\partial_j \tilde{m} | \tilde{m} ) \right| + \left| ( (\cA (\sz_1)-\cA (\sz_2)) \Delta m_2 | \tilde{m} ) \right|  \\
& \leq \|\tilde{g}_m\|_{L_2} \| \tilde{m}\|_{L_2}+ C  \| \partial_j \cA(\sz_1) \|_{L_4} \| \nabla \tilde{m}\|_{L_2} \| \tilde{m}\|_{L_4}\\
&\quad  +C  \left( \| \alpha(\vartheta_1)   -  \alpha (\vartheta_2) \|_{L_4} +\|  \beta(\vartheta_1) h_1 -   \beta(\vartheta_2) h_2 \|_{L_4}   \right)  \|\Delta m_2 \|_{L_2} \|\tilde{m}\|_{L_4} \\
&\leq K(M_1+M_2+\sd_1) \|\tilde{m}\|_{H^{s-1}}^2 + K(M_1+M_2+\sd_1) \| \tilde{\sz}\|_{H^{s-1}}^2 + \|\tilde{g}_m\|_{L_2}^2.
\end{align*}
Then we arrive at
\begin{equation}
\label{EE m-difference-prelim-0}
\begin{split}
   \frac{d}{dt} \|  \tilde{m}\|_{H^{s}}^2   + C \| \nabla \tilde{m} \|_{H^{s}}^2     
& \leq  K (M_1 +M_2 +\sd_1) \left(1 + \| \nabla \vartheta \|_{H^{s}}^{\frac{8}{8-N}}    \right)\|  \tilde{m} \|_{H^{s}}^2      +   C \|\tilde{g}_m\|_{H^{s-1}}^2  \\
& \quad + K(M_1+M_2+\sd_1) (1 + \|\Delta m_2 \|_{H^s}^{2\gamma} ) \| \tilde{\sz}\|_{H^{s-1}}^2  .
\end{split}
\end{equation}
It follows from   Gronwall's inequality that
\begin{equation}
\label{EE m-difference}
\begin{split}
&\sup_{t\in [0,T]} \|  \tilde{m}(t) \|_{H^{s}}^2 + \int_0^T \|\nabla \tilde{m} (t)\|_{H^{s}}^2 \, d t   \leq  C_0 e^{   K(M_1+M_2 +\sd_1) \left(T   + T^{\frac{4-N}{8-N}} M_3^{\frac{8}{8-N}}   \right) } \\
&\quad \times \left(   \int_0^T \|\tilde{g}_m (t) \|_{H^{s-1}}^2\, dt  +   K(M_1 + M_2 +\sd_1) \left(T   + T^{1-\gamma} M_3^{2\gamma}  \right) \sup_{t\in [0,T]}\| \tilde{\sz}\|_{H^{s-1}}^2  \right) .
\end{split}
\end{equation}

\subsubsection{Estimates for nonlinear terms}

To bound the solutions in the linear iteration argument, we will also need the following   energy estimates for the nonlinear terms on the right hand side of \eqref{magneto sys 3}.
\begin{proposition}\label{prop: nonlinear EE}
Given $z=(u, H, \theta, m)\in R_T(\sd_0,\sd_1,  M_1, M_2,M_3 )$, we have for all $\gamma\in \left( \frac{1}{2}, 1 \right)$
\begin{equation}
\label{nonlinear EE}
\begin{split}
\int_0^T \| G_u(z(t))\|_{H^{s-1}}^2 \, dt &\leq T K(M_1)  \\
\int_0^T \| G_H(z(t))\|_{H^s}^2 \, dt &\leq       K( M_1)  M_3^2 \\
\int_0^T \| G_\theta(z(t))\|_{H^{s-1}}^2 \, dt &\leq K(M_1+M_2+\sd_1) \left(T + T^{1-\gamma}  M_3^{2\gamma} \right) \\
\int_0^T \| G_m(z(t))\|_{H^{s}}^2 \, dt &\leq T K(M_1+M_2+\sd_1).
\end{split}
\end{equation}
\end{proposition}
\begin{proof}
To obtain \eqref{nonlinear EE}$_1$, we observe that
\begin{align*}
& \quad \int_0^T \| G_u(z(t))\|_{H^{s-1}}^2 \, dt \\
&\leq C \int_0^T \| \nabla m (t) \odot \nabla m(t)\|_{H^s}^2 \, dt  + C \int_0^T \| H(t) H^{\sT}(t)\|_{H^s}^2 \, dt  + C \int_0^T \| u (t) \cdot \nabla u(t) \|_{H^{s-1}}^2 \, dt   + C T M_1^2  \\
&  =: I.1 +I.2+I.3 + CTM_1^2.
\end{align*}
Using the fact that $H^s(\bR^N)$ is a Banach algebra, we have
$$
I.1 \leq C \int_0^T \| \nabla m (t) \|_{H^s}^4\, dt \leq T  C M_1^4.
$$
A similar argument shows
$ 
I.2 \leq   T C M_1^4.
$ 
It follows from \eqref{transport-est} that
$$
I.3 \leq CT    M_1^4.
$$
In virtue of \eqref{L-M embedding}, we have
$$
\nabla u \in C ([0,T]; H^{s-1}) .
$$
\eqref{nonlinear EE}$_2$ follows from similar computations.    The estimate for \eqref{nonlinear EE}$_3$  
is a direct consequence of the   fact that $H^s(\bR^N)$ is a Banach algebra.

\eqref{nonlinear EE}$_4$  can be split into
\begin{align*}
\int_0^T \| G_m(z(t))\|_{H^{s}}^2 \, dt &\leq C \int_0^T \| u (t)\cdot \nabla m (t)\|_{H^s}^2 \, dt + \int_0^T \| (\alpha ( \theta )  | \nabla m|^2 m)(t) \|_{H^s}^2 \, dt =:II.1+II.2.
\end{align*}
These terms can be estimated as follows:
\begin{align*}
II.1 \leq C \int_0^T\| u (t)\|_{H^s}^2 \|\nabla m (t)\|_{H^s}^2 \, dt   \leq T  K(M_1),
\end{align*} 
and it follows from Lemmas~\ref{Lem: Moser} an \ref{Lem: Nemyskii} that
\begin{align*}
&\quad \| \alpha ( \theta)  | \nabla m|^2 m \|_{H^s}   =\sum_{|\sigma|\leq s } \| \partial^\sigma ( \alpha ( \theta)  | \nabla m|^2 m  ) \|_{L_2} \\
& \leq C \| \alpha ( \theta)  | \nabla m|^2  \|_{H^s}  \|m\|_{L_\infty} + \|\alpha ( \theta)  | \nabla m|^2\|_\infty \| \nabla m\|_{H^{s-1}} \\
&\quad +   \! \sum_{0<|\sigma|\leq s } \sum_{\tau<\sigma} 
\begin{pmatrix}
\sigma \\
\tau
\end{pmatrix}
\| \partial^\tau ( \alpha ( \theta)  | \nabla m|^2 ) \|_{L_4} \|\partial^{\sigma-\tau}  m   \|_{L_4} \\
&\leq K(M_1+M_2+\sd_1) \left( \| \alpha(\theta)\|_\infty \|  |\nabla m|^2 \|_{H^s} + \| \nabla^s \alpha(\theta)\|_{L_2} \|   |\nabla m|^2 \|_\infty  + \| \nabla m\|_{H^{s-1}} \right) \\
&\quad + C \| \alpha ( \theta)  | \nabla m|^2  \|_{H^s}  \|\nabla m\|_{H^{s-1}} \\ 	
&\leq K(M_1+M_2+\sd_1) ,
\end{align*}
which implies  $ II.2 \leq  TK(M_1+M_2+\sd_1)$.

Direct computations show
\begin{align*}
  \int_0^T \| G_\theta(z(t))\|_{H^{s-1}}^2 \, dt  
&\leq  C \int_0^T \| \alpha(\theta(t)) | \Delta m(t)+ |\nabla m(t)|^2 m(t)|^2\|_{H^{s-1}}^2 \, dt \\
&\quad +C \int_0^T \| \nu(\theta(t)) | \nabla u(t)|^2\|_{H^{s-1}}^2 \, dt + C \int_0^T \| u (t) \cdot \nabla \theta(t) \|_{H^{s-1}}^2 \, dt\\
&  =: III.1 +III.2+III.3.
\end{align*}
By means of a similar computation applied to $I.3$, we get
\begin{align*}
III.3\leq T  C M_1^4. 
\end{align*}
$III.2$ can be estimated in a similar way to $II.2$. 
When $s\geq 3$, Lemma~\ref{Lem: Moser} implies  
\begin{align*}
III.2\leq  C ( \| \nu(\theta)\|_\infty \| \nabla u\|_{H^{s-1}}^2 +\| \nabla u\|_{\infty}^2 \|\nabla^{s-1} \nu (\theta)\|_{L_2} ) \leq   T K(M_1+M_2+\sd_1).
\end{align*}
When $s=2$, 
\begin{align*}
III.2 &= C \int_0^T \| \nu(\theta(t)) | \nabla u(t)|^2\|_{L_2}^2 + C \int_0^T \| \nabla (\nu(\theta(t)) | \nabla u(t)|^2)\|_{L_2}^2 =III.2.1+III.2.2.
\end{align*}
\eqref{Sobolev embedding}   implies  
\begin{align*}
III.2.1\leq T C \sup\limits_{t\in [0,T]} \| \nu (\theta (t))\|_{L_\infty}^2 \| \nabla u(t)\|_{L_4}^4 \leq T K(M_1 + M_2+\sd_1).
\end{align*}
On the other hand, employing \eqref{Sobolev embedding} and \eqref{W1 est} yields
\begin{align*}
III.2.2 &\leq T C \sup\limits_{t\in [0,T]} \| \nu'(\theta (t)) \nabla \theta (t) \|_{L_6}^2 \| \nabla u (t) \|_{L_6}^4 + C   \int_0^T \| \nu (\theta (t))\|_{L_\infty}^2 \| \nabla u (t)\|_{L_\infty}^2 \| \nabla^2 u (t)\|_{L_2}^2 \, dt \\
&\leq T K(M_1+M_2+\sd_1) \left( 1+  \int_0^T \| u(t) \|_{H^s}^{4- 2\gamma}  \| u(t) \|_{H^{s+1}}^{ 2\gamma} \, dt \right) \\
&\leq K(M_1+M_2+\sd_1) (T + T^{1-\gamma}   M_3^{2\gamma}).
\end{align*}
The estimate for $III.1$ follows in a similar manner.

\end{proof} 

\begin{corollary}\label{Cor:difference 2}
Given any  $z_1,z_2\in R_T(\sd_0,\sd_1,  M_1, M_2,M_3 )$, the following estimates   hold:
\begin{equation}
\label{nonlinear EE-difference}
\begin{split}
\int_0^T \| G_u(z_1(t)) - G_u(z_2(t))\|_{H^{s-2}}^2 \, dt &\leq  T K(M_1) \sup_{t\in [0,T]}  \| \tilde{z}(t)\|_{\widetilde{X}^{s-1}}^2    \\
\int_0^T \| G_H(z_1(t)) - G_H(z_2(t))\|_{H^{s-1}}^2 \, dt &\leq    K(M_3)  \sup_{t\in [0,T]}  \|   \tilde{H} (t)\|_{H^{s-1}}^2 + K(M_1) \int_0^T \| \tilde{v}(t)\|_{H^s}^2 \, dt      \\
\int_0^T \| G_\theta(z_1(t)) -  G_\theta(z_2(t))\|_{H^{s-2}}^2 \, dt &\leq K(M_1+M_2+\sd_1)  \left(T + T^{1-\gamma}   M_3^{2\gamma}  \right) \sup_{t\in [0,T]}   \| \tilde{z} (t)\|_{X^{s-1}}^2  \\
\int_0^T \| G_m(z_1(t)) - G_m(z_2(t))\|_{H^{s-1}}^2 \, dt &\leq  T K(M_1 +M_2+\sd_1) \sup_{t\in [0,T]}  \| \tilde{z}(t)\|_{\widetilde{X}^{s-1}}^2  ,
\end{split}
\end{equation}
for some $\gamma \in \left( \frac{1}{2}, 1 \right)$, where $\tilde{f} =f_1-f_2$ with $f\in \{ z, u,H,\theta ,m \}$.
\end{corollary}
\begin{proof}
Direct computations show
\begin{align*}
\| G_H(z_1 ) - G_H(z_2)\|_{H^{s-1}}      \leq  C \left( \| \tilde{v}\|_{H^s} +  \|H_1\|_{H^s} \| \tilde{v}\|_{H^s} + \| v_2\|_{H^{s+1}} \| \tilde{H}\|_{H^{s-1}}  \right),
\end{align*}
which implies \eqref{nonlinear EE-difference}$_2$.

To estimate \eqref{nonlinear EE-difference}$_1$, \eqref{nonlinear EE-difference}$_3$ and \eqref{nonlinear EE-difference}$_4$, 
we will only study the term involving $\alpha(\theta) | \Delta m |^2 $, as the estimates for the remaining terms in these inequalities can be established in a similar fashion.
When $s=2$,  in virtue of \eqref{W1 est}, we have
\begin{align*}
& \quad \int_0^T \| \alpha(\theta_1(t)) | \Delta m_1 (t) |^2 - \alpha(\theta_2 (t)) | \Delta m_2 (t) |^2\|_{H^{s-2}}^2 \, dt\\
&\leq C \int_0^T \left(  \| \alpha(\theta_1(t))  - \alpha (\theta_2(t)) \|_{L_6}^2 \| \Delta m_1 (t)\|_{L_6}^4  + \| \alpha(\theta_2(t))\|_{\infty}^2 \| \Delta \tilde{m}(t)\|_{L_2}^2 \| \Delta(m_1+m_2)(t)\|_{\infty}^2 \right)\, dt \\
& \leq   K(M_1+M_2+\sd_1)\| (  \tilde{\theta }, \nabla\tilde{m} )(t)\|_{H^{s-1}}^2  \left(T +  \int_0^T    \|\Delta (m_1+m_2)(t) \|_{H^{s}}^{ 2\gamma} \, dt  \right)\\
&\leq K(M_1+M_2+\sd_1) (T+T^{1-\gamma} M_3^{2\gamma} ) \| (  \tilde{\theta }, \nabla\tilde{m} )(t)\|_{H^{s-1}}^2  
\end{align*} 
for some $\gamma \in \left( \frac{1}{2}, 1 \right)$.
The case $s=3$ can be estimated in the same way. When $s\geq 4$, the fact that $H^{s-2}$ is a Banach algebra yields a similar estimate.
\end{proof}

\subsection{Existence}
\subsubsection{Approximating Sequence} 
Assume that 
$$
z_0 =(u_0,H_0, \theta_0 , m_0 )\in X^s  \qquad \text{such that}\quad  |m_0|\equiv 1 , \quad 0<\sd_0 \leq \theta_0 \leq \sd_1
$$ 
for some constants $\sd_0$ and $\sd_1$. We define
$$
z_1= (u_1,H_1,\theta_1, m_1) = \left( e^{t \PH \Delta }u_0, H_0, e^{t\Delta} \theta_0 , e^{t \Delta} m_0  \right),
$$
where we set  $H_1(t)=H_0$ for $t\in [0,T]$.
Here $\PH$ is the Helmholtz projection and $H_1$ is understood as the canonical extension of $H_0$ onto $[0,T]$. 

Given any $T>0$, the $L_p$-maximal regularity property of the Stokes operator, cf. \cite[Chapter 7]{PruSim16}, implies $u_1\in \bE^s_\sigma (J_T)$.
Moreover,
\begin{equation}
\label{heat-sg}
\nabla \theta_1 \in \bE^{s-1}(J_T)
\quad \text{and} \quad
\partial_t \theta_1 \in L_2((0,T);H^{s-1}).
\end{equation}
See \eqref{semigroup}.
The fact that $\sd_0 \le \theta_1 \le \sd_1$ follows from the expression
\[
\theta_1 (t) = \frac{1}{(4\pi t)^{N/2}} \int_{\bR^N} e^{- \frac{|x-y|^2}{4t}} \theta_0(y)\, dy .
\]
Similar arguments apply to $m_1$.
To sum up, $z_1\in R_T(\sd_0,\sd_1, \widetilde{M}_1,\widetilde{M}_2,\widetilde{M}_3)$ for some   $(\widetilde{M}_1,\widetilde{M}_2,\widetilde{M}_3)$, which is non-decreasing in $T$ and depends only on $\|(u_0, H_0, \nabla m_0)\|_{H^s}+ \|\nabla \theta_0\|_{H^{s-1}} $.  

Given $T>0$, define iteratively the solution to 
\begin{equation}
\label{magneto sys abstract-iterative}
\left\{\begin{aligned}
\partial_t u_n  -  \nabla \cdot (\nu (\theta_{n-1})   \nabla u_n  )   +\nabla \pi_n 
&= G_u (z_{n-1}) ,\\
\nabla \cdot u_n &=0 ,\\
 \partial_t H_n + u_{n-1} \cdot\nabla H_n   &=  G_H (z_{n-1})      ,\\
\partial_t \theta_n -  \nabla \cdot (\kappa (\theta_{n-1}) \nabla \theta_n)    &=  G_\theta (z_{n-1})    ,\\
\partial_t m_n -   \cA(z_{n-1} ) \Delta m_n    &= G_m(z_{n-1}) ,\\
z_n(0)& =z_0 , 
\end{aligned}\right. 
\end{equation}
by $z_n=(u_n, H_n, \theta_n, m_n)$, for $n\geq 2$.
The existence of a   solution $z_n$ to \eqref{magneto sys abstract-iterative} on   $[0,T]$ is asserted by Proposition~\ref{prop: existence linear}.
Let
\begin{equation}
\label{M-values}
M_1=M_2= M_3=  \max \left\{ 2 \sqrt{C_0} ( \|(u_0, H_0, \nabla m_0)\|_{H^s} + \|\nabla \theta_0\|_{H^{s-1}})  +1) , \widetilde{M}_1,\widetilde{M}_2,\widetilde{M}_3 \right\},
\end{equation}
where $C_0$ is the constant appearing in Section~\ref{sec.local.est}.
We will show that    $ z_n\in R_T( \sd_0,\sd_1,  M_1, M_2,M_3 )$ for all $n$ as long as $T>0$ is sufficiently small.

Assume that $ z_{n-1}\in R_T( \sd_0,\sd_1,  M_1, M_2,M_3 )$ for some $T>0$. We will update $(T,M_2) $ in the following inductive analysis if necessary.
In virtue of \eqref{EE linear-H}, \eqref{EE linear-theta}, \eqref{EE linear-u}, \eqref{EE linear m combine} and Proposition~\ref{prop: nonlinear EE}, we have the following estimates for $z_n$:
\begin{equation}
\label{EE-iterative}
\begin{split}
& \sup_{t\in [0,T]} \| u_n (t) \|_{H^s}^2 + \int_0^T \|  u_n (t)\|_{H^{s+1}}^2     \, dt  
  \leq C_0 e^{T  K( M_1 + M_2+\sd_1)}  \left( \| u_0\|_{H^s}^2 +   TK(M_1)  \right) \\
& \sup_{t\in [0,T]} \|H_n(t) \|_{H^s}^2    \leq   e^{ \sqrt{T}  C_0 M_3 } \left( \|H_0\|_{H^s}^2 +   T K(M_1)M_3^2  \right)     \\ 
& \sup_{t\in [0,T]}\left(\| \nabla \theta_n (t) \|_{H^{s-1}}^2 +\| \theta_n (t) -\theta_0 \|_{L_2}^2  \right) + \int_0^T \|\nabla \theta_n (t)\|_{H^{s}}^2     \, dt \\
& \leq C_0 e^{T  K( M_1 + M_2+\sd_1)}  \left( \| \nabla \theta_0\|_{H^{s-1}}^2 +      K( M_1 + M_2+\sd_1)(T +T^{1-\gamma} M_3^{2\gamma} )  \right) \\
&  \sup_{t\in [0,T]}\left(\| \nabla m_n(t) \|_{H^{s}}^2 +\| m_n (t) -m_0\|_{L_2}^2 \right) + \int_0^T \| \Delta m_n (t) \|_{H^{s}}^2  \, dt \\
&\leq C_0 e^{   K(M_1+M_2+\sd_1) \left(T+   T^{\frac{4-N}{8-N}}  M_3^{\frac{8}{8-N}}  \right)}     \left( \|\nabla m_0\|_{H^{s}}^2    +  TK(M_1 +M_2+\sd_1)  \right) .
\end{split}
\end{equation}
Choose $T>0$ sufficiently small so that  the right hand side of \eqref{EE-iterative} is dominated by $M_1^2$, see \eqref{M-values}.
Therefore, 
\begin{equation}\label{z_n_bd}
 z_n\in R_T( \sd_0,\sd_1, M_1, M_2,M_3 ) \quad \text{for all} \quad  n \in \bN .
\end{equation}
 \eqref{est h L infty}  and \eqref{est vartheta L infty}  further imply that there exists some uniform  bound $M_4$ such that
\begin{equation*}
\| (\theta_n, m_n) \|_\infty \leq M_4.
\end{equation*}
Furthermore, using \eqref{magneto sys abstract-iterative}, one can find  some uniform  bound $M_5$ such that
\begin{equation}\label{z_n_temporal_der_bd}
\| \partial_t z_n \|_{L_2(H^{s-1})\times C(H^{s-1})\times L_2(H^{s-1})\times L_2(H^s)} \leq M_5. 
\end{equation}

\subsubsection{Convergence} 
For $n\ge 2$, let $\tilde{f}_n =f_n-f_{n-1}$ with $f\in \{ z, u, H, \theta, m \}$. 
Based on \eqref{magneto sys abstract-iterative}, we can apply \eqref{EE linear-H-difference}, \eqref{EE linear-theta-difference}, \eqref{EE linear-u-difference}, \eqref{EE m-difference}, and Corollary~\ref{Cor:difference 2} to obtain
\begin{equation}
\label{EE-Cauchy}
\begin{split}
& \sup_{t\in [0,T]} \|   \tilde{u}_n(t) \|_{H^{s-1}}^2 + \int_0^T \|  \tilde{u}_n (t) \|_{H^{s}}^2  \, dt  \\
&   \leq K( M_1 + M_2+\sd_1 ) e^{T K( M_1 + M_2+\sd_1) }    (T+ T^{1-\gamma} M_3^{2\gamma})\sup_{t\in [0,T]}  \| \tilde{z}_{n-1}(t)\|_{X^{s-1}}^2   \\ 
&\sup_{t\in [0,T]}\|\tilde{H}_n(t) \|_{H^{s-1}}^2     \leq   e^{ \sqrt{T}  C_0 M_3 } \left(     T  K(M_3)  \sup_{t\in [0,T]}  \|   \tilde{H}_{n-1} (t)\|_{H^{s-1}}^2 +    T K(M_1) \int_0^T \| \tilde{u}_{n-1}(t)\|_{H^s}^2 \, dt  \right)      \\ 
& \sup_{t\in [0,T]} \|   \tilde{\theta}_n(t) \|_{H^{s-1}}^2 + \int_0^T \|\nabla \tilde{\theta}_n (t) \|_{H^{s-1}}^2  \, dt\\
&    \leq K( M_1 + M_2+\sd_1 ) e^{T K( M_1 + M_2+\sd_1) }    (T+ T^{1-\gamma} M_3^{2\gamma})\sup_{t\in [0,T]}  \| \tilde{z}_{n-1}(t)\|_{X^{s-1}}^2    \\
&\sup_{t\in [0,T]} \|  \tilde{m}_n(t) \|_{H^{s}}^2 + \int_0^T \|\nabla \tilde{m}_n (t)\|_{H^{s}}^2 \, d t   \\
& \leq   K(M_1 + M_2 +\sd_1) e^{   K(M_1+M_2 +\sd_1) \left(T   + T^{\frac{4-N}{8-N}} M_3^{\frac{8}{8-N}}   \right) } \left(T   + T^{1-\gamma} M_3^{2\gamma}  \right) \sup_{t\in [0,T]}\| \tilde{z}_{n-1}\|_{X^{s-1}}^2    .
\end{split}
\end{equation}
For $n \ge 2$, define 
\begin{align*}
a_n= \sup_{t\in [0,T]} \|\tilde{z}_n(t)\|_{\tilde{X}^{s-1}}^2 + \int_0^T \left( \|\tilde{u}_n(t)\|_{H^{s }}^2 +\|\nabla \tilde{\theta}_n(t)\|_{H^{s-1 }}^2 +\|\Delta \tilde{m}_n(t)\|_{H^{s-1 }}^2 \right) \, d t .
\end{align*}
Choosing $T>0$ sufficiently small in \eqref{EE-Cauchy}, we can ensure that
\begin{align*}
a_n  \leq   \frac{1}{2}a_{n-1} .
\end{align*}
This implies that 
\begin{equation}\label{Cauchy_sq}
\begin{split}
\{(u_n,H_n, \theta_n-\theta_1, m_n-m_1 )\}_n  \quad &\text{is Cauchy in }\quad L_\infty((0,T); \widetilde{X}^{s-1} )  \\
\{(u_n, \nabla \theta_n, \Delta m_n)\}_n  \quad &\text{is Cauchy in }\quad L_2((0,T); H^s\times H^{s-1} \times H^{s-1}).
\end{split}
\end{equation}
Furthermore, it follows from \eqref{magneto sys abstract linear difference}, Corollary~\ref{Cor:difference 2}, and the discussion leading to \eqref{EE linear-H-difference}, \eqref{EE linear-theta-difference}, \eqref{EE linear-u-difference}, and \eqref{EE m-difference}    that
\begin{align*}
\{\partial_t H_n \}_n  \quad &\text{is Cauchy in }\quad L_\infty((0,T); H^{s-2} ) \\
\{\partial_t(u_n,   \theta_n,   m_n)\}_n  \quad &\text{is Cauchy in }\quad L_2((0,T); H^{s-2}\times H^{s-2} \times H^{s-1}).
\end{align*}
Consequently, there exists some $ z=(u,H,\theta, m )$ such that
\begin{equation}
\label{converg 1}
\begin{split}
(u_n,H_n,   \theta_n -   \theta_1 ,   m_n -  m_1 )\to (u,H,  \theta- \theta_1,  m -m_1 )  \quad & \text{in} \quad   \bE^{s-1}_\sigma (J_T) \times \bC^r(J_T ) \times \bE^{s-1} (J_T)\times \bE^s (J_T) \\
(u_n,    \theta_n -   \theta_1 ,   m_n -  m_1 )\to (u,  \theta- \theta_1,  m -m_1 )  \quad  &\text{in}\quad   C(J_T; H^r_\sigma) \times C(J_T; H^r) \times C(J_T; H^{r+1}) \\
(\theta_n , m_n)  \overset{* }{\rightharpoonup} (\theta , m)   \quad  &\text{in}\quad L_\infty ((0,T)\times \bR^N) \\
(u_n,  \theta_n -   \theta_1 ,   m_n -  m_1 )\rightharpoonup (u,  \theta- \theta_1,  m -m_1 )  \quad  &\text{in} \quad   \bE^s_\sigma (J_T) \times \bE^s (J_T) \times  \bE^{s+1} (J_T)  
\end{split}
\end{equation}
for  each $r\in \bR_+$ with  $r< s$.
The  convergence of $H_n\to H$ in $\bC^r(J_T ) $ stated in \eqref{converg 1}$_1$ is a consequence of the boundedness of $\{H_n\}_n$ in $\bC^s(J_T ) $ and the convergence of $H_n\to H$ in $\bC^{s-1}(J_T ) $.

The assertions in \eqref{converg 1}$_2$ follow from $\bE^{s-1} (J_T) \hookrightarrow C(J_T; H^{s-1})$ by analogous arguments.

In addition, \eqref{converg 1}$_3$ follows from the boundedness of 
$(   \theta_n -   \theta_1 ,   m_n -  m_1 )$ in $L_\infty((0,T);   H^s \times H^{s+1})$,
which is a consequence of \eqref{Cauchy_sq}$_1$,   \eqref{Sobolev embedding}, and the fact that $(\theta_1,m_1)\in L_\infty ((0,T)\times \bR^N)$.
 
Finally, \eqref{z_n_bd}, \eqref{z_n_temporal_der_bd}, and \eqref{converg 1}$_1$ imply that 
\[
\{(u_n,  \theta_n -   \theta_1 ,   m_n -  m_1 )\}_n \quad \text{ is bounded in }\quad
\bE^s_\sigma (J_T) \times \bE^s (J_T) \times  \bE^{s+1} (J_T)  .
\]
Then \eqref{converg 1}$_4$ follows.

Letting $n\to \infty$, one can show that  $z$ solves \eqref{magneto sys 3}. 
Here we have used the fact that $(\theta_n -   \theta_1 ,   m_n -  m_1 ) \in \bE^s (J_T) \times  \bE^{s+1} (J_T) $.

The fact that $|m(t)|\equiv 1$ and $\theta(t) \geq \sd_0$ can be proved in a similar manner to \cite{DSS23, DSS2302}.
Since $H\in L_\infty((0,T); H^s)$, the regularity property $H\in \bC^s(J_T)$ follows from the standard theory of transport equations, cf. \cite[Theorem~3.19]{BahouriCheminDanchinBook}.

The above analysis shows that $T_0>0$ only depends on   $\|(u_0, H_0,  \nabla m_0)\|_{H^s} + \| \nabla \theta_0\|_{H^{s-1}} + \| \theta_0\|_\infty$.
Moreover,  $\tilde{z}_0=(\tilde{u},\tilde{H}_0, \tilde{\theta}_0,\tilde{m}_0):=z(T_0)$ satisfies 
\[
(\tilde{u},\tilde{H}_0, \nabla\tilde{\theta}_0, \nabla \tilde{m}_0)\in H^s_\sigma \times H^s \times H^{s-1} \times H^s
\]
with
$|\tilde{m}_0|\equiv 1$,  $\tilde{\theta}_0 \in L_\infty(  \bR^N)$  and  $\tilde{\theta}_0\geq \sd_0 $.
Using $\tilde{z}_0 $ as the new initial data, one can repeat the above argument and  obtain a solution $\tilde{z}$  such that
\begin{align*}
\tilde{z}   \in \bE^s_\sigma (J_{T_1} )\times \bC^s (J_{T_1} ) \times \bH^s(J_{T_1}) \times \bH^{s+1}(J_{T_1})  
\end{align*}
for some $T_1>0$. This allows us to extend the solution $z$ onto $[0,T_0+T_1] $ by 
\begin{align*}
z(t)=
\begin{cases}
z(t) ,\quad & t\in [0,T_0]\\
\tilde{z}(t-T_0) , & t\in [T_0,T_0+T_1].
\end{cases}
\end{align*}
Inductively, this yields a maximal interval of existence $[0,T_*)$, which is half-open for otherwise one can repeat the above argument and extend $z$ beyond $T_*$.


\subsection{Uniqueness}

Assume that 
$$z_1=(u_1, H_1, \theta_1, m_1), \, z_2=(u_2, H_2, \theta_2, m_2) \in \bE^s_\sigma(J_T)\times \bC^s(J_T)\times \bH^s(J_T)\times  \bH^{s+1}(J_T)$$ 
are two solutions to \eqref{magneto sys 3} with the same initial data $z_0= (u_0,H_0,\theta_0, m_0)\in X^s$ such that $0<\sd_0 \leq \theta_0$ and $|m_0|\equiv 1$. 
Then  for properly chosen parameters $(  \sd_1, M_1,M_2,M_3)$, $z_1,z_2\in R_T(\sd_0,\sd_1,M_1,M_2,M_3)$.
Let $\pi_1$ and $\pi_2$ be the corresponding fluid pressures.
Then $\tilde{z}=(\tilde{u}, \tilde{H}, \tilde{\theta},  \tilde{m}) = z_1-z_2$ satisfies
\begin{equation}
\label{magneto sys abstract difference-uniqueness}
\left\{\begin{aligned}
\partial_t \tilde{u}  -   \nabla \cdot ( \nu (\theta_1)    \nabla \tilde{u}   ) + \nabla (\pi_1-\pi_2)
&= G_u(z_1) -  G_u(z_2) + \nabla \cdot ( (  \nu(\theta_1) - \nu (\theta_2) )\nabla u_2  )   ,\\
\partial_t  \tilde{H} + u_1 \cdot\nabla  \tilde{H} &= G_H(z_1) -  G_H(z_2) - \tilde{u} \cdot \nabla H_2        ,\\
\partial_t \tilde{\theta}  - \nabla \cdot ( \kappa (\theta_1)    \nabla \tilde{\theta}   )
&= G_\theta(z_1) -  G_\theta(z_2) + \nabla \cdot ( (  \kappa(\theta_1) - \kappa (\theta_2) )\nabla \theta_2  ),\\
\partial_t \tilde{m} -\cA(z_1)   \Delta  \tilde{m}  &= G_m(z_1) -  G_m(z_2)  + (\cA(z_1)- \cA(z_2))  \Delta  m_2 ,\\
(\tilde{u}(0), \tilde{H}(0), \tilde{\theta}(0), \tilde{m}(0))& = 0.
\end{aligned}\right.
\end{equation}
Setting $\tilde{z}_{n-1}=\tilde{z}_n=\tilde{z}$ in \eqref{EE-Cauchy} shows that 
there exists some $T_*>0$ such that for all $T\in (0,T_*]$
\begin{align*}
&\sup_{t\in [0,T]} \|\tilde{z} (t) \|_{\tilde{X}^{s-1}}^2  + \int_0^T \left( \|\tilde{u}(t)\|_{H^{s }}^2 +\|\nabla \tilde{\theta}(t)\|_{H^{s -1}}^2 +\|\Delta \tilde{m}(t)\|_{H^{s-1 }}^2 \right) \, d t \\
&\leq \frac{1}{2} \left[ \sup_{t\in [0,T]} \|\tilde{z} (t) \|_{\tilde{X}^{s-1}}^2 + \int_0^T \left( \|\tilde{u}(t)\|_{H^{s }}^2 +\|\nabla \tilde{\theta}(t)\|_{H^{s-1 }}^2 +\|\Delta \tilde{m}(t)\|_{H^{s -1}}^2 \right) \, d t  \right].
\end{align*}
Hence $z_1(t)=z_2(t)$ for all $t\in (0,T]$. Repeating this argument for finitely many times, one  obtains the uniqueness of solution in the class $\bE^s_\sigma(J_T)\times \bC^s(J_T)\times \bH^s(J_T)\times  \bH^{s+1}(J_T)$.



\section{Global Existence}\label{Section:global existence}

In this section, we will prove Theorem~\ref{Thm: global wellposedness}. 
The equation governing $F$ is hyperbolic and contains no dissipative term. As a consequence, the basic energy estimate does not provide direct coercive control of the higher-order $F$-norm. 
To overcome this difficulty, we will adopt the strategy in \cite{ChenZhang06} in our proof, 
where the lack of dissipation in the deformation equation is compensated by introducing the auxiliary variable 
$M=\nu(\theta) \nabla u - G^{\sT}$.  

In the sequel, we assume 
$$
\det F_0=1, \quad \nabla\cdot F_0=0, \quad   \theta_0\geq c > 0, \quad \text{and} \quad  |m_0|\equiv 1. $$
 Note that the incompressibility condition $\nabla\cdot u (t)=0$ and   $\det F_0=1$ implies $\det F (t)=1$ for all $t\in [0,T_*)$, 
see Remark~\ref{Remark:main_thm}(iii).
Here $T_*$ is the maximal existence time in Theorem~\ref{Thm: Local wellposedness}. 

\medskip
This enables us to introduce a new variable  $G= F^{-1} -I_N$. 
Its corresponding initial value is defined by   $G_0= F_0^{-1} -I_N \in H^s $.
In virtue of the relation $H=-(I_N+G)^{-1}G$ and a Neumann series argument, one can show that  for any $r\in \bN\cup \{0\}$, there exist positive constants $\Lambda_r$ such that
\begin{equation}
\label{G and H}
\Lambda_r^{-1} \| H (t) \|_{H^r} \leq \| G (t) \|_{H^r} \leq \Lambda_r \| H (t) \|_{H^r}
\end{equation}
provided $\|G\|_{H^k} <  1/\widetilde{C}_k $, where $k=\max\{2,r\}$ and   $\widetilde{C}_k \geq 1$ is the multiplier constant in
$$
\| uv\|_{H^k} \leq \widetilde{C}_k \| u \|_{H^k} \| v \|_{H^k}.
$$ 
Further, according to \cite[Section 3]{ChenZhang06}
\begin{align*}
\nabla\cdot (F F^{\sT}) = \sum_{i,j=0}^\infty (-1)^{i+j} \nabla \cdot (G^i (G^{\sT})^j) = g(G)  - \nabla \cdot G - \nabla \cdot G^{\sT},
\end{align*}
where
\begin{equation}
\label{g-G}
g(G):=\sum_{\substack{i,j=0, \\ i+j\geq 2}}^\infty (-1)^{i+j} \nabla \cdot (G^i (G^{\sT})^j).
\end{equation}
The RHS of \eqref{g-G} converges in $H^{s-1}$ when $\|G\|_{H^s}$ is sufficiently small as
\[
\|g(G)\|_{H^{s-1}}   \le \sum_{k=2}^\infty  (k+1)  \widetilde{C}_s^{k-1}  \|G\|_{H^s}^k .
\]
Moreover, it follows from \cite[p.1807-1808]{ChenZhang06} that
\begin{equation}
\label{est of gG-0}
\begin{split}
\| \nabla g(G) \|_{H^{s-2}}  \leq \sum_{k=2}^\infty k(k+1) C_s^k  \|G\|_{H^s}^{k-1} \| \nabla^2 G\|_{H^{s-2}} 
\end{split}
\end{equation}
for some constant $C_s$ that only depends on $s$.
Without loss of generality, we may assume that $C_s\ge \max\{\widetilde{C}_s, 1\}$.
By Theorem~\ref{Thm: Local wellposedness}, when $\|F_0-I_N\|_{H^s}$ is sufficiently small, there is some $\widehat{T} >0 $ such that 
\[
\widehat{T}:=   \sup  \left\{ T\in [0,T_*):\, \sup_{t\in [0,T]}  \|G(t)\|_{H^s} < 1/ C_s \right\}
\]


\subsection{Basic energy Estimate}

\begin{lemma}
For all $t\in [0,\widehat{T})$, it holds that
\begin{equation}
\label{Conserved quantity}
\frac{d}{dt} ( \|u (t) \|_{L_2}^2 +  \| H (t) \|_{L_2}^2 + \|\nabla m (t)\|_{L_2}^2    ) + 2 \underline{\nu}  \| \nabla u (t)\|_{L_2}^2 + 2\underline{ \alpha } \| m \times \Delta m (t) \|_{L_2}^2  \leq 0.
\end{equation}
\end{lemma}
\begin{proof}
Note that the assumption $\nabla \cdot F=0$ implies $\nabla \cdot H =0$.
Multiplying \eqref{magneto sys 3}$_1$ by $u$ and integrating it over $ \bR^N$ yields
\begin{align*}
& \frac{1}{2} \frac{d}{dt} \| u \|_{L_2}^2 + \int_{\bR^N}  \nu (\theta)    | \nabla u  |^2 \, dx   + \int_{ \bR^N} (u\cdot \nabla u )\cdot u  \, dx + \int_{ \bR^N} (\nabla \cdot  (\nabla m \odot \nabla m) ) \cdot u \, dx  \\
& =  \int_{ \bR^N} (\nabla \cdot(H H^{\sT})) \cdot u \, dx +  \int_{ \bR^N} (\nabla \cdot  H^{\sT} ) \cdot u \, dx ,
\end{align*}
Using the condition $\nabla\cdot u=0$, we have
$$
\int_{ \bR^N} (u\cdot \nabla u )\cdot u  \, dx = 0.
$$
By \eqref{divergence-property},  we have
\begin{align*}
 \int_{ \bR^N} (\nabla \cdot  (\nabla m \odot \nabla m) ) \cdot u \, dx = - \int_{ \bR^N}  \nabla m \odot \nabla m : \nabla u \, dx
\end{align*}
and
\begin{align*}
\int_{ \bR^N} (\nabla \cdot(H H^{\sT})) \cdot u \, dx  = - \int_{ \bR^N} H H^{\sT} : \nabla u \, dx = - \int_{ \bR^N}   (\nabla u )^{\sT} H : H \, dx  .
\end{align*}
The last term can be expressed as
\begin{align*}
 \int_{ \bR^N} (\nabla \cdot  H^{\sT} ) \cdot u \, dx = - \int_{ \bR^N} (\nabla u)^{\sT} : H\, dx.
\end{align*}
Next, multiplying \eqref{magneto sys 3}$_3$ by $H$ and integrating it over $ \bR^N$ yield
\begin{align*}
\frac{1}{2} \frac{d}{dt} \| H \|_{L_2}^2 +  \int_{ \bR^N} (u\cdot \nabla H):H  \, dx = \int_{ \bR^N} (\nabla u)^{\sT} H:H \, dx + \int_{ \bR^N} (\nabla u)^{\sT} :H \, dx  .
\end{align*}
It again follows from the condition $\nabla\cdot u=0$ that
$$
\int_{ \bR^N} (u\cdot \nabla H ):H  \, dx = 0.
$$
Multiplying \eqref{magneto sys 3}$_5$ by $\Delta m$ and integrating it over $ \bR^N$ yields
\begin{align*}
-\frac{1}{2} \frac{d}{dt} \| \nabla m \|_{L_2}^2 +  \int_{ \bR^N} (u\cdot \nabla m) \cdot \Delta m \, dx  + \int_{ \bR^N}\alpha (\theta) (m\times (m\times \Delta m))\cdot\Delta m \, dx = 0 ,
\end{align*}
where we have used the original formulation of the $m$-equation:
\[
\partial_t m + u \cdot\nabla m  = - \alpha (\theta) m\times (m \times \Delta m) - \beta (\theta) m \times \Delta m
\]
and the fact that $(m\times \Delta m)\cdot \Delta m =0$. 
Direct computations show
\begin{align*}
\int_{ \bR^N} (u\cdot \nabla m) \cdot \Delta m \, dx & = \int_{ \bR^N}   u^i \partial_i m_j \partial_k^2 m_j  \, dx \\
&= - \int_{ \bR^N}  ( \partial_k u^i \partial_i m_j \partial_k m_j + u^i \partial_i \partial_k m_j \partial_k m_j  )\, dx = -\int_{ \bR^N}  ( \partial_k u^i \partial_i m_j \partial_k m_j)\, dx \\
&= - \int_{ \bR^N} \nabla m \odot \nabla m : \nabla u\, dx,
\end{align*}
and
\begin{align*}
 \int_{ \bR^N} \alpha (\theta) (m\times (m\times \Delta m)) \cdot \Delta m  \, dx& =- \int_{ \bR^N} \alpha (\theta) | m \times \Delta m|^2 \, dx.  
\end{align*}
Combining all these equations, we thus derive  \eqref{Conserved quantity}.
\end{proof}

\begin{remark}
Note that we have used the divergence-free condition $\nabla \cdot F=0$ in the above proof.
\end{remark}


\subsection{An augmented system}

According to Remark~\ref{Remark:main_thm}(iv), $G(t,x)=\left(\nabla_x X(t,x) \right)^{\sT}-I_N$, where $X$ denotes the Lagrangian coordinates and $x$ the Eulerian coordinates, for all $(t,x)\in [0,T_*)\times \bR^N$.
Therefore $G^{\sT}$ satisfies the curl-free relation
\begin{equation}
\label{G-relation}
\partial_k G_{ji}= \partial_i G_{jk},
\end{equation}
see also \cite[Formula~(2.2b)]{SiderisThomases05};
and due to \cite[Formula~(2.5a)]{SiderisThomases05}, \eqref{magneto sys}$_3$ can be replaced by
\begin{equation}
\label{eq-G}
\partial_t G + (\nabla u)^{\sT} + u\cdot \nabla G + G (\nabla u)^{\sT} = 0.
\end{equation}
Using \eqref{g-G}, \eqref{magneto sys}$_1$ becomes
\begin{equation}
\label{eq-u-new}
\partial_t u + u \cdot \nabla u - \nabla \cdot ( \nu (\theta) \nabla u )  +\nabla \cdot G + \nabla \cdot G^{\sT} +\nabla \pi  =g(G) -\nabla \cdot(\nabla m \odot \nabla m) .
\end{equation}
Taking $\nabla $ on both sides yields
\begin{align*}
& \partial_t (\nabla u) + \nabla (u\cdot \nabla u)  - \nabla (\nabla \cdot ( \nu(\theta) \nabla u)) + \nabla( \nabla \cdot G + \nabla \cdot G^{\sT} ) + \nabla^2 \pi \\
&= \nabla  g(G)  - \nabla (\nabla \cdot(\nabla m \odot \nabla m) ).
\end{align*}
Direct computations show
\begin{align*}
\nabla (u\cdot \nabla u) &= u\cdot \nabla (\nabla u)  +  \nabla u(\nabla u) 
\end{align*}
and by \eqref{curl-curl-matrix}
\begin{align*}
\nabla (\nabla \cdot ( \nu(\theta) \nabla u)) & = \Delta ( \nu (\theta) \nabla u)  +   \curl \curl \nabla ( \nu(\theta)   u) 
- \curl \curl  ( \nabla \nu (\theta) \otimes u)\\
&= \Delta ( \nu (\theta) \nabla u) - \curl \curl ( \nabla \nu (\theta) \otimes u).
\end{align*}
In virtue of the curl-free condition relation~\eqref{G-relation}, one immediately obtains 
\begin{equation}
\label{Laplacian of G}
\nabla (\nabla \cdot G^{\sT}) =\Delta G^{\sT}  .
\end{equation}
We introduce a new variable $M:=\nu (\theta )  \nabla u - G^{\sT} $ with initial value $M_0: = \nu (\theta_0)  \nabla u_0 - G_0^{\sT} \in H^{s-1}$. 
Then $M$ solves the equation
\begin{equation}
\label{eq-M}
\begin{split}
& \partial_t M + (\nabla u) (M -I_N)  + u\cdot \nabla M  -\nu (\theta)  \Delta M + \nu (\theta) \nabla (\nabla \cdot G) + \nu (\theta) \nabla^2 \pi \\
& \quad  + \nu (\theta) \curl \curl ( \nabla \nu (\theta) \otimes u) - \nu' (\theta) (\nabla u) \nabla \cdot ( \kappa (\theta) \nabla \theta)\\
&= \nu (\theta)  \nabla   g(G)  - \nu (\theta)  \nabla (\nabla \cdot(\nabla m \odot \nabla m) ) + \nu' (\theta)\nu (\theta) \nabla u |\nabla u|^2 + \nu' (\theta) \alpha (\theta) \nabla u | \Delta m + | \nabla m|^2 m |^2.
\end{split}
\end{equation}
Note that the function $  M$ here corresponds to $M^{\sT}$ in \cite{ChenZhang06}  because we use a different convention for the 
gradient operator. 
Gathering \eqref{eq-G}, \eqref{eq-u-new}, and \eqref{eq-M} leads to the following augmented system:
\begin{equation}
\label{magneto sys-augmented}
\left\{\begin{aligned}
\partial_t u + u \cdot \nabla u - \nabla \cdot (\nu (\theta ) \nabla u )  +\nabla \cdot G + \nabla \cdot G^{\sT} +\nabla \pi  & =g(G) -\nabla \cdot(\nabla m \odot \nabla m),\\
\nabla \cdot u &=0 ,\\
\partial_t G + (\nabla u )^{\sT} + u\cdot \nabla G + G (\nabla u)^{\sT} &= 0    ,\\
\partial_t\theta+u\cdot\nabla\theta -  \nabla \cdot ( \kappa (\theta) \nabla \theta ) &= \nu(\theta)   |\nabla u|^2 + \alpha(\theta) |\Delta m+|\nabla m|^2m|^2\\
\partial_t m + u \cdot\nabla m  -\alpha (\theta) \Delta m  &= \alpha (\theta) |\nabla m|^2 m - \beta (\theta) m \times \Delta m ,\\
\partial_t M   + u\cdot \nabla M   
 -\nu (\theta)  \Delta M + \nu (\theta) \curl \curl ( \nabla \nu (\theta) & \otimes u)  - \nu' (\theta) (\nabla u) \nabla \cdot ( \kappa (\theta) \nabla \theta)\\
\quad + (\nabla u) (M -I_N) + \nu (\theta) \nabla (\nabla \cdot G) + \nu (\theta) \nabla^2 \pi 
&= \nu (\theta)  \nabla   g(G)  - \nu (\theta)  \nabla (\nabla \cdot(\nabla m \odot \nabla m) )\\
  + \nu' (\theta)\nu (\theta) \nabla u |\nabla u|^2 
&  + \nu' (\theta) \alpha (\theta) \nabla u | \Delta m + | \nabla m|^2 m |^2, \\
(u(0), G(0), \theta (0), m(0), M(0))& =(u_0, G_0, \theta_0, m_0, M_0 ) . 
\end{aligned}\right.
\end{equation}

In the following, we will perform a priori estimate for equations in \eqref{magneto sys-augmented} one by one. To this end, we will introduce a number of notations and derive a series of estimates. 
For any $j\in \{1,\cdots, N\}$, we define $\delta^j = (\delta^j_i)_{i=1}^N$. Given any   multi-index $\sigma \in \bN^N$, we define $\sigma^j = \sigma + \delta^j$.  
Define the energy functional  to be
\begin{align*}
\cE(t) =  \| u (t)   \|_{H^s}^2 + \| G (t)  \|_{H^s}^2 + \| \theta (t) - \theta_* \|_{H^s}^2 + \|\nabla m (t) \|_{H^s}^2 + \| M(t) \|_{H^{s-1}}^2
\end{align*}
and the dissipation rate functional to be
\begin{align*}
\cD(t) = \underline{\nu} \| \nabla u (t)\|_{H^s}^2 + \underline{\alpha} \| \Delta m(t) \|_{H^s}^2 + \underline{\kappa} \| \nabla \theta (t) \|_{H^s}^2 + 
\underline{\nu} \| \nabla^2 M (t)\|_{H^{s-2}}^2.
\end{align*}
In what follows, we occasionally use estimates of the form
$$
\cE(t)\le \cE(t)^{1/2} + \cE(t)^{3/2},  \quad \cE(t)^2 \le \cE(t)^{1/2} + \cE(t)^3, \quad \cE(t)^{5/2}\le  \cE(t)^{1/2} + \cE(t)^3.
$$
Note that $\nabla^2 G^{\sT} = \nabla^2 (\nu (\theta) \nabla u) - \nabla^2 M$. Thus   we have by Lemmas~\ref{Lem: Moser} and \ref{Lem: Nemyskii} that
\begin{equation}
\label{2nd order est of G}
\begin{split}
\| \nabla^2 G \|_{H^{s-2}} & \leq   \|\nabla^2 (\nu (\theta) \nabla u) \|_{H^{s-2}} + \| \nabla^2 M \|_{H^{s-2}}  \leq   \| \nu (\theta) \nabla u  \|_{H^s} + \| \nabla^2 M \|_{H^{s-2}}  \\
&\leq C (  \| \nu(\theta) \|_\infty \|\nabla^{s+1} u \|_{L_2} + \|\nabla^s \nu (\theta) \|_{L_2} \| \nabla u\|_\infty ) + \| \nabla^2 M \|_{H^{s-2}} \\
&\leq K (\| \theta   \|_\infty  )(1+\| \theta  -\theta_*\|_{H^s} )\| \nabla  u \|_{H^s}   +  \| \nabla^2 M \|_{H^{s-2}} \\
&\leq K (\| \theta   \|_\infty  )    (1+ \cE(t)^{1/2} )  \cD(t)^{1/2}.
\end{split}
\end{equation}
Throughout the rest of this article, we will frequently use the following estimate
\begin{equation}
\label{est of gG}
\begin{split}
\| \nabla g(G) \|_{H^{s-2}} \leq K (\| \theta   \|_\infty  )\sum_{k=2}^\infty k(k+1) C_s^k  \|G\|_{H^s}^{k-1} (1+ \cE(t)^{1/2}) \cD(t)^{1/2} ,
\end{split}
\end{equation}
which follows from \eqref{est of gG-0} and \eqref{2nd order est of G}.
Taking the divergence on both sides of \eqref{magneto sys-augmented}$_1$ yields
\begin{equation}
\label{Lap of pressure}
\Delta \pi   = \nabla \cdot ( \nabla \cdot (\nu (\theta) \nabla u )) - \nabla\cdot (u\cdot \nabla u) - 2 \Delta \tr G + \nabla \cdot g(G) - \nabla \cdot ( \nabla \cdot ( \nabla m \odot \nabla m )) ,
\end{equation}
where we have used \eqref{G-relation} to obtain
\begin{align*}
\nabla\cdot (\nabla \cdot G + \nabla \cdot G^{\sT} )  = 2 \partial_i \partial_j G_{ij} = 2 \partial_j \partial_j  G_{ii} = 2 \Delta \tr G.
\end{align*}
For future analysis, we will need the following estimate for $\Delta \pi$.
\begin{lemma}\label{Lem: Lap pressure est}
It holds that 
$$
\|\nabla^2 \pi \|_{H^{s-2}} \leq  K (\| \theta   \|_\infty  ) \left(1+ \cE(t)^{1/2}\right)\left(1 + \sum_{k=2}^\infty k(k+1) C_s^k  \|G\|_{H^s}^{k-1} \right) \cD(t)^{1/2}.
$$
\end{lemma}
\begin{proof}
Based on \eqref{est of gG-0}, \eqref{Lap of pressure} and Lemma~\ref{Lem: Moser}, direct computations  show
\begin{align*}
& \quad \|\nabla^2 \pi \|_{H^{s-2}}   \overset{\cdot}{=}  \|\Delta \pi \|_{H^{s-2}} \\
&\leq \| \nabla^2 (\nu(\theta) \nabla u )\|_{H^{s-2}} +  \| \nabla \cdot (u \cdot \nabla u)\|_{H^{s-2}} + C \|\nabla^2 G \|_{H^{s-2}}  + \|\nabla g(G) \|_{H^{s-2}} \\
&\quad  + \| \nabla^2 ( \nabla m \odot \nabla m ) \|_{H^{s-2}}   \\
&\leq K (\| \theta   \|_\infty  )  \Big[ ( 1+ \| \nabla^2 \theta \|_{H^{s-2}} + \| u \|_{H^s}) \| \nabla u \|_{H^s} +  \left(1+ \sum_{k=2}^\infty k(k+1) C_s^k  \|G\|_{H^s}^{k-1} \right) \|\nabla^2 G \|_{H^{s-2}}     \\
&\quad   + \|\nabla m \|_{H^s}   \| \Delta m \|_{H^s} \Big],
\end{align*}
where $\overset{\cdot}{=}$ means equivalent norms.
\end{proof}

\subsection{Estimates of $u$ and $G$ equations}\label{Section:u est}
Taking $\partial^\sigma \partial_j$ with $  |\sigma|\leq s-1 $  on both sides of \eqref{magneto sys-augmented}$_1$, multiplying by $\partial^\sigma \partial_j u$, and then integrating over $ \bR^N$ yields
\begin{equation}
\label{higher order u}
\begin{split}
&\frac{1}{2} \frac{d}{dt} \|\nabla u \|_{H^{s-1}}^2  + \underline{\nu}  \| \nabla^2 u \|_{H^{s-1}}^2  \\
& \leq  \sum_{|\sigma|\leq s-1} \left[ ( \partial^\sigma \partial_j g(G) | \partial^\sigma \partial_j u) - ( \partial^\sigma \partial_j \nabla \cdot G | \partial^\sigma \partial_j u) -   ( \partial^\sigma \partial_j \nabla \cdot G^{\sT} | \partial^\sigma \partial_j u) \right. \\
&\quad \left. -  ( \partial^\sigma \partial_j  (\nabla \cdot (\nabla m \odot \nabla m ))|  \partial^\sigma \partial_j  u ) -  ( \partial^\sigma \partial_j (u\cdot \nabla u ) | \partial^\sigma \partial_j u) - ( \partial^\sigma \partial_j \nabla \pi | \partial^\sigma \partial_j u) \right. \\
& \quad \left.   - ( \cR^{\sigma^j}_{\nu } \nabla u | \partial^\sigma \partial_j  \nabla u ) \right] .
\end{split}
\end{equation}
From the incompressibility condition $\nabla \cdot u=0$, \eqref{G-relation}, and \eqref{Laplacian of G}, 	we infer  
\begin{equation}
\label{est 1 higher order u}
\begin{split}
( \partial^\sigma \partial_j \nabla \pi | \partial^\sigma \partial_j u) & = 0 \\
 ( \partial^\sigma \partial_j \nabla \cdot G | \partial^\sigma \partial_j u) &= ( \partial^\sigma \partial_j \partial_k G_{ki} | \partial^\sigma \partial_j u^i) = ( \partial^\sigma  \partial_k \partial_i G_{kj} | \partial^\sigma \partial_j u^i) =0 \\
 ( \partial^\sigma \partial_j \nabla \cdot G^{\sT} | \partial^\sigma \partial_j u)  & = ( \partial^\sigma \partial_j \partial_k G_{ik} | \partial^\sigma \partial_j u^i) =  ( \partial^\sigma \partial_k \partial_k  G_{ij} | \partial^\sigma \partial_j u^i) = ( \partial^\sigma \Delta   G  | \partial^\sigma (\nabla u)^{\sT}).
\end{split}
\end{equation}

Taking $\partial^\sigma \partial_j$ with $  |\sigma|\leq s-1 $  on both sides of \eqref{magneto sys-augmented}$_3$, multiplying by $\partial^\sigma \partial_j G$, and then integrating over $ \bR^N$ yield
\begin{equation}
\label{higher order G}
\begin{split}
&\frac{1}{2} \frac{d}{dt} \|\nabla G \|_{H^{s-1}}^2   \\
&=  - \sum_{|\sigma| \leq s-1} \left[ ( \partial^\sigma \partial_j  (\nabla u)^{\sT}  | \partial^\sigma \partial_j   G )  +   ( \partial^\sigma \partial_j  (u \cdot \nabla G)   | \partial^\sigma \partial_j   G )   +  ( \partial^\sigma \partial_j  (G(\nabla u)^{\sT} ) | \partial^\sigma \partial_j   G ) \right] \\
&=   \sum_{|\sigma| \leq s-1}  \left[   ( \partial^\sigma    (\nabla u)^{\sT}  | \partial^\sigma \Delta   G ) -   ( \partial^\sigma \partial_j  (u \cdot \nabla G)   | \partial^\sigma \partial_j   G )  -  ( \partial^\sigma \partial_j  (G(\nabla u)^{\sT} ) | \partial^\sigma \partial_j   G ) \right].
\end{split}
\end{equation}
Combining   \eqref{higher order u}-\eqref{higher order G}  yields
\begin{equation}
\label{higher order u and G-2}
\begin{split}
&\frac{1}{2} \frac{d}{dt} \left( \|\nabla u \|_{H^{s-1}}^2 + \|\nabla G \|_{H^{s-1}}^2  \right)  + \underline{\nu} \| \nabla^2 u \|_{H^{s-1}}^2  \\
&  \leq  \sum_{|\sigma| \leq s-1} \left[  
( \partial^\sigma \partial_j g(G) | \partial^\sigma \partial_j u)  -( \partial^\sigma \partial_j (u\cdot \nabla u ) | \partial^\sigma \partial_j u)  
 -  ( \partial^\sigma \partial_j  (\nabla \cdot (\nabla m \odot \nabla m ))|  \partial^\sigma \partial_j  u )   \right. \\  
&\quad \left. - ( \partial^\sigma \partial_j  (u \cdot \nabla G)   | \partial^\sigma \partial_j   G )  -  ( \partial^\sigma \partial_j  (G(\nabla u)^{\sT} ) | \partial^\sigma \partial_j   G )   - ( \cR^{\sigma^j}_{\nu } \nabla u | \partial^\sigma \partial_j  \nabla u )\right] \\
&=: G.1 + \cdots + G.6.
\end{split}
\end{equation}
According to \eqref{est of gG}, we have
\begin{equation}
\label{est gG u}
\begin{split}
  G.1
 &\leq  \sum_{k=2}^\infty k(k+1) C_s^k  \| G\|_{H^s}^{k-1} \| \nabla^2 G \|_{H^{s-2}} \| \nabla u \|_{H^s} \\
&\leq K (\| \theta   \|_\infty  )  \sum_{k=2}^\infty k(k+1) C_s^k  \| G\|_{H^s}^{k-1}  (1 +   \cE(t)^{1/2}  )\cD(t).
\end{split}
\end{equation}
To estimate $G.2$, we use the incompressibility condition $\nabla \cdot u=0$ and \eqref{Sobolev embedding}:
\begin{equation}
\label{est u 2 higher oder}
\begin{split}
  G.2 
&\leq C  \sum_{\tau< \sigma^j} | ( \partial^{\sigma^j -\tau} u \cdot \nabla (\partial^\tau u)| \partial^\sigma\partial_j u) |  
 \leq  C\sum_{\tau< \sigma^j} \| \partial^{\sigma^j -\tau} u \|_{L_4} \| \nabla (\partial^\tau u) \|_{L_4} \| \partial^\sigma\partial_j u \|_{L_2} \\
&\leq  C \sum_{\tau< \sigma^j} \| \partial^{\sigma^j -\tau} u \|_{H^1} \| \nabla (\partial^\tau u) \|_{H^1} \| \partial^\sigma\partial_j u \|_{L_2} \\
& \leq C \| u  \|_{H^s} \| \nabla  u \|_{H^s}^2 \leq C \cE(t)^{1/2}\cD(t). 
\end{split}
\end{equation}
Similarly, 
\begin{equation}
\label{est odot m higher oder}
\begin{split}
   G.3  
&\leq  C \sum_{\tau \leq  \sigma^j}  | ( \partial^{\sigma^j-\tau} \nabla m \odot \partial^\tau \nabla m | \partial^\sigma  \partial_j \nabla u )| \\
& \leq C \sum_{\tau \leq  \sigma^j} \|\partial^{\sigma^j-\tau} \nabla m\|_{H^1}  \|\partial^\tau \nabla m \|_{H^1} \|\partial^\sigma  \partial_j \nabla u\|_{L_2}  
  \leq C \| \nabla m \|_{H^s} \| \nabla^2   m\|_{H^s}\| \nabla u\|_{H^s} \\
& \leq  C  \| \nabla m \|_{H^s}   \left( \| \nabla^2   m\|_{H^s}^2 + \| \nabla u\|_{H^s}^2  \right) \leq C \cE(t)^{1/2}\cD(t) .
\end{split}
\end{equation}
We split $G.4$ into
\begin{align*}
 G.4 \leq 
 \left|  (  \partial_j  (u \cdot \nabla G)   |   \partial_j   G ) \right| +  \sum_{1\leq |\sigma| \leq s-1}  \left|  ( \partial^\sigma \partial_j  (u \cdot \nabla G)   | \partial^\sigma \partial_j   G ) \right|.
\end{align*} 
The incompressibility condition $\nabla \cdot u=0$ and \eqref{2nd order est of G}  imply  
\begin{align*}
  \left|  (  \partial_j  (u \cdot \nabla G)   |   \partial_j   G ) \right| &= \left|  (  (\partial_j  u^k) \partial_k   G    |   \partial_j   G ) \right| = \left|  (  (\partial_j  u^k )   G)   |   \partial_j \partial_k  G ) \right|  \leq   \|G\|_\infty \| \nabla u\|_{L_2}  \| \nabla^2 G\|_{L_2} \\
  & \leq \|G\|_{H^s} \| \nabla u\|_{H^s}  \| \nabla^2 G\|_{H^{s-2}}  \\
  & \leq K( \| \theta\|_\infty ) (\cE(t)^{1/2} + \cE(t)) \cD(t).
\end{align*}
The remaining part can again be estimated by following a similar argument to \eqref{est u 2 higher oder}
\begin{align*}
& \quad \sum_{1\leq |\sigma| \leq s-1}  \left|  ( \partial^\sigma \partial_j  (u \cdot \nabla G)   | \partial^\sigma \partial_j   G ) \right|  \\
&  \leq C \sum_{1\leq |\sigma| \leq s-1} \sum_{0\neq \tau<\sigma^j}   \left| ( \partial^{\sigma^j - \tau} u \cdot \nabla (\partial^\tau G)   | \partial^\sigma \partial_j   G ) \right|  + \sum_{1\leq |\sigma| \leq s-1} \left| ( \partial^\sigma \partial_j u \cdot \nabla G   | \partial^\sigma \partial_j   G ) \right| \\
& \leq C  ( \| \nabla u \|_{H^{s-1}} \| \nabla^2 G\|_{H^{s-2}} + \| \nabla G \|_{H^1} \| \nabla u \|_{H^s} ) \|\nabla^2 G \|_{H^{s-2}} \\
&\leq K(\|\theta\|_\infty)  (\cE(t)^{1/2} + \cE(t)^{3/2} ) \cD(t). 
\end{align*}
Combining the above computations, we have
\begin{equation}
\label{est diverg u other terms}
G.4 \leq  K(\|\theta\|_\infty)  (\cE(t)^{1/2} +  \cE(t)^{3/2} ) \cD(t). 
\end{equation}
A similar computation to \eqref{est diverg u other terms} shows
\begin{align*}
   G.5
   \leq K( \| \theta\|_\infty ) (\cE(t)^{1/2} + \cE(t)) \cD(t).
\end{align*}
Applying Lemma~\ref{Lem: Commutator}, one can obtain
\begin{align*}
  G.6 & \leq  \| \cR^{\sigma^j}_{\nu }  \nabla u \|_{L_2}      \|\nabla  u\|_{H^s}   \leq K(\|\theta  \|_\infty )   \| \theta - \theta_* \|_{H^s}  \|\nabla  u\|_{H^s}^2  \\
&\leq K (\| \theta   \|_\infty  )     \cE(t)^{1/2}   \cD(t). 
\end{align*}
To sum up, the following estimate for $u$ and $G$ holds.
\begin{equation}
\label{higher order u and G final}
\begin{split}
&\frac{1}{2} \frac{d}{dt} \left[ \|\nabla u \|_{H^{s-1}}^2 + \|\nabla G \|_{H^{s-1}}^2  \right]  + \underline{\nu}  \| \nabla^2 u \|_{H^{s-1}}^2  \\
& \leq K(\|\theta\|_\infty) \left[  \sum_{k=2}^\infty k(k+1) C_s^k  \| G\|_{H^s}^{k-1} (  1+ \cE(t)^{1/2}     )  + \cE(t)^{1/2}   + \cE(t)^{3/2}  \right] \cD(t)
\\
& \leq K(\|\theta\|_\infty) \left[  \sum_{k=2}^\infty k(k+1) C_s^k  \| G\|_{H^s}^{k-2} (   \cE(t)^{1/2} + \cE(t)   )  + \cE(t)^{1/2}   + \cE(t)^{3/2}  \right] \cD(t) .
\end{split}
\end{equation}

\subsection{Estimates of $\theta$-equation}
Taking $\partial^\sigma$ with $  |\sigma|\leq s $  on both sides of \eqref{magneto sys-augmented}$_4$, multiplying by $\partial^\sigma (\theta - \theta_*) $, and then integrating over $ \bR^N$ yield
\begin{equation}
\label{higher order theta}
\begin{split}
&\frac{1}{2} \frac{d}{dt} \| \theta (t) -\theta_*  \|_{H^s}^2  + \underline{\kappa} \| \nabla \theta \|_{H^s}^2    \\
&= \sum_{|\sigma|\leq s} \left[     ( \partial^\sigma (\nu (\theta)  |\nabla u|^2 ) | \partial^\sigma (\theta - \theta_*) )  +   (\partial^\sigma  ( \alpha (\theta) |  \Delta m + |\nabla m|^2 m |^2 ) | \partial^\sigma (\theta - \theta_*) ) \right. \\
&\quad \left. - ( \partial^\sigma   (u\cdot \nabla \theta ) | \partial^\sigma (\theta - \theta_*) )  - (\cR^\sigma_\kappa \nabla \theta | \partial^\sigma \nabla \theta )  \right] = : T.1 + \cdots + T.4.
\end{split}
\end{equation}
Direct computations show 
\begin{equation}
\label{est T1}
\begin{split}
	T.1 &=  ( \nu (\theta)  \partial^\sigma  |\nabla u|^2 | \partial^\sigma (\theta - \theta_*) ) + (  \partial^\sigma \nu (\theta)   |\nabla u|^2 | \partial^\sigma (\theta - \theta_*) ) \\
	&\quad + \sum_{0< \tau <\sigma }
\begin{pmatrix}
\sigma \\
\tau
\end{pmatrix}	
	  ( \partial^{\sigma -\tau } (\nu (\theta) ) \partial^\tau  |\nabla u|^2 | \partial^\sigma (\theta - \theta_*) )  \\
&\leq K(\|\theta \|_\infty ) \left(	\|\theta - \theta_* \|_{H^s}  \| |\nabla u |^2 \|_{H^s} +   \|\theta - \theta_* \|_{H^s}^2  \| |\nabla u |^2\|_\infty + \|\theta - \theta_* \|_{H^s}^2 \| |\nabla u |^2 \|_{H^s}  \right) \\
	&\leq K(\|\theta \|_\infty )  ( \|\theta - \theta_* \|_{H^s}  +\|\theta - \theta_* \|_{H^s}^2 ) \| \nabla u \|_{H^s}^2  \\
	&\leq K(\|\theta \|_\infty ) ( \cE(t)^{1/2} + \cE(t)   ) \cD(t).
\end{split}
\end{equation}
Here we have used  the fact that $H^s(\bR^N)$ is a Banach algebra. 
To estimate $T.2$, we apply Lemma~\ref{Lem: Moser} to obtain
\begin{equation}
\label{est Delta m sq}
\begin{split}
& \quad \| | \Delta m + |\nabla m|^2 m |^2 \|_{H^s}  = \| | m \times (m \times \Delta m ) |^2 \|_{H^s} \leq C  \|\Delta m \|_\infty \|m \times (m \times \Delta m )  \|_{H^s}   \\
& \leq C \|  \Delta m \|_{H^s}  ( \| m \times \Delta m  \|_{H^s} + \|\nabla^s m \|_{L_2} \| m \times \Delta m  \|_\infty ) \\
& \leq C \|  \Delta m \|_{H^s} \left[   \|  \Delta m \|_{H^s} +  2 \|\nabla m \|_{H^s}  \|  \Delta m \|_{H^s} \right] \\
&\leq C  ( 1+ \|\nabla m \|_{H^s} ) \|  \Delta m \|_{H^s}^2.
\end{split}
\end{equation} 
Based on \eqref{Sobolev embedding} and \eqref{est Delta m sq}, we have
\begin{equation}
\label{est T2}
\begin{split}
T.2 & = (  (\alpha (\theta) ) \partial^\sigma | \Delta m + |\nabla m|^2 m |^2  | \partial^\sigma (\theta - \theta_*)   )  +  ( \partial^{\sigma  } ( \alpha (\theta))  | m \times (m \times \Delta m)|^2  | \partial^\sigma (\theta - \theta_*)    ) \\
&\quad + \sum_{0< \tau  <\sigma } 
\begin{pmatrix}
\sigma \\
\tau
\end{pmatrix}	
( \partial^{\sigma-\tau} (\alpha (\theta)) \partial^\tau (| \Delta m + |\nabla m|^2 m |^2) | \partial^\sigma (\theta - \theta_*)  )  \\
&\leq  K(\|\theta \|_\infty )\left[( \| \theta - \theta_*  \|_{H^s}   +\| \theta - \theta_*  \|_{H^s}^2 ) \| | \Delta m + |\nabla m|^2 m |^2 \|_{H^s}  +     \|  | m \times (m \times \Delta m)|^2  \|_{L_2}   \|\nabla \theta \|_{H^s}^2 \right] \\
&\leq  K(\|\theta \|_\infty ) ( \cE(t)^{1/2}   +\cE(t)^{3/2} ) \cD(t).
\end{split}
\end{equation}
Note that in the first line of \eqref{est T2}, the second term on RHS is the same as the first term when $\sigma=0$.
$T.3$ can be estimated as in G.2:
\begin{equation}
\label{est T3}
T.3 \leq C \|\theta -\theta_* \|_{H^s} ( \|\nabla u \|_{H^s}^2 + \| \nabla  \theta\|_{H^s}^2) \leq   C \cE(t)^{1/2}   \cD(t).
\end{equation}
The last term can be estimated by using Lemma~\ref{Lem: Commutator}:
\begin{equation}
\label{est T4}
\begin{split}
T.4 & \leq  \| \cR^\sigma_\kappa \nabla \theta \|_{L_2} \|\nabla \theta \|_{H^s} \leq K(\|\theta \|_\infty) \|\nabla \theta\|_{H^{s-1}} \|\nabla \theta \|_{H^s}^2 \leq K(\|\theta\|_\infty) \cE(t)^{1/2} \cD(t).
\end{split}
\end{equation}
Gathering \eqref{est T1} and \eqref{est T2}-\eqref{est T4}, we derive that
\begin{equation}
\label{higher order theta-final}
\begin{split}
 \frac{1}{2} \frac{d}{dt} \| \theta (t) -\theta_* \|_{H^s}^2  + \underline{\kappa} \| \nabla \theta \|_{H^s}^2   
 \leq  K(\|\theta \|_\infty ) ( \cE(t)^{1/2}    +\cE(t)^{3/2} ) \cD(t).  
\end{split}
\end{equation}


\subsection{Estimates of $m$-equation}
Next, we will estimate the equation of $m$. 
Taking $\partial^\sigma  $ with $ |\sigma |\leq s $ on both sides of \eqref{magneto sys-augmented}$_5$, multiplying by $\partial^\sigma \Delta m$, and then integrating over $ \bR^N$ yields
\begin{equation}
\label{Higher order est m-eq}
\begin{split}
&\frac{1}{2} \frac{d}{dt} \|  \nabla  m\|_{H^s }^2 + \underline{\alpha} \| \Delta  m \|_{H^s}^2    \\
&\leq -     \sum_{ |\sigma|\leq s }\left[ (\partial^\sigma   (u\cdot \nabla m) | \partial^\sigma \Delta m  ) +( \partial^\sigma  ( \alpha (\theta) |\nabla m|^2 m ) | \partial^\sigma \Delta m ) \right. \\
&\quad \left. -  ( \partial^\sigma   ( \beta(\theta) m\times \Delta m ) | \partial^\sigma \Delta m ) +  ( \cR^\sigma_\alpha \Delta m   | \partial^\sigma \Delta m ) \right] = : m.1 + \cdots + m.4.
\end{split}
\end{equation}
By \cite[Equation~(5.31)]{JiangLiuLuo23}, we have
\begin{equation}
\label{est m1}
m.1  \leq C \cE(t)^{1/2} \cD(t).
\end{equation}
We split $m.2$ into
\begin{align*}
m.2 & \le C \sum_{ |\sigma|\leq s } \left[\sum_{\tau<\sigma}  \left| ( \partial^{\sigma -\tau}\alpha (\theta)  \partial^\tau (|\nabla m|^2 m ) | \partial^\sigma \Delta m ) \right| + \sum_{\tau\leq \sigma} \left| (  \alpha (\theta)  \partial^{\sigma - \tau
} (\nabla m) \cdot \partial^\tau (\nabla m) m ) | \partial^\sigma \Delta m ) \right| \right. \\
& \quad \left.+  \sum_{\substack{\tau_3>0, \\ \tau_1+ \tau_2+\tau_3 =\sigma}} \left| (  \alpha (\theta)  \partial^{\tau_1
} (\nabla m) \cdot \partial^{\tau_2} (\nabla m) \partial^{\tau_3} m ) | \partial^\sigma \Delta m ) \right| \right]=: m.2.1+ m.2.2 + m.2.3.
\end{align*} 
From Lemmas~\ref{Lem: Moser} and \ref{Lem: Nemyskii}, we infer that
\begin{align*}
m.2.1 &\leq C \|\nabla \alpha (\theta) \|_{H^{s-1}} \| |\nabla m|^2 m\|_{H^s} \| \Delta m \|_{H^s} \\
&\leq K(\|\theta\|_\infty ) ( \||\nabla m |^2 \|_{H^s}   +  \||\nabla m |^2 \|_\infty \|\nabla^s m \|_{L_2} ) \| \Delta m \|_{H^s} \| \nabla \theta\|_{H^s}\\
& \leq K(\|\theta\|_\infty )(\cE(t) + \cE(t)^{3/2}) \cD(t).
\end{align*}
It follows from \cite[Equations~(5.32) and (5.33)]{JiangLiuLuo23} that
\begin{align*}
m.2.2+m.2.3 \leq K(\|\theta\|_\infty ) \| \nabla m\|_{H^s} \| \Delta m \|_{H^s}^2 \leq  K(\|\theta\|_\infty ) \cE(t)^{1/2} \cD(t).
\end{align*}
In sum
\begin{equation}
\label{est m2}
m.2 \leq K(\|\theta\|_\infty )(\cE(t)^{1/2} + \cE(t)^{3/2}) \cD(t).
\end{equation}
Employing Lemma~\ref{Lem: Moser} and \eqref{Sobolev embedding}, we get
\begin{equation}
\label{est m3}
\begin{split}
m.3 &= - \sum_{ 0<|\sigma|\leq s }\left[\sum_{0< \tau<\sigma} 
\begin{pmatrix}
\sigma \\
\tau
\end{pmatrix}	
(  \partial^{\sigma -\tau}  ( \beta(\theta) m) \times \Delta \partial^\tau m   | \partial^\sigma \Delta m ) + (  \partial^\sigma ( \beta(\theta) m) \times \Delta   m   | \partial^\sigma \Delta m ) \right] \\
&\leq C \sum_{ 0<|\sigma|\leq s } \left[ \sum_{\tau<\sigma} \|\partial^{\sigma -\tau}  ( \beta(\theta) m)\|_{H^1} \|\Delta \partial^\tau m\|_{H^1} + \|\beta(\theta) m\|_{H^s} \|\Delta m\|_\infty \right]\| \Delta m \|_{H^s} \\
&\leq K(\|\theta\|_\infty )   \cE(t)^{1/2} \cD(t).
\end{split}
\end{equation}
Applying Lemma~\ref{Lem: Commutator} yields
\begin{equation}
\label{est m4}
m.4 \leq \sum_{ |\sigma|\leq s }  \|\cR^\sigma_\alpha \Delta m \|_{L_2}  \| \Delta m \|_{H^s} \leq K(\|\theta\|_\infty ) \| \nabla \theta\|_{H^{s-1}}\| \Delta m \|_{H^s}^2 \leq 
K(\|\theta\|_\infty ) \cE(t)^{1/2} \cD(t).
\end{equation}
To sum up, we end up with 
\begin{equation}
\label{Higher order est m-eq final}
\begin{split} 
\frac{1}{2} \frac{d}{dt} \|  \nabla  m\|_{H^s }^2 + \underline{\alpha} \| \Delta  m \|_{H^s}^2  \leq K(\|\theta\|_\infty ) (\cE(t)^{1/2} + \cE(t)^{3/2}) \cD(t).
\end{split}
\end{equation}

\subsection{Estimates of $M$-equation}
Taking $\partial^\sigma \partial_j $ with $  |\sigma|\leq s-2$  on both sides of \eqref{eq-M}, multiplying by $\partial^\sigma \partial_j M$, and then integrating over $ \bR^N$ yields
\begin{equation}
\label{higher order M}
\begin{split}
&\quad \frac{1}{2} \frac{d}{dt} \|\nabla M \|_{H^{s-2}}^2  + \underline{\nu}  \| \nabla^2 M \|_{H^{s-2}}^2  \\
& \leq      \sum_{|\sigma |\leq s-2} \left[ ( \partial^\sigma \partial_j (  \nu (\theta) \nabla g(G)) | \partial^\sigma \partial_j M) - ( \partial^\sigma \partial_j ( \nu (\theta) \nabla (\nabla \cdot G) ) | \partial^\sigma \partial_j M) \right. \\
& \quad \left.  -    ( \partial^\sigma \partial_j  ( \nu (\theta) \nabla(\nabla \cdot (\nabla m \odot \nabla m )) ) |  \partial^\sigma \partial_j  M ) 
+ ( \partial^\sigma \partial_j  (\nu'(\theta) \nu (\theta) (\nabla u) |\nabla u|^2)| \partial^\sigma \partial_j M)   \right. \\
& \quad \left. +  ( \partial^\sigma \partial_j  (\nu'(\theta) \alpha (\theta) (\nabla u) |\Delta m + |\nabla m|^2 m|^2)| \partial^\sigma \partial_j M)  - ( \partial^\sigma \partial_j (u\cdot \nabla M ) | \partial^\sigma \partial_j M) \right.\\
& \quad \left. - (  \partial^\sigma \partial_j ( \nu (\theta) \curl\curl ( \nabla \nu (\theta) \otimes u))|  \partial^\sigma \partial_j M)
+(  \partial^\sigma \partial_j  (\nu'(\theta) \nabla u \nabla \cdot (\kappa(\theta)  \nabla \theta))|  \partial^\sigma \partial_j M) \right. \\
 & \quad - ( \partial^\sigma \partial_j ((\nabla u ) M ) | \partial^\sigma \partial_j M)    +  ( \partial^\sigma \partial_j \nabla u | \partial^\sigma \partial_j M)  -  ( \partial^\sigma \partial_j ( \nu (\theta) \nabla^2 \pi) | \partial^\sigma \partial_j M) \\
 & \quad \left.  - ( \cR^\sigma_\nu \Delta M | \partial^\sigma \Delta M )   \right] \\
 &=: M.1 + \cdots + M.12.
\end{split}
\end{equation}
First, integration by parts yields
\begin{align*}
 M.1= -(    \nu (\theta) \nabla g(G)  | \Delta M) - \sum_{0< |\sigma | \leq s-2 } ( \partial^\sigma  (  \nu (\theta) \nabla g(G)) | \partial^\sigma \Delta M)  =:M.1.1+M.1.2.
\end{align*}
In view of \eqref{est of gG}, we have
\begin{align*}
 M.1.1  \leq K(\|\theta\|_\infty) \|\nabla g(G) \|_{H^{s-2}}  \|\nabla^2 M\|_{H^{s-2}} \leq  K (\| \theta   \|_\infty  )\sum_{k=2}^\infty k(k+1) C_s^k   \|G\|_{H^s}^{k-1} (1+ \cE(t)^{1/2}) \cD(t).
\end{align*}
Note that $M.1.2 $ is nonempty only when $s\geq 3$. 
\begin{align*}
 M.1.2  & \leq | (     \nu (\theta) \partial^\sigma\nabla g(G)  | \partial^\sigma \Delta M)| +  C \sum_{0< |\sigma | \leq s-2 } \sum_{\tau < \sigma  } |( \partial^{\sigma -\tau}  (  \nu (\theta) )\partial^\tau \nabla g(G)  | \partial^\sigma \Delta M)| \\
&\leq K(\|\theta\|_\infty) \|\nabla g(G) \|_{H^{s-2}}  \|\nabla^2 M\|_{H^{s-2}} +   C \| \nabla \nu (\theta) \|_{H^{s-2}} \| \nabla g(G)\|_{H^{s-2}} \|\nabla^2 M\|_{H^{s-2}} \\
&\leq  K (\| \theta   \|_\infty  )\sum_{k=2}^\infty k(k+1) C_s^k   \|G\|_{H^s}^{k-1} (1+ \cE(t)^{1/2}+ \cE(t)) \cD(t).
\end{align*}
In sum,
\begin{equation}
\label{est M1}
 M.1  \leq  K (\| \theta   \|_\infty  )\sum_{k=2}^\infty k(k+1) C_s^k   \|G\|_{H^s}^{k-1} (1+ \cE(t)^{1/2}+ \cE(t)) \cD(t).
\end{equation}
Next, by the incompressibility condition $\nabla \cdot u=0$ and \eqref{G-relation}
\begin{align*}
M.2 & =\sum_{  |\sigma | \leq s-2 } \left[ ( \partial^\sigma \partial_l (\nu(\theta) \partial_i \partial_k G_{kj} ) | \partial^\sigma \partial_l (\nu (\theta) \partial_i u^j ) ) - ( \partial^\sigma \partial_l (\nu(\theta) \partial_i \partial_k G_{kj}) | \partial^\sigma \partial_l G_{ji} )  \right] \\
&= \sum_{  |\sigma | \leq s-2 } \left[ ( \partial^\sigma \partial_l (\nu(\theta) \partial_i \partial_j \tr G ) | \partial^\sigma \partial_l (\nu (\theta) \partial_i u^j ) )  - ( \partial^\sigma \partial_l (\nu(\theta) \partial_i \partial_j \tr G ) | \partial^\sigma \partial_l G_{ji} ) \right]\\
& = - \sum_{  |\sigma | \leq s-2 } \left[ ( \partial^\sigma \partial_l (\partial_j \nu(\theta) \partial_i   \tr G ) | \partial^\sigma \partial_l (\nu (\theta) \partial_i u^j ) ) 
+( \partial^\sigma \partial_l ( \nu(\theta) \partial_i   \tr G ) | \partial^\sigma \partial_l (\partial_j \nu (\theta) \partial_i u^j ) ) \right.\\ 
&\quad  \left. + ( \partial^\sigma \partial_l (\nu(\theta) \partial_i \partial_j \tr G ) | \partial^\sigma \partial_l G_{ji} ) \right]=: M.2.1+M.2.2+M.2.3.
\end{align*}
We will estimate these terms one by one as follows:
\begin{align*}
 M.2.1  & \leq\sum_{  |\sigma | \leq s-2 } \Big[  | ( \partial_j \nu(\theta) \partial^\sigma \partial_l  \partial_i   \tr G | \partial^\sigma \partial_l (\nu (\theta) \partial_i u^j ) )|  \\
 &\quad  + C \sum_{\tau<\sigma^l} |(  ( \partial^{\sigma^l - \tau} \partial_j \nu(\theta)) \partial^\tau \partial_i   \tr G | \partial^\sigma \partial_l (\nu (\theta) \partial_i u^j ) )| \Big] \\
&\leq C  \| \nabla \nu (\theta) \|_{H^1} \| \nu (\theta) \nabla u \|_{H^s} \| \nabla^2 \tr G\|_{H^{s-2}} + C \| \nabla \nu (\theta) \|_{H^{s-1}} \| \nu (\theta) \nabla u \|_{H^s} \|  \tr G\|_{H^s} \\
&\leq K(\|\theta\|_\infty ) ( \|G \|_{H^s} + \| \theta-\theta_*\|_{H^s}+ \|G \|_{H^s}   \| \theta-\theta_*\|_{H^s} + \| \theta-\theta_*\|_{H^s}^2)\\
&\quad \times ( \| \nabla^2 G\|_{H^{s-2}} + \|\nabla \theta \|_{H^s} ) \|\nabla u \|_{H^s} \\
&\leq K(\|\theta\|_\infty ) (\cE(t)^{1/2} + \cE(t)^{3/2}  ) \cD(t) ,
\end{align*}
where we have used Lemma~\ref{Lem: Moser} and \eqref{Sobolev embedding},
and
\begin{align*}
 M.2.2  & \leq   \sum_{  |\sigma | \leq s-2 } \left[  |(    \nu(\theta) \partial^\sigma \partial_l\partial_i   \tr G   | \partial^\sigma \partial_l (\partial_j \nu (\theta) \partial_i u^j ) )| \right. \\
 &\quad \left. + \, C  \sum_{\tau<\sigma^l}  |( \partial^{\sigma^l -\tau} \nu(\theta) \partial^\tau \partial_i   \tr G   | \partial^\sigma \partial_l (\partial_j \nu (\theta) \partial_i u^j ) )  | \right]\\
&\leq K(\|\theta\|_\infty) \| \nabla^2 \tr G\|_{H^{s-2}} \|  \partial_j \nu (\theta) \partial_i u^j \|_{H^{s-1}} + C \|\nabla  	\nu (\theta) \|_{H^{s-1}} \|G\|_{H^s} \|   \partial_j \nu (\theta) \partial_i u^j \|_{H^{s-1}}  \\
&\leq K(\|\theta\|_\infty)	( \cE(t)^{1/2} + \cE(t)  ) \cD(t) 
\end{align*}
in virtue of \eqref{transport-est}.
To estimate $M.2.3$, we observe that
\begin{align*}
M.2.3 & = - \sum_{  |\sigma | \leq s-2 } \left[ ( \partial^\sigma \partial_l (\partial_j  \nu(\theta) \partial_i \tr G ) | \partial^\sigma \partial_l G_{ji} ) + ( \partial^\sigma \partial_l (  \nu(\theta) \partial_i \tr G ) | \partial^\sigma \partial_l \partial_j G_{ji} ) \right] \\
&=  \sum_{  |\sigma | \leq s-2 } \left[ ( \partial^\sigma   (\partial_j  \nu(\theta) \partial_i \tr G ) | \partial^\sigma \Delta G_{ji} ) -  ( \partial^\sigma \partial_l (  \nu(\theta) \partial_i \tr G ) | \partial^\sigma \partial_l \partial_i \tr G  ) \right] =: M.2.3.1 + M.2.3.2.
\end{align*}
Direct computations show
\begin{align*}
 M.2.3.1  \leq K(\|\theta\|_\infty)  \| G\|_{H^s} \| \nabla \theta\|_{H^s} \|\nabla^2 G \|_{H^{s-2}} \leq  K(\|\theta\|_\infty)	( \cE(t)^{1/2} + \cE(t)  ) \cD(t) 
\end{align*}
and
\begin{align*}
 M.2.3.2  & \le    \sum_{  |\sigma | \leq s-2 } \left[  \left| (  \nu(\theta) \partial^\sigma   \partial_l \partial_i  \tr G   | \partial^\sigma \partial_l \partial_i \tr G  )  \right|
 + C \sum_{\tau< \sigma^l} \left|  ( \partial^{\sigma^l -\tau} (  \nu(\theta) ) \partial^\tau \partial_i \tr G   | \partial^\sigma \partial_l \partial_i \tr G  ) \right| \right] \\
 & \leq  \overline{\nu} \| \nabla^2 \tr G\|_{H^{s-2}}^2  + C \| \nabla \nu (\theta) \|_{H^{s-1}} \| G\|_{H^s} \| \nabla^2 G \|_{H^{s-2}} \\
 &  \leq  \overline{\nu} \| \nabla^2 \tr G\|_{H^{s-2}}^2  + K(\|\theta\|_\infty ) (\cE(t)^{1/2} + \cE(t) ) \cD(t) .
\end{align*}
According to \cite[Equation~(3.11)]{ChenZhang06} and \eqref{2nd order est of G},
\begin{align*}
\| \nabla^2 \tr G\|_{H^{s-2}}^2   \leq C ( \| G\|_{H^s}^2 + \| G\|_{H^s}^4 ) \|\nabla^2 G \|_{H^{s-2}}^2  
 \leq K(\|\theta\|_\infty ) (\cE(t)   +   \cE(t)^3 ) \cD(t) .
\end{align*}
We note in passing that the condition ${\rm det} (G+  I_N)=1$ is essential for deriving
the estimate 
$$ \| \nabla^2 \tr G\|_{H^{s-2}}^2   \leq C ( \| G\|_{H^s}^2 + \| G\|_{H^s}^4 ) \|\nabla^2 G \|_{H^{s-2}}^2.$$ 
In sum,
\begin{equation}
\label{est M2}
 M.2  \leq  K(\|\theta\|_\infty ) (\cE(t)^{1/2}   + \cE(t)^3 ) \cD(t)  .
\end{equation}
Via a similar computation, we have
\begin{equation}
\label{est M3}
\begin{split}
M.3 & =  \sum_{  |\sigma | \leq s-2 }    ( \partial^\sigma   ( \nu (\theta) \nabla(\nabla \cdot (\nabla m \odot \nabla m )) ) |  \partial^\sigma \Delta  M )  \\
&\leq    \sum_{  |\sigma | \leq s-2 } \Big[ \left|(  \nu (\theta) \partial^\sigma   (\nabla(\nabla \cdot (\nabla m \odot \nabla m )) ) |  \partial^\sigma \Delta  M ) \right| \\
&\quad   + C\sum_{\tau < \sigma} \left| ( \partial^{\sigma -\tau }   ( \nu (\theta)) \partial^\tau ( \nabla(\nabla \cdot (\nabla m \odot \nabla m )) ) |  \partial^\sigma \Delta  M ) \right| \Big] \\
&\leq K(\|\theta\|_\infty ) ( 1+ \| \theta- \theta_*\|_{H^s} ) \|\nabla m \|_{H^s} \| \Delta m \|_{H^s}  \| \nabla^2 M\|_{H^{s-2}}\\
& \leq  K(\|\theta\|_\infty ) (\cE(t)^{1/2} + \cE(t)   ) \cD(t) 
\end{split}
\end{equation}
In $M.4$, when $\sigma=0$, it follows from \eqref{Sobolev embedding} that
\begin{align*}
 & -( \   (\nu'(\theta) \nu (\theta) (\nabla u) |\nabla u|^2)|   \Delta M)  \leq K(\|\theta\|_\infty )\| |\nabla u |^3 \|_{L_2} \| \nabla^2 M\|_{H^{s-2}}\\
&\leq  K(\|\theta\|_\infty )\| \nabla u \|_{L_6}^3 \| \nabla^2 M\|_{H^{s-2}} \leq  K(\|\theta\|_\infty ) \| \nabla u \|_{H^1}^3  \| \nabla^2 M\|_{H^{s-2}} \leq  K(\|\theta\|_\infty ) \cE(t) \cD(t);
\end{align*}
when $|\sigma|>0$ (in this case $s\geq 3$), it follows from Lemma~\ref{Lem: Moser} that
\begin{align*}
 & \quad -\sum_{ 0< |\sigma | \leq s-2 }   ( \partial^\sigma   (\nu'(\theta) \nu (\theta) (\nabla u) |\nabla u|^2)| \partial^\sigma \Delta M)   \\
& \le   \sum_{  0<|\sigma | \leq s-2 } \left[ \left| ( \nu'(\theta) \nu (\theta) \partial^\sigma   ((\nabla u) |\nabla u|^2)| \partial^\sigma \Delta M)\right| C  + \sum_{\tau <\sigma} \left| ( \partial^{\sigma   - \tau} (\nu'(\theta) \nu (\theta)) \partial^\tau ( (\nabla u) |\nabla u|^2)| \partial^\sigma \Delta M) \right| \right] \\
&\leq  K(\|\theta\|_\infty ) ( 1+  \|\nabla( \nu'(\theta) \nu (\theta)) \|_{H^k}  ) \| |\nabla u|^3 \|_{H^{s-2}} \| \nabla^2 M\|_{H^{s-2}}\\
&\leq  K(\|\theta\|_\infty ) ( 1+  \|\nabla  \theta \|_{H^k}  )    ( \|\nabla u \|_\infty^2  \| \nabla u \|_{H^{s-2}} + \|\nabla u\|_\infty \| |\nabla u|^2 \|_{H^{s-2}} ) \| \nabla^2 M\|_{H^{s-2}} \\
& \leq K(\|\theta\|_\infty ) (\cE(t) + \cE(t)^{3/2} ) \cD(t) .
\end{align*}
Here $k=\max\{1,s-3\}$.
In sum,
\begin{equation}
\label{est M4}
M.4 \leq K(\|\theta\|_\infty ) (\cE(t) + \cE(t)^{3/2} ) \cD(t) .
\end{equation}
In virtue of \eqref{est Delta m sq}, we obtain
\begin{equation}
\label{est M5}
\begin{split}
M.5 & \le     \sum_{  |\sigma | \leq s-2 } \Big[ C \sum_{ \tau< \sigma^j }
\left| (  \partial^\tau (\nu'(\theta) \alpha (\theta)  \nabla u) \partial^{\sigma^j  -\tau } ( |\Delta m + |\nabla m|^2 m |^2) | \partial^\sigma \partial_j M) \right| \\
&\quad   +  \left| (  \partial^{\sigma} \partial_j ( \nu'(\theta) \alpha (\theta)  \nabla u )   |\Delta m + |\nabla m|^2 m |^2  | \partial^\sigma \partial_j M)  \right| \Big] \\
&\leq C \sum_{  |\sigma | \leq s-2 } \Big[  \sum_{  \tau< \sigma^j } \| \partial^\tau (\nu'(\theta) \alpha (\theta)  \nabla u) \|_{H^1} \||\Delta m + |\nabla m|^2 m |^2\|_{H^s} \|M\|_{H^{s-1}}     \\
& \quad  + \| \partial^\sigma \partial_j (\nu'(\theta) \alpha (\theta)  \nabla u) \|_{L_2} \||\Delta m + |\nabla m |^2 m|^2\|_\infty \|M\|_{H^{s-1}}    \Big] \\
&\leq  K(\|\theta \|_\infty )  (1 + \| \theta-\theta_*\|_{H^s}^2)\|    u   \|_{H^s}  \|M\|_{H^{s-1}}  (1+    \|\nabla m \|_{H^s} ) \| \Delta m \|_{H_s}^2  \\
& \leq K(\|\theta\|_\infty ) (  \cE(t)      +  \cE(t)^{5/2}   ) \cD(t).  
\end{split}
\end{equation}
In the above computations, we have used the  fact that for any $|\tau |\leq s-1$
\begin{align*}
&\quad \| \partial^\tau (\nu'(\theta) \alpha (\theta)  \nabla u)  \|_{L_2} \\
& \leq \|\nu'(\theta) \alpha (\theta)\|_\infty \|  \partial^\tau \nabla u \|_{L_2} + \sum_{\substack{\tau_3<\tau, \\ \tau_1+\tau_2+\tau_3=\tau}} \| \partial^{\tau_1} \nu'(\theta) \|_{H^1} \| \partial^{\tau_2} \alpha(\theta) \|_{H^1} \| \partial^{\tau_3} \nabla u\|_{H^1} \\
&\leq K(\|\theta \|_\infty )  (1 + \| \theta-\theta_*\|_{H^s}^2)\|    u   \|_{H^s}.
\end{align*}
Estimating as in \eqref{est diverg u other terms}, we have
\begin{equation}
\label{est M6}
M.6 \leq C \cE(t)^{1/2} \cD (t) .
\end{equation}
Similarly,
\begin{equation}
\label{est M9}
M.9 \leq   C \|M\|_{H^{s-1}} \|\nabla u \|_{H^s} \| \nabla^2 M \|_{H^{s-2}} \leq  C \cE(t)^{1/2} \cD (t) .
\end{equation}
Based on \eqref{higher order u} and \eqref{est 1 higher order u}, we have
\begin{align*}
M.10 & = \sum_{|\sigma| \leq s-2} \left[ ( \partial^\sigma \partial_j \nabla u |  \partial^\sigma \partial_j (\nu (\theta) \nabla u ))  + (\partial^\sigma \nabla u | \partial^\sigma \Delta G^{\sT} ) \right] \\
& = - \frac{1}{2} \frac{d}{dt} \| \nabla u \|_{H^{s-2}}^2 +  \sum_{|\sigma|\leq s-2} \left[ ( \partial^\sigma \partial_j g(G) | \partial^\sigma \partial_j u)  -  ( \partial^\sigma \partial_j (u\cdot \nabla u ) | \partial^\sigma \partial_j u)    \right. \\
&\quad  \left. -  ( \partial^\sigma \partial_j  (\nabla \cdot (\nabla m \odot \nabla m ))|  \partial^\sigma \partial_j  u )  \right]  .
\end{align*}
The terms in the bracket on the RHS above are bounded by their counterparts in Section~\ref{Section:u est}.
We thus obtain
\begin{equation}
\label{est M10}
M.10 \leq   K (\| \theta   \|_\infty  ) \left[ \sum_{k=2}^\infty k(k+1) C_s^k  \| G\|_{H^s}^{k-1}   (1 + \cE(t)^{1/2} )  + \cE(t)^{1/2} \right] \cD(t) - \frac{1}{2} \frac{d}{dt} \| \nabla u \|_{H^{s-2}}^2.
\end{equation}
Next, we estimate the highest order terms. Note that 
\begin{align*}
M.7 & = \sum_{  |\sigma | \leq s-2 }   (  \partial^\sigma ( \nu (\theta) \curl\curl ( \nabla \nu (\theta) \otimes u))|  \partial^\sigma \Delta M) \\
M.8 & = -\sum_{  |\sigma | \leq s-2 }  (  \partial^\sigma    (\nu'(\theta) \nabla u \nabla \cdot (\kappa(\theta)  \nabla \theta))|  \partial^\sigma \Delta M).
\end{align*}
When $ s=2$, we have
\begin{align*}
M.7 & \leq K(\|\theta\|_\infty) \| \nabla^2 M\|_{H^{s-2}}   \|  \nabla \nu (\theta) \otimes u \|_{H^2}  
 \leq   K(\|\theta\|_\infty)\|u \|_{H^s} \| \nabla^2 M\|_{H^{s-2}} \| \nabla \theta \|_{H^s}  \\ 
 &\leq K(\|\theta\|_\infty) \cE(t)^{1/2} \cD(t) \\
M.8 & \leq K(\|\theta\|_\infty) \|\nabla (\kappa(\theta)  \nabla \theta) \|_{L_2} \| \nabla^2 M\|_{H^{s-2}} \| \nabla u \|_{H^s}  \\
& \leq 
K(\|\theta\|_\infty	)   \| \nabla^2 M\|_{H^{s-2}} \|\nabla u \|_{H^s} ( \| \nabla \kappa (\theta) \|_{H^1} \|\nabla \theta \|_{H^1} + \| \kappa (\theta)\|_\infty \|\nabla  \theta \|_{H^1} )  \\
 &\leq K(\|\theta\|_\infty) ( \cE(t)^{1/2} + \cE(t) ) \cD(t) .
\end{align*}
When $s\geq 3$, we  derive from Lemma~\ref{Lem: Moser} and \eqref{Sobolev embedding} that
\begin{align*}
M.7 & \leq   \sum_{  |\sigma | \leq s-2 } \Big[ \left| (   \nu (\theta) \partial^\sigma (\curl\curl (\nu'(\theta) \nabla \theta \otimes u))|  \partial^\sigma \Delta M)  \right|\\
&\quad   + C \sum_{\tau<\sigma} \left| (  \partial^{\sigma-\tau} ( \nu (\theta)) \partial^\tau (\curl\curl (\nu'(\theta) \nabla \theta \otimes u))|  \partial^\sigma \Delta M) \right| \Big]  \\
&\leq K(\|\theta\|_\infty) \| \nabla^2 M\|_{H^{s-2}} \left( 1 + \| \nabla \theta \|_{H^{s-1}}  \right) \| \nabla \nu (\theta) \otimes u \|_{H^s} \\
&\leq K(\|\theta\|_\infty) \| \nabla^2 M\|_{H^{s-2}} \left( 1 + \| \nabla \theta \|_{H^{s-1}}  \right)    \|u\|_{H^s}  \|\nabla \theta\|_{H^s}  \\
&\leq K(\|\theta\|_\infty) ( \cE(t)^{1/2} + \cE(t) ) \cD(t)  
\end{align*}
and
\begin{align*}
&\quad M.8 \\
& \le    \sum_{  |\sigma | \leq s-2 } \left[  \left|(  \nu'(\theta) \nabla u \partial^\sigma    ( \nabla \cdot (\kappa(\theta)  \nabla \theta))|  \partial^\sigma \Delta M) \right| + C \sum_{\tau<\sigma} \left|(  \partial^{\sigma - \tau}    (\nu'(\theta) \nabla u ) \partial^\tau ( \nabla \cdot (\kappa(\theta)  \nabla \theta))|  \partial^\sigma \Delta M) \right| \right]\\
&\leq   K(\|\theta\|_\infty) \| \nabla^2 M\|_{H^{s-2}} ( \|\nabla u\|_\infty + \| \nu'(\theta) \nabla u \|_{H^{s-1}} ) \| \kappa (\theta) \nabla \theta \|_{H^{s-1}} \\
&\leq K(\|\theta\|_\infty) \| \nabla^2 M\|_{H^{s-2}} \| \nabla u\|_{H^{s-1}} (1+  \|\nabla^{s-1} \theta \|_{L_2})   ( \| \nabla \theta \|_{s-1} + \|\nabla \theta \|_\infty \| \nabla^{s-1} \theta\|_{L_2}  ) \\
&\leq K(\|\theta\|_\infty) ( \cE(t)^{1/2}   + \cE(t)^{3/2} ) \cD(t)  .
\end{align*}
In sum,
\begin{equation}
\label{est M78}
M.7+M.8 \leq K(\|\theta\|_\infty) ( \cE(t)^{1/2}   + \cE(t)^{3/2} ) \cD(t)  .
\end{equation}
To estimate the term containing  pressure,  we observe that
\begin{align*}
M.11 &= - \sum_{  |\sigma | \leq s-2 } \left[ \sum_{\tau <\sigma^l} 
\begin{pmatrix}
\sigma^l \\
\tau
\end{pmatrix}
 ( \partial^{\sigma^l - \tau} ( \nu (\theta) ) \partial^\tau \nabla^2 \pi  | \partial^\sigma \partial_l M)  +   (    \nu (\theta) \partial^\sigma \partial_l \nabla^2 \pi) | \partial^\sigma \partial_l M) \right]\\
 &=:M.11.1+M.11.2.
\end{align*}
It follows from Lemma~\ref{Lem: Lap pressure est}  that
\begin{align*}
M.11.1 &\leq  C \sum_{ 0< |\sigma | \leq s-2 }  \sum_{\tau <\sigma^l} 
\| \partial^{\sigma^l - \tau}   \nu (\theta)\|_{H^1} \| \partial^\tau \nabla^2 \pi \|_{L_2}\| \nabla^2 M \|_{H^{s-2}}\\
&\leq K (\| \theta   \|_\infty  ) \| \nabla \theta \|_{H^{s-1}} \| \nabla^2 \pi \|_{H^{s-2}} \| \nabla^2 M \|_{H^{s-2}}\\
&\leq  K (\| \theta   \|_\infty  ) \left(  \cE(t)^{1/2}  +\cE(t) \right) \left(1 + \sum_{k=2}^\infty k(k+1) C_s^k   \|G\|_{H^s}^{k-1} \right) \cD(t) .
\end{align*}
By the incompressibility condition $\nabla \cdot u =0$, 
\begin{align*}
M.11.2 & = \sum_{  |\sigma | \leq s-2 } \left[ (    \nu (\theta) \partial^\sigma \partial_l \partial_i \partial_j \pi) | \partial^\sigma \partial_l G_{ji})   -   (    \nu (\theta) \partial^\sigma \partial_l \partial_i \partial_j \pi) | \partial^\sigma \partial_l (\nu (\theta) \partial_i u^j)) \right] \\
& =  \sum_{  |\sigma | \leq s-2 } \left[  (    \partial_j (\nu (\theta)) \partial^\sigma \partial_l \partial_i  \pi) | \partial^\sigma \partial_l (\nu (\theta) \partial_i u^j)) +  (    \nu (\theta) \partial^\sigma \partial_l \partial_i  \pi) | \partial^\sigma \partial_l ( \partial_j (\nu (\theta)) \partial_i u^j)) \right.\\
&\quad \left. - (    \partial_j ( \nu (\theta)) \partial^\sigma \partial_l \partial_i \pi) | \partial^\sigma \partial_l G_{ji})  - (    \nu (\theta) \partial^\sigma \partial_l \partial_i  \pi) | \partial^\sigma \partial_l \partial_j G_{ji}) \right] \\
&=: M.11.21+ \cdots + M.11.24. 
\end{align*}
By Lemmas~\ref{Lem: Moser}, \ref{Lem: Nemyskii} and  \ref{Lem: Lap pressure est}, we have
\begin{align*}
M.11.21 &\leq C \| \nabla    \nu (\theta) \|_{H^1} \|\nabla^2 \pi \|_{H^{s-2}} \| \nu (\theta) \nabla u  \|_{H^{s }} \\
&\leq K (\| \theta   \|_\infty  ) \| \theta - \theta_* \|_{H^s}\|\nabla^2 \pi \|_{H^{s-2}} ( \| \nabla u \|_{H^s} + \|\nabla u \|_\infty \| \nabla^{s } \theta \|_{L_2} ) \\
& \leq K (\| \theta   \|_\infty  ) \left(  \cE(t)^{1/2}    +\cE(t)^{3/2}  \right) \left(1 + \sum_{k=2}^\infty k(k+1) C_s^k   \|G\|_{H^s}^{k-1} \right) \cD(t) .
\end{align*}
Similarly,  
\begin{align*}
M.11.22 &\leq K (\| \theta   \|_\infty  )\|\nabla^2 \pi \|_{H^{s-2}}  \| (\nabla u) \nabla \nu (\theta)  \|_{H^{s-1}} \\
& \leq  K(\| \theta   \|_\infty  )\|\nabla^2 \pi \|_{H^{s-2}}   (\|\nabla \theta \|_\infty \| \nabla^s u \|_{L_2} +  \|\nabla u\|_\infty \| \nabla^s \theta \|_{L_2}) \\
& \leq K (\| \theta   \|_\infty  ) \left(  \cE(t)^{1/2}  +\cE(t)  \right) \left(1 + \sum_{k=2}^\infty k(k+1) C_s^k   \|G\|_{H^s}^{k-1} \right) \cD(t)  
\end{align*}
and
\begin{align*}
M.11.23 & = K (\| \theta   \|_\infty  )  \| \nabla    \theta\|_{H^s} \|\nabla^2 \pi \|_{H^{s-2}} \| G \|_{H^s}\\
& \leq K (\| \theta   \|_\infty  ) \left(  \cE(t)^{1/2}  +\cE(t)  \right) \left(1 + \sum_{k=2}^\infty k(k+1) C_s^k   \|G\|_{H^s}^{k-1} \right) \cD(t) .
\end{align*}
It follows from \eqref{G-relation} that
\begin{align*}
M.11.24 & = -\sum_{  |\sigma | \leq s-2 }  (    \nu (\theta) \partial^\sigma \partial_l \partial_i  \pi) | \partial^\sigma \partial_l \partial_i \tr G ) \\
&= \sum_{  |\sigma | \leq s-2 } \left[ (    \partial_i (\nu (\theta)) \partial^\sigma \partial_l \partial_i  \pi) | \partial^\sigma \partial_l  \tr G ) + (    \nu (\theta) \partial^\sigma \partial_l \Delta  \pi | \partial^\sigma \partial_l   \tr G ) \right]= :M.11.241 + M.11.242.
\end{align*}
The term $M.11.241$ can be estimated in a similar way to $M.11.23$:
\begin{align*}
M.11.241 \leq K (\| \theta   \|_\infty  ) \left(  \cE(t)^{1/2}  +\cE(t)  \right) \left(1 + \sum_{k=2}^\infty k(k+1) C_s^k   \|G\|_{H^s}^{k-1} \right) \cD(t) .
\end{align*}
It follows from \eqref{Lap of pressure} that
\begin{align*}
M.11.242 &= \sum_{  |\sigma | \leq s-2 } [( \nu (\theta) \partial^\sigma \partial_l ( \nabla \cdot ( \nabla \cdot (\nu (\theta) \nabla u )) )  | \partial^\sigma \partial_l   \tr G ) 
- ( \nu (\theta) \partial^\sigma \partial_l (\nabla\cdot (u\cdot \nabla u)) | \partial^\sigma \partial_l   \tr G ) \\
&\quad 
 - 2 ( \nu (\theta) \partial^\sigma \partial_l (\Delta \tr G ) | \partial^\sigma \partial_l   \tr G ) 
  + ( \nu (\theta) \partial^\sigma \partial_l ( \nabla \cdot g(G) ) | \partial^\sigma \partial_l   \tr G )\\ 
&\quad  - ( \nu (\theta) \partial^\sigma \partial_l ( \nabla \cdot ( \nabla \cdot ( \nabla m \odot \nabla m ))) | | \partial^\sigma \partial_l   \tr G )]
=: E.1 + \cdots + E.5.
\end{align*}
According to Lemmas~\ref{Lem: Moser} and \ref{Lem: Nemyskii}, \cite[Equation~(3.11)]{ChenZhang06} and \eqref{2nd order est of G}, we have
\begin{align*}
E.1 & = -\sum_{  |\sigma | \leq s-2 }  [ ( \partial_l  (\nu (\theta)) \partial^\sigma  ( \nabla \cdot ( \nabla \cdot (\nu (\theta) \nabla u )) )  | \partial^\sigma \partial_l   \tr G )
+ (  \nu (\theta)  \partial^\sigma  ( \nabla \cdot ( \nabla \cdot (\nu (\theta) \nabla u )) )  | \partial^\sigma \Delta   \tr G ) ]\\
&\leq K (\| \theta   \|_\infty  ) \| \nabla \theta\|_\infty  \| \nu (\theta) \nabla u\|_{H^s} \| G\|_{H^s} + 	K (\| \theta   \|_\infty  ) \| \nu (\theta) \nabla u\|_{H^s} \| \nabla^2 \tr G\|_{H^{s-2}}  \\
&\leq K (\| \theta   \|_\infty  ) (1+ \| \theta -\theta_*\|_{H^s}) \| G\|_{H^s}   \| \nabla \theta\|_{H^s}  \| \nabla u\|_{H^s}  \\
&\quad +  K(\| \theta   \|_\infty  ) (1+ \| \theta -\theta_*\|_{H^s}) ( \| G\|_{H^s} +\| G\|_{H^s}^2   )\| \nabla u\|_{H^s} \| \nabla^2   G\|_{H^{s-2}} \\
&\leq K (\| \theta   \|_\infty  ) \left(  \cE(t)^{1/2}    +\cE(t)^2  \right) \cD(t).
\end{align*}
Similarly, 
\begin{align*}
E.2 & = ( \nu (\theta) \partial^\sigma  	(\nabla\cdot (u\cdot \nabla u)) | \partial^\sigma \Delta   \tr G )  + ( \partial_l \nu (\theta) \partial^\sigma  	(\nabla\cdot (u\cdot \nabla u)) | \partial^\sigma \partial_l   \tr G ) \\
&=  \sum_{\tau\leq \sigma }
\begin{pmatrix}
\sigma \\
\tau
\end{pmatrix}
\left[ ( \nu (\theta) (\nabla\cdot (\partial^{\sigma  -\tau}  u\cdot \nabla \partial^\tau u)) | \partial^\sigma \Delta \tr G ) + ( \partial_l \nu (\theta) (\nabla\cdot (\partial^{\sigma  -\tau}  u\cdot \nabla \partial^\tau u)) | \partial^\sigma \partial_l   \tr G ) \right] \\
& \leq K (\| \theta   \|_\infty  ) \| u \|_{H^s}\| \nabla u \|_{H^s} \left(    \| \nabla^2   G\|_{H^{s-2}} + \|G\|_{H^s} \|\nabla \theta \|_{H^s}     \right) \\
& \leq K (\| \theta   \|_\infty  ) \left(   \cE(t)^{1/2} +\cE(t)    \right) \cD(t) 
\end{align*}
and
\begin{align*}
E.3&= 2 ( \nu (\theta) \partial^\sigma    \Delta \tr G  | \partial^\sigma \Delta   \tr G )  + 2 (\partial_l \nu (\theta) \partial^\sigma    \Delta \tr G  | \partial^\sigma \partial_l   \tr G )\\
&  \leq 2 \overline{\nu} \|\nabla^2 \tr G\|_{H^{s-2}}^2 + K (\| \theta   \|_\infty  )\|\nabla   \theta\|_{H^1} \| \partial^\sigma \nabla   \tr G  \|_{H^1} \|\nabla^2 \tr G\|_{H^{s-2}} \\ 
&\leq C (\|G\|_{H^s}^2 + \|G\|_{H^s}^4 ) \| \nabla^2 G \|_{H^{s-2}}^2 +  K (\| \theta   \|_\infty  )\|   G  \|_{H^s} \|\nabla   \theta\|_{H^s}  \|\nabla^2   G\|_{H^{s-2}} \\ 
& \leq K (\| \theta   \|_\infty  ) \left(   \cE(t)^{1/2}+ \cE(t)^3  \right) \cD(t) 
\end{align*}
and in view of \eqref{est of gG}
\begin{align*}
E.4 & =-( \nu (\theta) \partial^\sigma   ( \nabla \cdot g(G) ) | \partial^\sigma \Delta  \tr G ) -( \partial_l \nu (\theta) \partial^\sigma   ( \nabla \cdot g(G) ) | \partial^\sigma \partial_l   \tr G )\\
& \leq K (\| \theta   \|_\infty  )  \| \nabla \cdot g(G) \|_{H^{s-2}} \|\nabla^2 \tr G\|_{H^{s-2}} + K (\| \theta   \|_\infty  )\|\nabla   \theta\|_{H^1} \| \partial^\sigma \nabla   \tr G  \|_{H^1} \| \nabla \cdot g(G) \|_{H^{s-2}}\\   
&\leq K (\| \theta   \|_\infty  )\sum_{k=2}^\infty k(k+1) C_s^k   \|G\|_{H^s}^{k-2} (  \cE(t)^{1/2}  +  \cE(t)^{3/2} ) \cD(t) 
\end{align*}
and
\begin{align*}
E.5 & =( \nu (\theta) \partial^\sigma   ( \nabla \cdot ( \nabla \cdot ( \nabla m \odot \nabla m ))) |   \partial^\sigma \Delta   \tr G )  +( \partial_l \nu (\theta) \partial^\sigma   ( \nabla \cdot ( \nabla \cdot ( \nabla m \odot \nabla m ))) |   \partial^\sigma  \partial_l    \tr G )  \\
&\leq K (\| \theta   \|_\infty  ) \| \nabla m \odot \nabla m  \|_{H^s} \|\nabla^2 \tr G\|_{H^{s-2}}  + K (\| \theta   \|_\infty  )\|\nabla   \theta\|_{H^1} \| \partial^\sigma \nabla   \tr G  \|_{H^1} \| \nabla m \odot \nabla m  \|_{H^s}\\
&\leq K (\| \theta   \|_\infty  ) \|\nabla m\|_\infty \| \nabla^{s+1} m \|_{L_2} (\|\nabla^2  G\|_{H^{s-2}} + \|G\|_{H^s} \|\nabla \theta \|_{H^s} ) \\
&\leq K (\| \theta   \|_\infty  )\left(   \cE(t)^{1/2} +  \cE(t) \right) \cD(t).
\end{align*}
In sum,
\begin{equation}
\label{est M11}
M.11\leq K (\| \theta   \|_\infty  ) \left[  \sum_{k=2}^\infty k(k+1) C_s^k   \|G\|_{H^s}^{k-2} (  \cE(t)^{1/2}  +  \cE(t)^{2}) +  \cE(t)^{1/2}+ \cE(t)^3  \right] \cD(t) 
\end{equation}
Lastly, in virtue of Lemma~\ref{Lem: Commutator}, we have
\begin{align*}
M.12 &=  -  ( \cR^\sigma_\nu \Delta M | \partial^\sigma \Delta M )  \leq K(\|\theta\|_\infty ) \|\nabla \theta\|_{H^{s-1}} \| \Delta M \|_{H^{s-2}}^2 \leq K(\|\theta\|_\infty ) \cE(t)^{1/2} \cD(t) .
\end{align*}
Combining \eqref{est M1}-\eqref{est M11}, we conclude that
\begin{equation}
\label{higher order M final}
\begin{split}
&\frac{1}{2} \frac{d}{dt} \left[  \|\nabla M \|_{H^{s-2}}^2 + \|\nabla u \|_{H^{s-2}}^2  \right]  + \underline{\nu}  \| \nabla^2 M \|_{H^{s-2}}^2  \\
& \leq K (\| \theta   \|_\infty  ) \left[  \sum_{k=2}^\infty k(k+1) C_s^k   \|G\|_{H^s}^{k-2} (  \cE(t)^{1/2}  + \cE(t)^2 ) +  \cE(t)^{1/2}+ \cE(t)^3  \right] \cD(t) .
\end{split}
\end{equation}


\subsection{Control higher order norms}
Gathering \eqref{higher order u and G final}, \eqref{higher order theta-final},   \eqref{Higher order est m-eq final}, and  \eqref{higher order M final}, we derive
\begin{equation}
\label{EE global existence 1}
\begin{split}
&\frac{1}{2} \frac{d}{dt}\left[ \| \nabla u(t)\|_{H^{s-1}}^2 + \|\nabla u(t)\|_{H^{s-2}}^2 +  \|\nabla G (t) \|_{H^{s-1}}^2  + \|\nabla M (t) \|_{H^{s-2}}^2  +    \| \nabla  m(t)\|_{H^s}^2 +  \|\theta (t) -\theta_* \|_{H^s}^2 \right]  \\
& \quad  + \underline{\nu }  \left[ \|\nabla^2 u(t) \|_{H^{s-1}}^2 + \| \nabla^2 M(t) \|_{H^{s-2}}^2  \right] + \underline{\alpha }\| \Delta m(t) \|_{H^s}^2  + \underline{\kappa} \|\nabla \theta \|_{H^s}^2  \\
& \leq  K (\| \theta  \|_\infty  ) \left[  \sum_{k=2}^\infty k(k+1) C_s^k   \|G\|_{H^s}^{k-2} (  \cE(t)^{1/2}  + \cE(t)^2) +  \cE(t)^{1/2}+ \cE(t)^3  \right] \cD(t)
\end{split}
\end{equation}
for all $t\in [0,\widehat{T})$. 
Adding up \eqref{Conserved quantity} and \eqref{EE global existence 1}  yields
\begin{equation}
\label{EE global existence 3}
\begin{split}
&   \frac{1}{2} \frac{d}{dt} \widetilde{\cE}(t)  + \cD(t) \\
&\quad \leq  K (\|  \theta   \|_\infty  ) \left[  \sum_{k=2}^\infty k(k+1) C_s^k   \|G\|_{H^s}^{k-2} (  \cE(t)^{1/2}  + \cE(t)^2) +  \cE(t)^{1/2}+ \cE(t)^3  \right] \cD(t),
\end{split}
\end{equation}
where 
\begin{align*}
\widetilde{\cE}(t) :& =  \| u(t)\|_{H^s}^2 + \|\nabla u(t)\|_{H^{s-2}}^2 + \| H (t)\|_{L_2}^2  + \|\nabla G (t) \|_{H^{s-1}}^2  + \|\nabla M (t) \|_{H^{s-2}}^2 +   
\|\nabla m(t)\|_{L_2}^2  \\
&\quad    +    \|\nabla m(t)\|_{H^s}^2  +  \|\theta (t) - \theta_* \|_{H^s}^2 .
\end{align*}
Fix $\lambda \in \left( 0,  C_s^{-1}   \right) $  
and $\sM>0$ such that whenever $\cE(t) \leq \cE(0) + 2 \sM $,
$$
K(\| \theta(t)\|_{\infty} ) \leq K_0
$$
for some $1\leq K_0<\infty$, where $C_s$ is the constant in \eqref{est of gG-0}. 
With
$
C(\lambda):=C_s^2 \sum_{k=2}^\infty k(k+1)  ( C_s \lambda)^{k-2}
$
we define
\begin{equation}
\begin{aligned}
\eta: &= \min \left\{ \left(\frac{1}{4K_0(C(\lambda) +1 )} \right)^2 ,  \lambda^2 \,  , \cE(0) + 2\sM  \right\} \\
T_\lambda : &= \sup \left\{ T\in [0,\widehat{T}):   \sup_{t\in [0,T]} \cE(t)  < \eta \right\}.
\end{aligned}
\end{equation}
Since 
\begin{equation}
\label{G_H_s}
\| G(t) \|_{H^s}\le \cE(t)^{1/2}\le \lambda<1/C_s, \quad
\text{for } t\in [0, T_\lambda),
\end{equation}
we have
\[
\sum_{k=2}^\infty k(k+1) C_s^k   \|G(t) \|_{H^s}^{k-2} < C_s^2 \sum_{k=2}^\infty k(k+1)  ( C_s   \lambda)^{k-2}=:C(\lambda).
\]
Then for $t\in [0,T_\lambda)$, we have, as $\eta\le 1$, 
\begin{align*}
 K (\|  \theta   \|_\infty  ) \left[  \sum_{k=2}^\infty k(k+1) C_s^k   \|G\|_{H^s}^{k-2} (  \cE(t)^{1/2}  + \cE(t)^2 ) +  \cE(t)^{1/2}+ \cE(t)^3  \right] 
 \le 2 K_0 (C(\lambda)+1)\eta^{1/2}  \leq  \frac{1}{2}.
\end{align*}
Substituting this estimate into \eqref{EE global existence 3} yields
$$
\frac12\frac{d}{dt}\widetilde {\cE} (t) + \cD (t) \le \frac12 \cD(t), \qquad t\in [0, T_\lambda).
$$
Therefore,
\begin{align*}
& \frac{d}{dt} \widetilde{\cE} (t) \leq 0,\qquad  t\in [0,T_\lambda).
\end{align*}
Note that $\nabla M = (\nabla u)\otimes (\nabla  \nu (\theta) )+ \nu (\theta) \nabla^2 u - \nabla G^{\sT}$. \eqref{Sobolev embedding}, Lemmas~\ref{Lem: Moser} and \ref{Lem: Nemyskii} imply
\begin{align*}
\|\nabla M_0\|_{H^{s-2}}^2  \leq   K_1  ( \| u_0 \|_{H^s}^2  + \|\theta_0 - \theta_*\|_{H^{s }}^2+ \| G_0\|_{H^s}^2  ) .
\end{align*}
for some $K_1$  non-decreasing with respect to $\cE(0) $.
Employing  \eqref{G and H},    whenever $t\in [0,T_\lambda)$, it holds
\begin{align*}
& \quad \cE(t) \\
& \leq \|u(t)\|_{H^s}^2 +  2 \bar{\nu}^2 \| \nabla u(t) \|_{H^{s-2}}^2 +   3\|G(t)\|_{L_2}^2  + \|\nabla G(t)\|_{H^{s-1}}^2  +  \|\nabla M(t)\|_{H^{s-2}}^2 + \| \nabla m(t) \|_{L_2}^2 \\
&\quad +  \| \nabla m(t) \|_{H^s}^2 +   \| \theta - \theta_* \|_{H^s}^2   \\
& \leq (1+2\bar{\nu}^2 + 3 \Lambda_0^2) \widetilde{\cE}(t) \\
& \leq (1+2 \bar{\nu}^2 + 3 \Lambda_0^2)\widetilde{\cE}(0) \\
& \leq (1+2 \bar{\nu}^2 + 3 \Lambda_0^2) \left[ (2+K_1) \| u_0\|_{H^s}^2   +   (1+K_1+ \Lambda_0^2)\|  G_0\|_{H^s}^2    +          2 \|\nabla m_0\|_{H^s}^2  +  (1+K_1)\|\theta_0 - \theta_* \|_{H^s}^2 \right] \\
& \leq C_2\left[  \| u_0\|_{H^s}^2     +  \|  F_0 - I_N\|_{H^s}^2    +           \|\nabla m_0\|_{H^s}^2  +  \|\theta_0 - \theta_* \|_{H^s}^2 \right] .
\end{align*}
where $C_2=(1+2 \bar{\nu}^2 + 3 \Lambda_0^2) (2+K_1 +\Lambda_0^2)\Lambda_s^2.$

Choosing $ \| u_0\|_{H^s}^2     +  \|  F_0 - I_N\|_{H^s}^2    +      \|\nabla m_0\|_{H^s}^2  +  \|\theta_0 - \theta_* \|_{H^s}^2 \le C_2 $  sufficiently small, we can always achieve
\begin{align*}
C_2 
\left[  \| u_0\|_{H^s}^2     +  \|  F_0 - I_N\|_{H^s}^2    +     \|\nabla m_0\|_{H^s}^2  +  \|\theta_0 - \theta_* \|_{H^s}^2 \right]  < \frac{\eta}{2}  .
\end{align*}
By the definition of $T_\lambda$, this implies that $T_\lambda=\widehat{T}$.
From \eqref{G_H_s} and the definition of $\widehat{T}$, we further obtain $\widehat{T}=T_*$.
In view of Theorem~\ref{Thm: Local wellposedness}, we thus infer that $T_*= +\infty$. 
This completes the proof of Theorem~\ref{Thm: global wellposedness}.

\appendix

\section{Existence of a Solution to the linearized system}\label{Appendix:existence}

In the appendix, we will give a proof for Proposition~\ref{prop: existence linear}. 
Focus will be given to the $u$ and $\theta$-equations. The $m$-equation can be treated
in a similar way.

\begin{proposition}\label{prop:appendix existence}
Let $s\in \mathbb N$ with $s\ge 2$. 
We consider  the Cauchy problem on $\mathbb R^N$
\begin{equation}\label{appendix:theta}
u_t-\nabla \cdot (a\nabla u)=f \qquad \text{in }(0,T)\times \mathbb R^N,
\qquad
u(0)=u_0,
\end{equation}
where
\begin{equation}
\label{appendix:a cond}
\begin{aligned}
&  a\in L_\infty(J_T\times \bR^N),\qquad  \nabla a \in \bE^{s-1}(J_T) ,   \\
 & 0<c_0 I_N \le a(t,x)\le c_1 I_N \ \ \text{for all }\  (t,x)\in [0,T]\times \mathbb R^N, \\
& f\in L_2((0,T); H^{s-1}),\quad u_0\in H^s.
\end{aligned}
\end{equation}
Then \eqref{appendix:theta} admits a unique solution satisfying
\(
u\in \bE^{s}(J_T) .
\)
\end{proposition}

Before proving this proposition, we first prove an auxiliary lemma on the regularity of the diffusion term under the given assumptions.
\begin{lemma}\label{Appendix_lem}
Under the conditions of Proposition~\ref{prop:appendix existence}, if
$ u\in L_2((0,T);H^{s+1}),$ 
then
\begin{equation*}
    \nabla\cdot(a\nabla u)  \in  L_2((0,T);H^{s-1}(\mathbb{R}^N)).
\end{equation*}
More precisely, we have
\begin{equation}
\label{est_diffuse}
\begin{aligned}
\|\nabla\cdot(a\nabla u)\|_{L_2((0,T);H^{s-1})}
& \le  C\left( \|a\|_{L_\infty} + \|\nabla a\|_{\bE^{s-1}(J_T)}  \right) \|u\|_{L_2((0,T);H^{s+1}).} \\
\end{aligned}
\end{equation}
\end{lemma}
\begin{proof}
Direct computations show that
\begin{equation}
\label{appendix:split}
    \nabla\cdot(a\nabla u)
    =
    a\Delta u+(\nabla a )\nabla u.
\end{equation}
Lemma~\ref{lem: transport-est} gives
\begin{equation}\label{Apppendix:2nd_term}
\| (\nabla a ) \nabla u\|_{H^{s-1}}
\leq
C
\|\nabla a\|_{H^{s-1}}
\|\nabla u\|_{H^s}.
\end{equation}
For the first term in \eqref{appendix:split}, we use the corresponding product
estimate with one factor controlled in $L_\infty$:
\begin{equation}\label{Apppendix:1st_term}
\|a\Delta u\|_{H^{s-1}}
\leq
C
\left(
    \|a\|_{L_\infty}\|\Delta u\|_{H^{s-1}}
    +
    \|\nabla a\|_{H^{s-1}}\|\Delta u\|_{H^{s-1}}
\right).
\end{equation}
Indeed, when $s\geq 3$, \eqref{Apppendix:1st_term} follows from Lemma~\ref{Lem: Moser}(i). 
When $s=2$, \eqref{Sobolev embedding} and H\"older's inequality yield
\begin{align*}
\|a\Delta u\|_{H^{s-1}} &\le \|a\|_\infty \|\Delta u\|_{L_2} + \| \nabla a\|_{L_4} \|\Delta u \|_{L_4} + \|a\|_\infty \| \Delta u \|_{H^1} \\
&\le  C
\left(
    \|a\|_{L_\infty}\|\Delta u\|_{H^{s-1}}
    +
    \|\nabla a\|_{H^{s-1}}\|\Delta u\|_{H^{s-1}}
\right).
\end{align*}
Combining \eqref{Apppendix:2nd_term} and \eqref{Apppendix:1st_term}, we obtain, for almost every
$t\in(0,T)$,
\begin{equation*}
\begin{aligned}
\|\nabla\cdot(a\nabla u)(t)\|_{H^{s-1}}
\leq C\left(
\|a(t)\|_{L_\infty}
\|u(t)\|_{H^{s+1}}
+
\|\nabla a(t)\|_{H^{s-1}}
\|u(t)\|_{H^{s+1}}
\right).
\end{aligned}
\end{equation*}
The claim now follows from 
$$
a\in L_\infty ((0,T)\times \bR^N),\quad
\nabla a\in \bE^{s-1}(J_T) \hookrightarrow L_\infty((0,T), H^{s-1}),
$$
see~\eqref{L-M embedding}, and the assumption on $u$. 
%
\end{proof}

\begin{proof}[Proof of Proposition~\ref{prop:appendix existence}]
We only derive the a priori estimate; existence follows by standard approximation of the coefficients and data.

\medskip
\noindent
By  extending  $a$ from $[0,T]$ to $\bR$ by
\begin{align*}
a(t,x)=
\begin{cases}
a(0,x) , \quad & \text{for } t<0 \\
a(t,x), & \text{for } 0\le t\le T \\
a(T,x), & \text{for } t>T
\end{cases}
\end{align*}
and   mollifying in $(t,x)$, we can find smooth functions $a^\varepsilon$ such that
$$
0<\frac{c_0}{2} I_N\le a^\varepsilon \le 2c_1 I_N,
\qquad
\nabla a^\varepsilon \to \nabla a \quad \text{in } \bE^{s-1}(J_T),
\qquad
 a^\varepsilon \to  a \quad \text{a.e.}.
$$
In virtue of \eqref{L-M embedding}, we have additionally $\nabla a^\varepsilon \to \nabla a$ in $C(J_T;H^{s-1})$.
Likewise, choose smooth $f^\varepsilon\to f$ in $L_2((0,T);H^{s-1})$ and
$u_0^\varepsilon\to u_0$ in $H^s$.
Let $u^\varepsilon \in \bE^{s}(J_T)$ solve
$$
u_t^\varepsilon-\nabla \cdot (a^\varepsilon \nabla u^\varepsilon)=f^\varepsilon,
\qquad
u^\varepsilon(0)=u_0^\varepsilon.
$$
See \cite[Chapters 3 and 6]{PruSim16}.
It suffices to prove estimates that are independent of $\varepsilon$. 
There exists a uniform constant $M>0$ such that
$$
 \sup_{t\in[0,T]}\left( \|\nabla a^\varepsilon(t)\|_{H^{s-1}}^2 + \| a^\varepsilon(t)\|_{\infty}^2 \right) + \int_0^T \|\nabla a^\varepsilon(t)\|_{H^{s}}^2 \, dt \le M.
$$
Following the estimates in Section~\ref{sec.local.theta.est}, we obtain
\begin{align*}
\sup_{t\in[0,T]} \| u^\varepsilon (t) \|_{H^s}^2 +  \int_0^T \|\nabla u^\varepsilon(t)\|_{H^{s}}^2 \, dt \leq C_0 e^{TK(M)} \left( \| u_0 \|_{H^s}^2 + \int_0^T \|f^\varepsilon (t) \|_{H^{s-1}}^2\, dt \right).  
\end{align*}
Combining with Lemma~\ref{Appendix_lem}, this implies the existence of a uniform bound $B$ for
\begin{equation}
\label{appendix unif bdd}
 \|u^\varepsilon\|_{\bE^s(J_T)} + \|u^\varepsilon\|_{L_\infty(H^s)} \leq B.
\end{equation}
\eqref{appendix unif bdd} implies that, after extraction of a subsequence (not relabelled), we can find a function $u$ such that
$$
u^\varepsilon \rightharpoonup u \quad \text{in } L_2((0,T);H^{s+1}),
\qquad
u_t^\varepsilon \rightharpoonup u_t \quad \text{in } L_2((0,T);H^{s-1}),
$$
and weak-$*$ in $L^\infty(0,T;H^s)$.

Let $B_R \subset \mathbb R^N$ be any ball. On $B_R$, we have the following  compact embeddings
$$
H^{s+1}(B_R) \hookrightarrow H^s(B_R) \hookrightarrow H^{s-1}(B_R).
$$
In virtue of \eqref{appendix unif bdd}, the Aubin-Lions lemma, cf. \cite[Proposition III.1.3]{ShowalterBook}, implies 
$$
u^\varepsilon \to u 
\quad \text{in } L_2((0,T);H^s(B_R)).
$$
Since $R$ is arbitrary, 
$$
\nabla u^\varepsilon \to \nabla u 
\quad \text{strongly in } L_{2, \mathrm{loc}}((0,T)\times \mathbb R^N).
$$
By construction, \(|a^\varepsilon|\le C\) and
$$
a^\varepsilon \to a \quad \text{a.e. in } (0,T)\times \mathbb R^N.
$$
Hence
$$
a^\varepsilon \rightharpoonup^\ast a \quad \text{in } L^\infty((0,T)\times \mathbb R^N).
$$
Let $\phi \in C_c^\infty((0,T)\times \mathbb R^N)$. Then
\begin{equation}\label{Appendix:est_1}
\int_0^T \int_{\mathbb R^N} 
a^\varepsilon \nabla u^\varepsilon \cdot \nabla \phi
\,dx\,dt
= \int a^\varepsilon (\nabla u^\varepsilon - \nabla u)\cdot \nabla \phi
+ \int (a^\varepsilon - a)\nabla u \cdot \nabla \phi
+ \int a \nabla u \cdot \nabla \phi.
\end{equation}
We show the first two terms on RHS vanish.
First,
$$
\left|
\int a^\varepsilon (\nabla u^\varepsilon - \nabla u)\cdot \nabla \phi
\right|
\le \|a^\varepsilon\|_{L^\infty}
\|\nabla u^\varepsilon - \nabla u\|_{L_2(K)}
\|\nabla \phi\|_{L_2},
$$
where $K=\operatorname{supp}\phi$.
Since $\nabla u^\varepsilon \to \nabla u$ strongly in $L_2(K)$,
this term tends to $0$.

The second term on RHS of \eqref{Appendix:est_1} can be estimated as follows.
Since $a^\varepsilon \to a$ a.e. and is bounded,
$$
(a^\varepsilon - a)\nabla u \cdot \nabla \phi
\to 0 \quad \text{a.e.}, \quad \text{and}
\quad
|(a^\varepsilon - a)\nabla u \cdot \nabla \phi| \le C |\nabla u||\nabla \phi| \in L^1.
$$
Thus by the dominated convergence theorem,
$$
\int (a^\varepsilon - a)\nabla u \cdot \nabla \phi \to 0.
$$
Therefore,
$$
\int a^\varepsilon \nabla u^\varepsilon \cdot \nabla \phi
\to
\int a \nabla u \cdot \nabla \phi.
$$
Passing to the limit in all terms yields
$$
\partial_t u - \nabla \cdot (a\nabla u)=f
$$
in the sense of distributions. Then it follows that
$$
u \in \bE^{s}(J_T) \hookrightarrow C([0,T];H^s).
$$
\end{proof}
For the fluid velocity field, we will need the following result.
\begin{proposition}\label{prop:appendix existence-u}
Let $s\in \mathbb N$ with $s\ge 2$.
We consider
\begin{equation}\label{appendix:u}
\partial_t u-\nabla\cdot(a\nabla u)+\nabla\pi  =f,
\qquad
\nabla\cdot u =0,
\qquad
u(0)=u_0,
\end{equation}
where
\begin{equation*}
\label{appendix:a-cond-stokes}
\begin{aligned}
&  a\in L_\infty((0,T)\times \bR^N),\qquad  \nabla a \in \bE^{s-1}(J_T) ,   \\
 & 0<c_0 I_N \le a(t,x)\le c_1 I_N \ \ \text{for all }\  (t,x)\in [0,T]\times \mathbb R^N, \\
& f\in L_2((0,T); H^{s-1}),\quad u_0\in H^s_\sigma.
\end{aligned}
\end{equation*}
Then \eqref{appendix:u} admits a solution satisfying
\[
u\in\bE^s_\sigma(J_T),
\qquad
\nabla\pi\in L_2((0,T);H^{s-1}).
\]
\end{proposition}
\begin{proof}
As in the proof of Proposition~\ref{prop:appendix existence}, we first
approximate $a,f$, and $u_0$ by smooth functions
$a^\varepsilon,f^\varepsilon$, and $u_0^\varepsilon$, respectively,
such that
$$
\frac{c_0}{2}\leq a^\varepsilon\leq 2c_1,
\qquad
\nabla a^\varepsilon\to\nabla a
\quad\text{in }\bE^{s-1}(J_T),
$$
$$
f^\varepsilon\to f
\quad\text{in }L_2((0,T);H^{s-1}),
\qquad
u_0^\varepsilon\to u_0
\quad\text{in }H^s_\sigma.
$$
For each $\varepsilon>0$, the standard Galerkin approximation in the
space of divergence-free vector fields, cf. \cite{TemamBook}, yields a smooth solution
$(u^\varepsilon,\pi^\varepsilon)$ of
\begin{equation}\label{appendix:u-eps}
\partial_tu^\varepsilon
-\nabla\cdot(a^\varepsilon\nabla u^\varepsilon)
+\nabla\pi^\varepsilon
=f^\varepsilon,
\qquad
\nabla\cdot u^\varepsilon =0,
\qquad
u^\varepsilon(0) =u_0^\varepsilon.
\end{equation}
As in the proof of Proposition~\ref{prop:appendix existence}, we can derive
\begin{equation}\label{appendix:u-uniform}
\sup_{t\in[0,T]}
\|u^\varepsilon(t)\|_{H^s}^2
+
\int_0^T
\|u^\varepsilon(t)\|_{H^{s+1}}^2\,dt
\leq
C,
\end{equation}
where $C$ is independent of $\varepsilon$.

We next estimate the time derivative and the pressure.
Recall that the Helmholtz projection
$$
\PH:H^r(\mathbb R^N;\mathbb R^N)
\longrightarrow H^r_\sigma(\mathbb R^N;\mathbb R^N)
$$
is a bounded linear operator for every $r\ge 0$.
Applying $\PH$ to \eqref{appendix:u-eps} gives
$$
\partial_tu^\varepsilon
=
\PH\left[
\nabla\cdot(a^\varepsilon\nabla u^\varepsilon)
+f^\varepsilon
\right].
$$
Since $\PH$ is bounded on $H^{s-1}$, it follows that
\begin{align*}
\|\partial_tu^\varepsilon\|_{H^{s-1}}
&\leq
C\left(
\|\nabla\cdot(a^\varepsilon\nabla u^\varepsilon)\|_{H^{s-1}}
+
\|f^\varepsilon\|_{H^{s-1}}
\right)
\\
&\leq
K\left(
\|a^\varepsilon\|_{L_\infty} +
\|\nabla a^\varepsilon\|_{H^{s-1}}
\right)
\|u^\varepsilon\|_{H^{s+1}}
+
C\|f^\varepsilon\|_{H^{s-1}}.
\end{align*}
Similarly, applying $I-\PH$ to \eqref{appendix:u-eps} yields
$$
\nabla\pi^\varepsilon
=
(I-\PH)
\left[
f^\varepsilon
+
\nabla\cdot(a^\varepsilon\nabla u^\varepsilon)
\right]
$$
which gives
\begin{align*}
\|\nabla\pi^\varepsilon\|_{H^{s-1}}
&\leq
C\left(
\|f^\varepsilon\|_{H^{s-1}}
+
\|\nabla\cdot(a^\varepsilon\nabla u^\varepsilon)\|_{H^{s-1}}
\right)
\\
&\leq
K\left(
\|a^\varepsilon\|_{L_\infty} +
\|\nabla a^\varepsilon\|_{H^{s-1}}
\right)
\|u^\varepsilon\|_{H^{s+1}}
+
C\|f^\varepsilon\|_{H^{s-1}}.
\end{align*}
Consequently,
$$
\{\partial_tu^\varepsilon\}_{\varepsilon}
\quad\text{and}\quad
\{\nabla\pi^\varepsilon\}_{\varepsilon}
\qquad \text{are uniformly bounded in }L_2((0,T);H^{s-1}).
$$
Moreover, $\{u^\varepsilon\}_{\varepsilon}$ is bounded in $\bE^s_\sigma(J_T).$
Thus, after passing to a subsequence,
\begin{align*}
u^\varepsilon\rightharpoonup u
\quad\text{in }\bE^s_\sigma(J_T)
\qquad
\text{and}
\qquad
\nabla\pi^\varepsilon\rightharpoonup\nabla\pi
\quad \text{in }L_2((0,T);H^{s-1}).
\end{align*}
As in the proof of Proposition~\ref{prop:appendix existence},
the Aubin-Lions lemma gives, for every ball $B_R\subset\mathbb R^N$,
$$
u^\varepsilon\to u
\quad\text{strongly in }
L_2((0,T);H^s(B_R)).
$$
This local strong convergence, together with
$a^\varepsilon\to a$ almost everywhere and the uniform
$L_\infty$-bound for $a^\varepsilon$, allows  to pass to the limit
in the diffusion term. We therefore obtain
$$
\partial_tu-\nabla\cdot(a\nabla u)+\nabla\pi=f,
\qquad
\nabla\cdot u=0
$$
in the sense of distributions, with $u(0)=u_0$.
Consequently,
$$
u\in\bE^s_\sigma(J_T),
\qquad
\nabla\pi\in L_2((0,T);H^{s-1}),
$$
as asserted.
\end{proof}

\begin{proof}[Proof of Proposition~\ref{prop: existence linear}]

For notational brevity, we put 
\[
\sg=( g_{u},g_{H},g_{\theta}, g_{m})= (G_u(\sz), G_H (\sz), G_\theta (\sz), G_m(\sz)),
\quad   \sz \in R_T(\sd_0,\sd_1,  M_1,M_2,M_3). 
\]
Under the conditions of Proposition~\ref{prop: existence linear},  the following regularity properties hold
\begin{equation}
\label{nonlinear term regularity}
\begin{split}
\sg \in L_2((0,T); H^{s-1}) \times C(J_T;H^{s-1})\times L_2((0,T); H^{s-1}) \times L_2((0,T); H^{s}).
\end{split}
\end{equation}
The existence of a solution $(u,H)\in \bE^s_\sigma (J_T)\times \bC^s(J_T)$ follows from Proposition~\ref{prop:appendix existence-u} and \cite[Theorem~3.19]{BahouriCheminDanchinBook}. 
We will prove the existence of a solution $\theta$ to \eqref{magneto sys abstract linear}$_4$.
The assertion for the $m$-equation follows in a similar manner. 

\medskip\noindent
Suppose $\theta_0$ satisfies the assumptions of  Proposition~\ref{prop: existence linear}. 
As in Section 4.2,  let   $\theta_1 := e^{t\Delta} \theta_0$.
Then
\begin{equation} 
\label{semigroup}
\theta_1\in L_\infty((0,T)\times \bR^N),\quad 
\partial_t \theta_1 =   e^{t\Delta} \Delta \theta_0\in  L_2((0,T);H^{s-1}), \quad \nabla \theta_1 = e^{t\Delta} \nabla\theta_0\in \bE^{s-1}(J_T),
\end{equation} 
which follows from the $L_p$-maximal regularity property of the Laplace operator, cf. \cite[Chapter 6]{PruSim16}. 
Because $\theta_1$ satisfies the heat equation
$$
    \partial_t\theta_1
    =\Delta\theta_1,
    \qquad
    \theta_1(0)=\theta_0,
$$
it follows that
$
    \theta_1(t)-\theta_0
    =
    \int_0^t\Delta\theta_1(\tau)\,d\tau.
$
Consequently, the Cauchy-Schwarz inequality in time gives
\begin{align*}
    \|\theta_1(t)-\theta_0\|_{L_2}^2
    &\leq
    t\int_0^t
    \|\Delta\theta_1(\tau)\|_{L_2}^2\,d\tau.
\end{align*}
To estimate the RHS, we test the heat equation by
$-\Delta\theta_1$. Integration by parts yields
$$
    \frac{1}{2}\frac{d}{dt}
    \|\nabla\theta_1(t)\|_{L_2}^2
    +
    \|\Delta\theta_1(t)\|_{L_2}^2
    =0.
$$
Integrating this identity over $(0,t)$ gives
\begin{align*}
    \int_0^t
    \|\Delta\theta_1(\tau)\|_{L_2}^2\,d\tau
    =
    \frac{1}{2}\|\nabla\theta_0\|_{L_2}^2
    -
    \frac{1}{2}
    \|\nabla\theta_1(t)\|_{L_2}^2   
    \leq
    \frac{1}{2}\|\nabla\theta_0\|_{L_2}^2.
\end{align*}
Therefore,
\begin{equation}
\label{est_initial}
    \|\theta_1(t)-\theta_0\|_{L_2}^2
    \leq
    \frac{t}{2}\|\nabla\theta_0\|_{L_2}^2.
\end{equation}
Let $\phi=\theta - \theta_1$. Then $\phi$ solves the equation
\begin{equation}
\label{magneto sys theta}
\partial_t \phi -  \nabla \cdot (\kappa (\vartheta) \nabla \phi)    
 =  g_\theta  +  \nabla \cdot (\kappa (\vartheta) \nabla \theta_1) -\partial_t \theta_1, 
 \quad  \phi (0)   = 0 .
\end{equation} 
For  any $\sz \in R_T(\sd_0,\sd_1,   M_1,M_2,M_3)$, the conditions in \eqref{appendix:a cond} are satisfied for $\kappa(\vartheta)$ in virtue of Lemma~\ref{Lem: Nemyskii}.
Hence, Lemma~\ref{Appendix_lem} and \eqref{semigroup} imply 
$$ \nabla \cdot (\kappa (\vartheta) \nabla \theta_1) -\partial_t \theta_1 \in L_2((0,T);H^{s-1}).$$
Proposition~\ref{prop:appendix existence} then shows that \eqref{magneto sys theta} has a solution $\phi\in \bE^s(J_T)$. 
Together with \eqref{semigroup}, this implies 
\begin{equation*}
\begin{aligned}
\nabla \theta &= \nabla \theta_1 + \nabla \phi \in \bE^{s-1}(J_T) \hookrightarrow C([0,T]; H^{s-1}), \\
\partial_t \theta &= \partial_t \theta_1 + \partial_t \phi \in L_2((0,T); H^{s-1}).
\end{aligned}
\end{equation*}
Moreover, $\theta$ is the (unique) solution of \eqref{magneto sys abstract linear}$_4$ with initial value $\theta_0$. 
Next we observe that  \eqref{est_initial}  yields
\begin{align*}
\| \theta(t) - \theta_0\|_{L_2}^2  \leq 2\| \theta_1(t) - \theta_0\|_{L_2}^2 + 2\|\phi(t)\|_{L_2}^2   
 \leq T\|\nabla \theta_0\|_{L_2}^2 + 2\sup_{t\in [0,T]} \|\phi(t)\|_{L_2}^2.
\end{align*}
Similar to \eqref{est h L infty}, we can then derive the estimate
\begin{align*}
\| \theta(t)\|_\infty & \le \| \theta(t) - \theta_0\|_\infty + \|\theta_0\|_\infty  \\
&\le C \| \theta(t) - \theta_0\|_{L_2} + C \| \nabla  \theta (t)\|_{H^1}  + \|\nabla \theta_0\|_{H^1} +   \|\theta_0\|_\infty  .
\end{align*}
In summary, we have shown that $\theta$ satisfies
\begin{equation*}
\label{theta-reg}
\theta \in L_\infty ((0,T)\times \bR^N), \quad \nabla\theta \in \bE^{s-1}(J_T), \quad \partial_t \theta \in  L_2((0,T); H^{s-1}).
\begin{aligned}
\end{aligned}
\end{equation*}
Consequently,  $\theta $ enjoys the properties in Proposition~\ref{prop: existence linear}. 
\end{proof}


\section*{Acknowledgements}
This work was supported by a grant from the Simons Foundation [MPS-TSM-00853237, GS].
The first author gratefully acknowledges support from NSF grants DMS-2306991 and DMS-2512104.
 

 \end{document}